\documentclass[11pt]{article}

\usepackage[latin1]{inputenc} 
\usepackage{amsmath}
\usepackage{amsthm}
\usepackage{amsfonts}
\usepackage{amssymb}
\usepackage{mathtools}
\usepackage{graphics}
\usepackage{graphicx}
\usepackage[top=1in, bottom=1in, left=1in, right=1in,footskip=.5in]{geometry}
\usepackage{amssymb,amsmath,amsthm,amscd,epsf,latexsym,verbatim,graphicx,amsfonts} 
\input epsf.tex
\usepackage[english]{babel}
\usepackage{enumerate}
\usepackage{enumitem}
\usepackage{mathrsfs}
\usepackage{tikz-cd}
\usepackage{hyperref}
\usepackage{tabularx}
\usepackage[format=plain,
            font=it]{caption}
\usepackage{makecell}
\usepackage{authblk}
\usepackage{titling}
\usepackage[linesnumbered,ruled,vlined]{algorithm2e}

\newcounter{Theorem}
\numberwithin{Theorem}{section}

\theoremstyle{definition}

\newtheorem{theorem}[Theorem]{Theorem}
\newtheorem{example}[Theorem]{Example}
\newtheorem{definition}[Theorem]{Definition}
\newtheorem{corollary}[Theorem]{Corollary}

\newtheorem{lemma}[Theorem]{Lemma}
\newtheorem{remark}[Theorem]{Remark}
\newtheorem*{problem*}{Problem}

\makeatletter
    \let\c@algocf\c@Theorem
\makeatother

\usepackage{comment}
\usepackage[backend=biber,style=numeric,maxbibnames=99]{biblatex}
\renewbibmacro{in:}{}

\usepackage{lineno}

\newcommand{\bbN}{\mathbb{N}}

\def\Vhrulefill{\leavevmode\leaders\hrule height 0.7ex depth \dimexpr0.4pt-0.7ex\hfill\kern0pt}
\newcommand{\algname}[1]{\texttt{#1}}

\newcommand{\xra}[1]{\xrightarrow{#1}}

\newcommand{\op}[1]{\operatorname{#1}}
\newcommand{\precdot}{\prec\hspace{-0.075in}\cdot\hspace{0.0375in}}
\newcommand{\succdot}{\succ \hspace{-0.165in} \cdot \hspace{0.125in}}

\DeclareMathOperator{\leaf}{\operatorname{Lf}}
\DeclareMathOperator{\cross}{\operatorname{Crs}}
\DeclareMathOperator{\rut}{\operatorname{root}}
\DeclareMathOperator{\crt}{\operatorname{crt}}
\DeclareMathOperator{\crtei}{\operatorname{crtei}}

\title{Characterizing planar tanglegram layouts and \\applications to edge insertion problems}
\author{Kevin Liu}
\date{June 2, 2022}

\begin{document}
\maketitle

\begin{abstract}
    Tanglegrams are formed by taking two rooted binary trees $T$ and $S$ with the same number of leaves and uniquely matching each leaf in $T$ with a leaf in $S$. They are usually represented using layouts that embed the trees and matching in the plane.
    Given the numerous ways to construct a layout, one problem of interest is the Tanglegram Layout Problem, which is to efficiently find a layout that minimizes the number of crossings. This parallels a similar problem involving drawings of graphs, where a common approach is to insert edges into a planar subgraph. In this paper, we explore inserting edges into a planar tanglegram. Previous results on planar tanglegrams include a Kuratowski Theorem, enumeration, and an algorithm for finding a planar layout. We build on these results and characterize all planar layouts of a planar tanglegram. We then apply this characterization to construct a quadratic-time algorithm that inserts a single edge optimally. Finally, we generalize some results to multiple edge insertion.
\end{abstract}

\bigskip

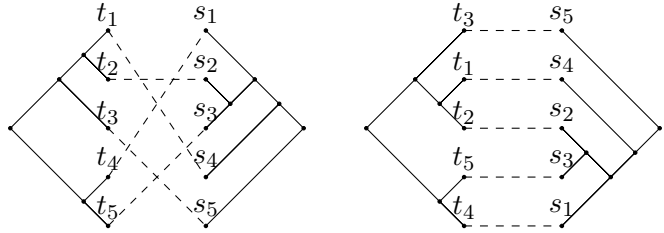
\begin{figure}[h]
\begin{center}
\begin{tikzpicture}[scale=0.65]
\filldraw[fill=black,draw=black] (-1,2) circle (1pt);
\filldraw[fill=black,draw=black] (-1,1) circle (1pt);
\filldraw[fill=black,draw=black] (-1,0) circle (1pt);
\filldraw[fill=black,draw=black] (-1,-1) circle (1pt);
\filldraw[fill=black,draw=black] (-1,-2) circle (1pt);
\filldraw[fill=black,draw=black] (1,2) circle (1pt);
\filldraw[fill=black,draw=black] (1,1) circle (1pt);
\filldraw[fill=black,draw=black] (1,0) circle (1pt);
\filldraw[fill=black,draw=black] (1,-1) circle (1pt);
\filldraw[fill=black,draw=black] (1,-2) circle (1pt);
\draw[dashed] (-1,2) -- (1,-1);
\draw[dashed] (-1,1) -- (1,1);
\draw[dashed] (-1,0) -- (1,-2);
\draw[dashed] (-1,-1) -- (1,2);
\draw[dashed] (-1,-2) -- (1,0);
\draw (-1,2) -- (-1.5,1.5) -- (-1,1) -- (-1.5,1.5) -- (-2,1) -- (-1,0) -- (-2,1) -- (-3,0) -- (-1,-2) -- (-1.5,-1.5) -- (-1,-1);
\draw (1,2) -- (2,1) --(1.5,0.5) -- (1,1) -- (1.5,0.5) -- (1,0) -- (1.5,0.5) -- (2,1) -- (2.5,.5) -- (1,-1) -- (2.5,.5)-- (3,0) -- (1,-2);
\filldraw[fill=black,draw=black] (-1.5,1.5) circle (1pt);
\filldraw[fill=black,draw=black] (-2,1) circle (1pt);
\filldraw[fill=black,draw=black] (-3,0) circle (1pt);
\filldraw[fill=black,draw=black] (-1.5,-1.5) circle (1pt);
\filldraw[fill=black,draw=black] (1.5,0.5) circle (1pt);
\filldraw[fill=black,draw=black] (2,1) circle (1pt);
\filldraw[fill=black,draw=black] (3,0) circle (1pt);
\filldraw[fill=black,draw=black] (2.5,0.5) circle (1pt);
\node at (-1,2.325) {$t_1$};
\node at (-1,1.325) {$t_2$};
\node at (-1,0.325) {$t_3$};
\node at (-1,-0.625) {$t_4$};
\node at (-1,-1.625) {$t_5$};
\node at (1,2.325) {$s_1$};
\node at (1,1.325) {$s_2$};
\node at (1,0.325) {$s_3$};
\node at (1,-0.625) {$s_4$};
\node at (1,-1.625) {$s_5$};
\end{tikzpicture}\qquad
\begin{tikzpicture}[scale=0.65]
\filldraw[fill=black,draw=black] (-1,2) circle (1pt);
\filldraw[fill=black,draw=black] (-1,1) circle (1pt);
\filldraw[fill=black,draw=black] (-1,0) circle (1pt);
\filldraw[fill=black,draw=black] (-1,-1) circle (1pt);
\filldraw[fill=black,draw=black] (-1,-2) circle (1pt);
\filldraw[fill=black,draw=black] (1,2) circle (1pt);
\filldraw[fill=black,draw=black] (1,1) circle (1pt);
\filldraw[fill=black,draw=black] (1,0) circle (1pt);
\filldraw[fill=black,draw=black] (1,-1) circle (1pt);
\filldraw[fill=black,draw=black] (1,-2) circle (1pt);
\draw[dashed] (-1,2) -- (1,2);
\draw[dashed] (-1,1) -- (1,1);
\draw[dashed] (-1,0) -- (1,0);
\draw[dashed] (-1,-1) -- (1,-1);
\draw[dashed] (-1,-2) -- (1,-2);
\draw (-1,0) -- (-1.5,0.5) -- (-1,1) -- (-1.5,0.5) -- (-2,1) -- (-1,2) -- (-2,1) -- (-3,0) -- (-1,-2) -- (-1.5,-1.5) -- (-1,-1);
\draw (1,1) -- (2.5,-0.5) -- (2,-1) -- (1,0) -- (1.5,-0.5) -- (1,-1) -- (1.5,-0.5) -- (2,-1) -- (1,-2) -- (3,0) -- (1,2);
\filldraw[fill=black,draw=black] (-1.5,0.5) circle (1pt);
\filldraw[fill=black,draw=black] (-2,1) circle (1pt);
\filldraw[fill=black,draw=black] (-3,0) circle (1pt);
\filldraw[fill=black,draw=black] (-1.5,-1.5) circle (1pt);
\filldraw[fill=black,draw=black] (2.5,-0.5) circle (1pt);
\filldraw[fill=black,draw=black] (2,-1) circle (1pt);
\filldraw[fill=black,draw=black] (3,0) circle (1pt);
\filldraw[fill=black,draw=black] (1.5,-0.5) circle (1pt);
\node at (-1,2.325) {$t_3$};
\node at (-1,1.325) {$t_1$};
\node at (-1,0.325) {$t_2$};
\node at (-1,-0.625) {$t_5$};
\node at (-1,-1.625) {$t_4$};
\node at (1,2.325) {$s_5$};
\node at (1,1.325) {$s_4$};
\node at (1,0.325) {$s_2$};
\node at (1,-0.625) {$s_3$};
\node at (1,-1.625) {$s_1$};
\end{tikzpicture}
    \caption{Two layouts for the same tanglegram, one with six crossings and one with  no crossings.}\label{tanglegramexample}
    \end{center}
\end{figure}

\section{Introduction}

Let $T$ and $S$ be two rooted binary trees with leaves respectively labeled as $\{t_i\}_{i\in I}$ and $\{s_j\}_{j\in J}$, where $I,J\subseteq \mathbb{N}$ are finite index sets of the same size. If we let $\phi:I\to J$ be a bijection, then we can denote a tanglegram as $(T,S,\phi)$, where $\phi$ indicates that $t_i$ is matched with $s_{\phi(i)}$. A \textit{layout} of a tanglegram draws $T$, $S$, and the edges $(t_i,s_{\phi(i)})$ in the plane such that $T$ is planarly embedded left of the line $x=0$ with all leaves on $x=0$, $S$ is planarly embedded right of the line $x=1$ with all leaves on $x=1$, and the edges $(t_i,s_{\phi(i)})$ are drawn using straight lines. See Figure \ref{tanglegramexample} for examples. 
A \textit{crossing} is any pair of edges $(t_i,s_{\phi(i)})$ and $(t_j,s_{\phi(j)})$ that intersect in the layout, and the \textit{crossing number} of a tanglegram $(T,S,\phi)$, denoted $\crt(T,S,\phi)$, is the minimum number of crossings over all layouts of $(T,S,\phi)$. The Tanglegram Layout Problem attempts to efficiently find a layout that achieves the crossing number.

Tanglegrams initially arose in biology and computer science. Biologists use binary trees to model evolution and tanglegrams to model relationships between species. Finding optimal layouts helps determine how two species may have co-evolved \cite{reduction}.  Applications in computer science include clustering, decomposition of programs into layers, or analyzing the difference in hierarchy between similar programs or different versions of the same program \cite{algorithm}. Combinatorial interest in tanglegrams developed more recently. Matsen et al. formalized tanglegrams as mathematical objects and described connections with phylogenetics \cite{reduction}. Billey, Konvalinka, and Matsen then enumerated tanglegrams and constructed an algorithm to generate them uniformly at random \cite{count}. Subsequently, Konvalinka and Wagner studied the properties of random tanglegrams \cite{randomtanglegram}, Ralaivaosaona, Ravelomanana, and Wagner counted planar tanglegrams \cite{countingplanar}, and Gessel counted several variations of tanglegrams using combinatorial species \cite{species}.

The crossing number of a tanglegram has connections with the crossing number of a graph $G$, denoted $\operatorname{cr}(G)$, which is the minimum number of crossings over all drawings of $G$. Determining if $\operatorname{cr}(G)\leq k$ for $k\in \mathbb{N}$ is NP-complete \cite{CNNP}, and the same is true for determining if $\crt(T,S,\phi)\leq k$ \cite{tlhard}. Some of the known results in graph drawing have analogous results in tanglegram layouts, and some have approached the Tanglegram Layout Problem by translating what we know about graphs to tanglegrams.  Czabarka, Sz\'ekely, and Wagner recently used the well-known Kuratowski Theorem characterizing planarity of graphs to construct a Tanglegram Kuratowski Theorem characterizing \textit{planar tanglegrams}, which are tanglegrams with crossing number zero \cite{kura}. Prior to this, Lozano et al. constructed their \algname{Untangle} Algorithm for drawing a planar layout of a planar tanglegram \cite{layout}. Anderson et al. recently proved that removing a between-tree edge $(t_i,s_{\phi(i)})$ from a tanglegram reduces crossing number by at most $n-3$, and they produced a family of tanglegrams to show that this bound is sharp \cite{analogy}. They also found that the maximum crossing number over all tanglegrams asymptotically approaches $\frac{1}{2}{n\choose 2}$, where $n$ is the number of leaves in each tree. 

Given the difficulty of minimizing crossings in graph drawings, some have studied approximating the minimum number of crossings rather than finding it exactly. One approach to this is edge insertion. The Edge Insertion Problem for graphs starts with a graph $G$ and an edge $e \in G$ such that $G\setminus\{e\}$ is planar, and attempts to find an embedding of $G$ into the plane so that the drawing of $G\setminus \{e\}$ is planar and the number of crossings from $e$ is minimized. This problem  is well studied. A linear-time algorithm exists to solve it, and some bounds have been found relating an optimal drawing of $G$ and a solution to the Edge Insertion Problem for $G\setminus \{e\}$ with $\{e\}$ inserted \cite{insert1,almostplanar}.
The Edge Insertion Problem generalizes to the Multiple Edge Insertion Problem, where we insert several edges $\{e_1,\ldots ,e_n\}$ into planar $G\setminus \{e_1,\ldots,e_n\}$ optimally, and current approximation algorithms for graph drawing still use multiple edge insertion with planar subgraphs \cite{graphcrossing}. Given the role that edge insertion with planar subgraphs plays in graph drawings, it is plausible that edge insertion can play a similar role for tanglegram layouts. In this paper, we consider the tanglegram versions of the insertion problems for graphs, which we now state.

\begin{problem*}[\textbf{Tanglegram Single Edge Insertion}]
Given a tanglegram $(T,S,\phi)$ and a planar subtanglegram $(T_I,S_{\phi(I)},\phi|_I)$ induced by $I=[n]\setminus \{i\}$ for $i\in [n]$, find a layout of $(T,S,\phi)$ that restricts to a planar layout of $(T_I,S_{\phi(I)},\phi|_I)$ and has the minimal number of crossings possible.
\end{problem*}

\begin{problem*}[\textbf{Tanglegram Multiple Edge Insertion}]
Given a tanglegram $(T,S,\phi)$ and a planar subtanglegram $(T_I,S_{\phi(I)},\phi|_I)$ induced by $I\subseteq [n]$, find a layout of $(T,S,\phi)$ that restricts to a planar layout of $(T_I,S_{\phi(I)},\phi|_I)$ and has the minimal number of crossings possible.
\end{problem*}

We will start by characterizing the planar layouts of a planar tanglegram. For a planar tanglegram $(T,S,\phi)$, we will define a \textit{leaf-matched pair} $(u,v)$ as a pair of internal vertices $u\in T$ and $v\in S$ whose descendant leaves are matched by $\phi$, and we will define an operation called a \textit{paired flip}. These pairs of internal vertices have the property that their descendant leaves are matched by $\phi$. By adding steps to the \algname{Untangle} algorithm by Lozano et al. for drawing a planar layout of a planar tanglegram, we construct \algname{ModifiedUntangle} (Algorithm \ref{modifieduntangle}), which also identifies leaf-matched pairs and stores them in a set $L$, obtaining the following result.

\begin{theorem}\label{planarlayouts} 
Let $(T,S,\phi)$ be a planar tanglegram, and let $\mathscr{P}(T,S,\phi)$ denote its collection of planar layouts. Let the output of $\algname{ModifiedUntangle}(T,S,\phi)$ be the layout $(X,Y)$ and set of leaf-matched pairs $L$. Every $(X',Y')\in \mathscr{P}(T,S,\phi)$ can be obtained by starting with $(X,Y)$ and performing a sequence of paired flips at $(u,v)\in L$.
\end{theorem}

Letting $\text{size}(T,S,\phi)$ be the number of leaves in $T$ or $S$, we then consider the generating function
\begin{equation}\label{Fxq}
    F(x,q)=\sum_{\text{planar  }(T,S,\phi)}x^{\text{size}(T,S,\phi)}q^{|\{\text{leaf-matched pairs of } (T,S,\phi)\}|}.
\end{equation}
The coefficient of $x^nq^k$ is the number of tanglegrams of size $n$ with $k$ leaf-matched pairs. When $\text{size}(T,S,\phi)\geq 2$, the roots of $T$ and $S$ always form a leaf matched pair. A tanglegram is \textit{irreducible} if this is the only leaf-matched pair. We let 
\begin{equation}\label{Hx}
    H(x) = \sum_{\text{irreducible planar $(T,S,\phi)$}} x^{\text{size}(T,S,\phi)},
\end{equation}
where we use the convention that the coefficient of $x^2$ is $\frac{1}{2}$, as in \cite[Proposition 8]{countingplanar}. These generating functions have a relationship, which we can use to find coefficients of $x^nq^k$ in $F(x,q)$.

\begin{theorem}\label{Fxqrelation}
The generating function $F(x,q)$ satisfies the relation
\begin{equation}\label{Fxqequation}
    F(x,q)=x+q\cdot H(F(x,q))+\frac{q\cdot F(x^2,q^2)}{2}.
\end{equation}
\end{theorem}

Afterwards, we use our characterization of planar layouts to solve the Tanglegram Single Edge Insertion Problem. After considering various cases, we construct \algname{Insertion} (Algorithm \ref{insertionalgorithm}) and show the following result.

\begin{theorem}\label{insertiontime}
The \algname{Insertion} Algorithm solves the Tanglegram  Single Edge Insertion Problem in $O(n^2)$ time and space, where $n$ is the size of the tanglegram.
\end{theorem}

Finally, we consider Tanglegram Multiple Edge Insertion. Similar to the corresponding graph theory problems, Single Edge Insertion can be solved efficiently, but Multiple 
Edge Insertion is significantly more difficult.

\begin{theorem}\label{NP-hard}
The Tanglegram Multiple Edge Insertion Problem  is NP-hard.
\end{theorem}

Nevertheless, we generalize some of our results from \algname{Insertion} to construct our \algname{MultiInsertion} algorithm. While the space required will still be $O(n^2)$, the runtime is potentially exponential, depending on where vertices and edges are inserted, as well as how many leaf-matched pairs there are. Nevertheless, in certain situations, \algname{MultiInsertion} efficiently solves the Tanglegram Multiple Edge Insertion Problem.

We start in Section \ref{background} by outlining terminology, notation, and previous results. In Section \ref{planar}, we give our \algname{ModifiedUntangle} Algorithm, establish Theorem \ref{planarlayouts}, and then show the relation in Theorem \ref{Fxqrelation}. In Section \ref{TEI}, we give our \algname{Insertion} Algorithm and prove Theorem 1.3. In Section \ref{MEI}, we prove Theorem 1.4 and generalize some of our results from Section \ref{TEI} to the Tanglegram Multiple Edge Insertion Problem. We conclude by discussing future work in Section \ref{future}.

\setcounter{section}{1}

\section{Preliminaries}\label{background}

In this section, we begin by defining rooted binary trees and outlining the terminology and notation that we will use. We then introduce our notation for tanglegram layouts as ordered lists of leaves in the two trees. Afterwards, we define subtanglegrams, which correspond to subgraphs in graph theory. Note that removing vertices or edges in a tanglegram does not produce another tanglegram, so we give a description of the steps needed to construct subtanglegrams. An example of our notation and terminology will be given in Example \ref{layoutexample}, and the reader is encouraged to refer to this example as they read this section. We conclude  with some known results.

A \textit{rooted binary tree} $T$ is a tree in which every vertex has either zero or two children, and where a designated vertex called the root, denoted $\rut(T)$, is allowed to have degree two. A vertex that has children is called an \textit{internal vertex}, and a vertex with no children is called a \textit{leaf}. If $v$ has children $v_1$ and $v_2$, we call $v$ the parent of $v_1$ and $v_2$. 
We say vertex $v_1$ is a \textit{descendant} of $v_k$ or $v_k$ is an \textit{ancestor} of $v_1$ if there is a sequence of vertices $v_1,v_2,\ldots ,v_k$ such that $v_{i+1}$ is the parent of $v_i$ for  $i=1,2,\ldots, k-1$, and we use the notation $v_1<v_k$ or $v_k>v_1$ to denote this. When needed, we use a subscript with the name of a tree to specify ancestry in that tree, such as $v_k>_T v_1$. 

For an internal vertex $v\in T$, the \textit{subtree rooted at $v$} is the tree formed by all vertices $u$ with $u\leq v$, and this subtree then has $v$ as its root. Using subtrees, we can represent trees using the nested lists notation from Section 2.3.2 of \cite{knuth}, where each set of parenthesis represents a subtree. Unless otherwise stated, we will index leaves with $[n]=\{1,2,\ldots ,n\}$, and usually we will omit labels for internal vertices. 

All trees are considered up to isomorphism, so in particular, relabeling vertices does not produce a different tree. Given an internal vertex $v\in T$, a \textit{flip} at vertex $v$ is the operation that interchanges the order of the children for all $u\leq v$. Pictorially, if we start with a drawing of a tree $T$, a flip at $v\in T$ reflects the subtree rooted at $v$, which motivates the name ``flip." Notice that each flip has order two, all flips commute with one another, and for any rooted binary tree $T$, flips generate all trees isomorphic to $T$.

Tanglegrams $(T,S,\phi)$ are formed from a pair of rooted binary trees $T,S$ and a bijection $\phi$ matching their leaves. The \textit{size} of $(T,S,\phi)$ is the common number of leaves in $T$ or $S$. We will call the edges in $T$ and $S$ \textit{tree edges} and call the edges induced by $\phi$ \textit{between-tree edges}. For any vertex $u\in T$, we use $\leaf(u)$ to denote the leaves $\ell\in T$ such that $\ell \leq_T u$, and similarly for $\leaf(v)$ when $v\in S$. As with trees, we consider tanglegrams up to isomorphism. Relabeling vertices or replacing $T$ and $S$ with isomorphic trees does not produce a new tanglegram, provided that $\phi$ is modified appropriately. See \cite{reduction} for more details. Notice that for any tanglegram $(T,S,\phi)$, flips will generate all tanglegrams isomorphic to $(T,S,\phi)$, as they generate all isomorphisms of the underlying trees.

Our notation for tanglegram layouts builds on the notation used in \cite{layout}. In any layout, the number of crossings is completely determined by the order of the leaves in the two trees and the bijection $\phi$ matching these leaves, as the between-tree edges $(t_i,s_{\phi(i)})$ and $(t_j,s_{\phi(j)})$ intersect when $t_i$ is embedded above $t_j$ and $s_{\phi(i)}$ is embedded below $s_{\phi(j)}$. Since we are primarily interested in the number of crossings rather than specific coordinates of the plane embedding, we give the following definition.

\begin{definition}
Let $(T,S,\phi)$ be a tanglegram drawn in the plane with a given layout. The \textit{leaf order} of the given layout is a pair of ordered lists $(X,Y)$, where $X$ and $Y$ respectively list the leaves of $T$ and $S$ in order of appearance from top to bottom in the layout.
\end{definition}

One can view the leaf order of a layout $(X,Y)$ as an equivalence class of layouts, where two layouts are equivalent if they draw the leaves of $T$ and $S$ in the same order from top to bottom.  To recover a layout from the ordered lists $(X,Y)$, one can draw the leaves listed in $X$ and $Y$ from top to bottom respectively on $x=0$ and $x=1$, and then use the information from $T$, $S$, and $\phi$ to draw the trees and between-tree edges. Flips generate all trees isomorphic to $T$ or $S$, so they can act on leaf orders $(X,Y)$ to obtain all possible leaf orders, where a flip at an internal vertex $u$ acts on $(X,Y)$ by reversing the order of the elements in $\leaf(u)$ in the appropriate list $X$ or $Y$. Throughout this paper, we abuse terminology and refer to this pair of lists $(X,Y)$ also as a tanglegram layout.

We will often decompose $X$ and $Y$ into concatenated lists $(X_1X_2\ldots X_m,Y_1Y_2\ldots Y_n)$, where each $X_i$ or $Y_j$ is some ordered collection of consecutive leaves. We do not impose any restrictions on $X_i$ and $Y_j$ beyond the fact that they must contain consecutive leaves, but usually we will use $X_i$ or $Y_j$ to represent the leaves in $\leaf(u)$ or $\leaf(v)$ for some $u\in T$ or $v\in S$. If we start with the layout $(X_1\ldots X_i\ldots X_m,Y)$ and the sublist $X_i$ contains the elements in $\leaf(u)$, we will denote the layout after a flip at $u$ as $(X_1\ldots \overline{X_i}\ldots X_m,Y)$, where the bar indicates that the sublist $X_i$ is reversed. Note that while this notation using lists is convenient, the reader should visualize $(X,Y)$ as embeddings into the plane, and when possible, we will include drawings to help illustrate the arguments that we make using these ordered lists.

Finally, we will define induced subtrees and induced subtanglegrams using a similar definition as in \cite{kura}. Notice that layouts $(X',Y')$ of subtanglegrams correspond to taking sub-lists in layouts $(X,Y)$ of the original tanglegram.

\begin{definition}\label{subtreedefinition}
Let $T$ be a tree with leaves indexed by $[n]$. For any $I\subseteq [n]$, the \textit{rooted binary subtree induced by $I$}, denoted $T_I$, is formed by taking the minimal subtree of $T$ containing the leaves indexed by $I$ and suppressing all internal vertices that have only one child.
\end{definition}
\begin{definition}\label{subtanglegramdefinition}
Let $(T,S,\phi)$ be a tanglegram with the leaves of $T$ and $S$ indexed by $[n]$. For any $I\subseteq [n]$, the \textit{subtanglegram induced by $I$} is the tanglegram $(T_I,S_{\phi(I)},\phi|_I)$, that is, the tanglegram formed from the induced subtrees $T_I$ and $S_{\phi(I)}$ with leaves matched using the restriction $\phi|_I$. 
\end{definition}
\begin{example}\label{layoutexample}
Consider the tanglegram $(T,S,\phi)$ with

\begin{center}
    $T=(((t_1,t_2),t_3),(t_4,t_5))$ \quad $S=(((s_1,(s_2,s_3)),s_4),s_5)$ \quad 
    \begin{tabular}{c|c|c|c|c|c}
    $i$ & $1$ & $2$ & $3$ & $4$ & $5$ \\
    \hline
    $\phi(i)$ & $4$ & $2$ & $5$ & $1$ & $3$ 
\end{tabular}
\end{center}

\noindent
By writing the leaves of $T$ and $S$ in the order given, we obtain $(X,Y)=(t_1t_2t_3t_4t_5,s_1s_2s_3s_4s_5)$. A flip at the root of $((s_1,(s_2,s_3)),s_4)$ results in $(X',Y')=(t_1t_2t_3t_4t_5,s_4s_3s_2s_1s_5)$. By taking sub-lists of $(X',Y')$ based on the elements in $\{1,2,4,5\}$ and $\phi(\{1,2,4,5\})$, we obtain the layout $(t_1t_2t_4t_5,s_4s_3s_2s_1)$ for $(T_{\{1,2,4,5\}},S_{\phi(\{1,2,4,5\})},\phi|_{\{1,2,4,5\}})$. Using $T$ and $S$, we produce the drawings in Figure \ref{notation} corresponding to these ordered lists.
\end{example}

\begin{figure}[h]
\begin{center}
\begin{tikzpicture}[scale=0.6]
\filldraw[fill=black,draw=black] (-1,2) circle (1pt);
\filldraw[fill=black,draw=black] (-1,1) circle (1pt);
\filldraw[fill=black,draw=black] (-1,0) circle (1pt);
\filldraw[fill=black,draw=black] (-1,-1) circle (1pt);
\filldraw[fill=black,draw=black] (-1,-2) circle (1pt);
\filldraw[fill=black,draw=black] (1,2) circle (1pt);
\filldraw[fill=black,draw=black] (1,1) circle (1pt);
\filldraw[fill=black,draw=black] (1,0) circle (1pt);
\filldraw[fill=black,draw=black] (1,-1) circle (1pt);
\filldraw[fill=black,draw=black] (1,-2) circle (1pt);
\draw[dashed] (-1,2) -- (1,-1);
\draw[dashed] (-1,1) -- (1,1);
\draw[dashed] (-1,0) -- (1,-2);
\draw[dashed] (-1,-1) -- (1,2);
\draw[dashed] (-1,-2) -- (1,0);
\draw (-1,2) -- (-1.5,1.5) -- (-1,1) -- (-1.5,1.5) -- (-2,1) -- (-1,0) -- (-2,1) -- (-3,0) -- (-1,-2) -- (-1.5,-1.5) -- (-1,-1);
\draw (1,2) -- (2,1) --(1.5,0.5) -- (1,1) -- (1.5,0.5) -- (1,0) -- (1.5,0.5) -- (2,1) -- (2.5,.5) -- (1,-1) -- (2.5,.5)-- (3,0) -- (1,-2);
\filldraw[fill=black,draw=black] (-1.5,1.5) circle (1pt);
\filldraw[fill=black,draw=black] (-2,1) circle (1pt);
\filldraw[fill=black,draw=black] (-3,0) circle (1pt);
\filldraw[fill=black,draw=black] (-1.5,-1.5) circle (1pt);
\filldraw[fill=black,draw=black] (1.5,0.5) circle (1pt);
\filldraw[fill=black,draw=black] (2,1) circle (1pt);
\filldraw[fill=black,draw=black] (3,0) circle (1pt);
\filldraw[fill=black,draw=black] (2.5,0.5) circle (1pt);
\node at (-1,2.325) {\small{$t_1$}};
\node at (-1,1.325) {\small{$t_2$}};
\node at (-1,0.325) {\small{$t_3$}};
\node at (-1,-0.625) {\small{$t_4$}};
\node at (-1,-1.625) {\small{$t_5$}};
\node at (1,2.325) {\small{$s_1$}};
\node at (1,1.325) {\small{$s_2$}};
\node at (1,0.325) {\small{$s_3$}};
\node at (1,-0.625) {\small{$s_4$}};
\node at (1,-1.625) {\small{$s_5$}};


\filldraw[fill=black,draw=black] (6,2) circle (1pt);
\filldraw[fill=black,draw=black] (6,1) circle (1pt);
\filldraw[fill=black,draw=black] (6,0) circle (1pt);
\filldraw[fill=black,draw=black] (6,-1) circle (1pt);
\filldraw[fill=black,draw=black] (6,-2) circle (1pt);
\filldraw[fill=black,draw=black] (8,2) circle (1pt);
\filldraw[fill=black,draw=black] (8,1) circle (1pt);
\filldraw[fill=black,draw=black] (8,0) circle (1pt);
\filldraw[fill=black,draw=black] (8,-1) circle (1pt);
\filldraw[fill=black,draw=black] (8,-2) circle (1pt);
\draw[dashed] (6,2) -- (8,2);
\draw[dashed] (6,1) -- (8,0);
\draw[dashed] (6,0) -- (8,-2);
\draw[dashed] (6,-1) -- (8,-1);
\draw[dashed] (6,-2) -- (8,1);
\draw (6,2) -- (5.5,1.5) -- (6,1) -- (5.5,1.5) -- (5,1) -- (6,0) -- (5,1) -- (4,0) -- (6,-2) -- (5.5,-1.5) -- (6,-1);
\draw (8,2) -- (9.5,0.5) -- (9,0) -- (8,1) -- (8.5,0.5) -- (8,0) -- (8.5,0.5) -- (9,0) -- (8,-1) -- (9.5,0.5) -- (10,0) -- (8,-2);
\filldraw[fill=black,draw=black] (5.5,1.5) circle (1pt);
\filldraw[fill=black,draw=black] (5,1) circle (1pt);
\filldraw[fill=black,draw=black] (4,0) circle (1pt);
\filldraw[fill=black,draw=black] (5.5,-1.5) circle (1pt);
\filldraw[fill=black,draw=black] (8.5,0.5) circle (1pt);
\filldraw[fill=black,draw=black] (9,0) circle (1pt);
\filldraw[fill=black,draw=black] (10,0) circle (1pt);
\filldraw[fill=black,draw=black] (9.5,0.5) circle (1pt);
\node at (6,2.325) {\small{$t_1$}};
\node at (6,1.325) {\small{$t_2$}};
\node at (6,0.325) {\small{$t_3$}};
\node at (6,-0.625) {\small{$t_4$}};
\node at (6,-1.625) {\small{$t_5$}};
\node at (8,2.325) {\small{$s_4$}};
\node at (8,1.325) {\small{$s_3$}};
\node at (8,0.325) {\small{$s_2$}};
\node at (8,-0.625) {\small{$s_1$}};
\node at (8,-1.625) {\small{$s_5$}};


\filldraw[fill=black,draw=black] (12.5,1.5) circle (1pt);
\filldraw[fill=black,draw=black] (12.5,0.5) circle (1pt);
\filldraw[fill=black,draw=black] (12.5,-0.5) circle (1pt);
\filldraw[fill=black,draw=black] (12.5,-1.5) circle (1pt);
\filldraw[fill=black,draw=black] (14.5,1.5) circle (1pt);
\filldraw[fill=black,draw=black] (14.5,0.5) circle (1pt);
\filldraw[fill=black,draw=black] (14.5,-0.5) circle (1pt);
\filldraw[fill=black,draw=black] (14.5,-1.5) circle (1pt);
\node at (12.5,1.875) {\small{$t_1$}};
\node at (12.5,0.875) {\small{$t_2$}};
\node at (12.5,-0.125) {\small{$t_4$}};
\node at (12.5,-1.125) {\small{$t_5$}};
\node at (14.5,1.875) {\small{$s_4$}};
\node at (14.5,0.875) {\small{$s_3$}};
\node at (14.5,-0.125) {\small{$s_2$}};
\node at (14.5,-1.125) {\small{$s_1$}};
\draw[dashed] (12.5,1.5) -- (14.5,1.5);
\draw[dashed] (12.5,0.5) -- (14.5,-0.5);
\draw[dashed] (12.5,-0.5) -- (14.5,-1.5);
\draw[dashed] (12.5,-1.5) -- (14.5,0.5);
\draw (12.5,1.5) -- (12,1) -- (12.5,0.5) -- (12,1) -- (11,0) -- (12,-1) -- (12.5,-0.5) -- (12,-1) -- (12.5,-1.5);
\draw (14.5,1.5) -- (16,0) -- (15.5,-0.5) -- (14.5,0.5) -- (15,0) -- (14.5,-0.5) -- (15,0) -- (15.5,-0.5) -- (14.5,-1.5);
\filldraw[fill=black,draw=black] (12,1) circle (1pt);
\filldraw[fill=black,draw=black] (11,0) circle (1pt);
\filldraw[fill=black,draw=black] (12,-1) circle (1pt);
\filldraw[fill=black,draw=black] (15,0) circle (1pt);
\filldraw[fill=black,draw=black] (15.5,-0.5) circle (1pt);
\filldraw[fill=black,draw=black] (16,0) circle (1pt);
\end{tikzpicture}
\end{center}
\caption{Layouts from Example \ref{layoutexample}.}
\label{notation}
\end{figure}
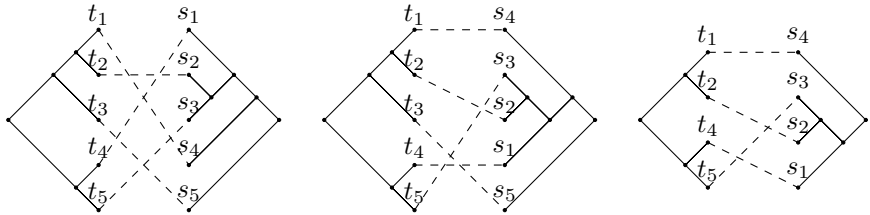

We conclude this section with known results that will be used in our work. These results were previously mentioned, and we include them below for ease of citation. Recall that $\crt(T,S,\phi)$ is the minimum number of crossings over all layouts of a tanglegram $(T,S,\phi)$.

\begin{theorem}\label{nphard}
\cite[Theorem 3]{tlhard} The Tanglegram Layout Problem is NP-hard. 
\end{theorem}
\begin{theorem}\label{change}
\cite[Theorem 3]{analogy} Let $(T,S,\phi)$ be a tanglegram of size $n$ and let $(T_I,S_{\phi(I)},\phi|_I)$ be an induced subtanglegram of size $n-1$. Then $\crt(T,S,\phi)-\crt(T_I,S_{\phi(I)},\phi|_I)\leq n-3$  .
\end{theorem}
\begin{theorem}\label{nchoose2}
\cite[Theorem 5]{analogy} If $(T,S,\phi)$ is a tanglegram of size $n$, then $\crt(T,S,\phi)<\frac{1}{2}{n\choose 2}$.
\end{theorem}

\section{Characterization of planar tanglegram layouts}\label{planar}

In this section, we will start by giving our \algname{ModifiedUntangle} algorithm, based on the \algname{Untangle} algorithm  by Lozano et al. in \cite{layout}. Our additional steps involve the set $L$. We will describe the significance of elements in this set, and then use it to characterize all planar layouts of a planar tanglegram. We conclude this section with our enumerative result on planar tanglegrams.

\subsection{ModifiedUntangle Algorithm}

The \algname{Untangle} Algorithm by Lozano et al. starts by computing a table of Boolean values $P$, where $P[u,v]$ is True for $u\in T,v\in S$ if a descendant of $u$ is matched with a descendant of $v$ by $\phi$. It begins with the ordered lists $X=(\rut(T))$ and $Y=(\rut(S))$ and then refines these lists by replacing vertices with their children in an order chosen based on the Boolean table $P$. They call these ordered lists of vertices $(X,Y)$ from each step \textit{partial layouts}, as they prescribe the order that some of the vertices in the tanglegram are drawn. \algname{Untangle} terminates when the lists in $(X,Y)$ contain only leaves of $T$ and $S$, respectively. It then outputs this pair of lists $(X,Y)$, which is an actual tanglegram layout. We give our modified version of \algname{Untangle} below, which has additional steps involving a set $L$. Removing these steps yields the original \algname{Untangle} algorithm. We will explain the significance of the set $L$ in the next subsection. \\

\vspace{-11pt}

\begin{algorithm}[H]\DontPrintSemicolon
\caption{$\algname{ModifiedUntangle}$ (based on \cite[Algorithm 2]{layout})}\label{modifieduntangle}
\KwIn{planar tanglegram $(T,S,\phi)$ with leaves $\{t_1,\ldots ,t_n\}$ and $\{s_1,\ldots ,s_n\}$}
\KwOut{a planar layout $(X,Y)$ of $(T,S,\phi)$, list of leaf-matched pairs $L\subseteq T\times S$}
$P\coloneqq$ Boolean table with $P[u,v]=False$ $\forall$ vertices $u\in T$, $v\in S$\label{booleantable}\\
set $P[t_i,s_{\phi(i)}]=True$ $\forall$ $i\in [n]$\\
recursively set $P[u,v]=True$ for internal vertices $u\in T,v\in S$ if there exists $u'\leq_T u$, $v'\leq_S v$ with $P[u',v']=True$\\
$X\coloneqq (\rut(T))$, $Y\coloneqq (\rut(S))$ as ordered lists \label{rootlayout}\\
$E\coloneqq\{(\rut(T),\rut(S))\}$ as a set of edges\\
$L\coloneqq\emptyset$\\
\While{$X\cup Y$ contains an internal vertex of $T$ or $S$\label{partiallayouts}}
{$u\coloneqq$ internal vertex of $T\cup S$ with highest degree in the bipartite graph $G=(X,Y,E)$\label{highestdegree} \\
\If{$u\in X$} 
{
\If{$u$ has degree 1 in $G$} 
{update $L\coloneqq L\cup (u,v)$, where $v$ is the unique neighbor of $u$ in $G$}
update $X,E\coloneqq \algname{Refine}(X,Y,u,E,P)$ }
\ElseIf{$u\in Y$}
{
\If{$u$ has degree 1 in $G$} 
{update $L\coloneqq L\cup (v,u)$, where $v$ is the unique neighbor of $u$ in $G$}
update $Y,E\coloneqq \algname{Refine}(Y,X,u,E,P)$}
}
\Return $(X,Y),L$
\end{algorithm}

\phantom{-}\\
\vspace{-11pt}

\begin{algorithm}[H]\DontPrintSemicolon
\caption{$\algname{Refine}$ (based on  \cite[Algorithm 3]{layout})}
\KwIn{ordered lists of vertices $(A,B)$, $u\in A$, edges $E$ on $A\cup B$, Boolean table $P$}
\KwOut{$A,E$ after $u$ has been replaced with its children}
$u_1,u_2\coloneqq$ children of $u$ in $T\cup S$\\
\For{$j\in [m]$ such that $(u,b_j)\in E$ where $B=(b_1,\ldots ,b_m)$} 
{
update $E\coloneqq E\setminus \{(u,b_j)\}$ \tcp*{delete edges involving $u$}
\For{ $i\in \{1,2\}$ \label{edges}}
{
\If{$P[u_i,b_j]=True$}
{update $E\coloneqq E\cup \{(u_i,b_j)\}$ \tcp*{insert edges involving $u_1$ or $u_2$}}
}
}
$k\coloneqq \max\{j\in [m]:(u_1,b_j)\in E\}$ \tcp*{last vertex in $B$ adjacent to $u_1$}
\If {$j>k$ for all $(u_2,b_j)\in E$}
{
replace $u$ with $u_1u_2$ in $A$
}
\Else
{
replace $u$ with $u_2u_1$ in $A$
}
\Return $A,E$
\end{algorithm}

\begin{figure}[h]
    \begin{center}
    \begin{tikzpicture}[scale=0.6]
    \filldraw[fill=black,draw=black] (1,0) circle (2pt);
    \filldraw[fill=black,draw=black] (1,1) circle (2pt);
    \filldraw[fill=black,draw=black] (1,2) circle (2pt);
    \filldraw[fill=black,draw=black] (1,3) circle (2pt);
    \filldraw[fill=black,draw=black] (1,4) circle (2pt);
    \filldraw[fill=black,draw=black] (0,0.5) circle (2pt);
    \filldraw[fill=black,draw=black] (0,1.5) circle (2pt);
    \filldraw[fill=black,draw=black] (0,2.5) circle (2pt);
    \draw (0,2.5)--(1,4);
    \draw (0,2.5)--(1,3);
    \draw (0,1.5)--(1,3);
    \draw (0,1.5)--(1,2);
    \draw (0,1.5)--(1,1);
    \draw (0,0.5)--(1,0);
    \node at (-0.325,1.5) {$u$};
    \node at (1.5,0) {$b_5$};
    \node at (1.5,1) {$b_4$};
    \node at (1.5,2) {$b_3$};
    \node at (1.5,3) {$b_2$};
    \node at (1.5,4) {$b_1$};
    \end{tikzpicture}
    \qquad
        \begin{tikzpicture}[scale=0.6]
    \filldraw[fill=black,draw=black] (1,0) circle (2pt);
    \filldraw[fill=black,draw=black] (1,1) circle (2pt);
    \filldraw[fill=black,draw=black] (1,2) circle (2pt);
    \filldraw[fill=black,draw=black] (1,3) circle (2pt);
    \filldraw[fill=black,draw=black] (1,4) circle (2pt);
    \filldraw[fill=black,draw=black] (0,0.5) circle (2pt);
    \filldraw[fill=black,draw=black] (0,1.75) circle (2pt);
    \filldraw[fill=black,draw=black] (0,1.25) circle (2pt);
    \filldraw[fill=black,draw=black] (0,2.5) circle (2pt);
    \draw (0,2.5)--(1,4);
    \draw (0,2.5)--(1,3);
    \draw (0,1.25)--(1,3);
    \draw (0,1.25)--(1,2);
    \draw (0,1.75)--(1,2);
    \draw (0,1.75)--(1,1);
    \draw (0,0.5)--(1,0);
    \node at (-0.5,1.25) {$u_2$};
    \node at (-0.5,1.75) {$u_1$};
    \node at (1.5,0) {$b_5$};
    \node at (1.5,1) {$b_4$};
    \node at (1.5,2) {$b_3$};
    \node at (1.5,3) {$b_2$};
    \node at (1.5,4) {$b_1$};
    \end{tikzpicture}
    \qquad 
            \begin{tikzpicture}[scale=0.6]
    \filldraw[fill=black,draw=black] (1,0) circle (2pt);
    \filldraw[fill=black,draw=black] (1,1) circle (2pt);
    \filldraw[fill=black,draw=black] (1,2) circle (2pt);
    \filldraw[fill=black,draw=black] (1,3) circle (2pt);
    \filldraw[fill=black,draw=black] (1,4) circle (2pt);
    \filldraw[fill=black,draw=black] (0,0.5) circle (2pt);
    \filldraw[fill=black,draw=black] (0,1.75) circle (2pt);
    \filldraw[fill=black,draw=black] (0,1.25) circle (2pt);
    \filldraw[fill=black,draw=black] (0,2.5) circle (2pt);
    \draw (0,2.5)--(1,4);
    \draw (0,2.5)--(1,3);
    \draw (0,1.75)--(1,3);
    \draw (0,1.75)--(1,2);
    \draw (0,1.25)--(1,2);
    \draw (0,1.25)--(1,1);
    \draw (0,0.5)--(1,0);
    \node at (-0.5,1.25) {$u_1$};
    \node at (-0.5,1.75) {$u_2$};
    \node at (1.5,0) {$b_5$};
    \node at (1.5,1) {$b_4$};
    \node at (1.5,2) {$b_3$};
    \node at (1.5,3) {$b_2$};
    \node at (1.5,4) {$b_1$};
    \end{tikzpicture}
    \end{center}
    \caption{A visualization of the {\normalfont{\algname{Refine}}} algorithm. Draw the inputs as a bipartite graph $G=(A,E,B)$ shown on the left, where the vertices are drawn from top to bottom based on the order of elements in $A$ and $B$. To refine the vertex $u$ of highest degree, consider the embedding shown in the middle with $u_1$ drawn above $u_2$, where new edges involving $u_1$ and $u_2$ are drawn using the Boolean table $P$. Since this drawing has crossings, {\normalfont{\algname{Refine}}} will replace $u$ with $u_2u_1$. Notice that in this case, drawing $u_2$ above $u_1$ results in a planar drawing, as shown on the right.}
    \label{untangleex}
\end{figure}
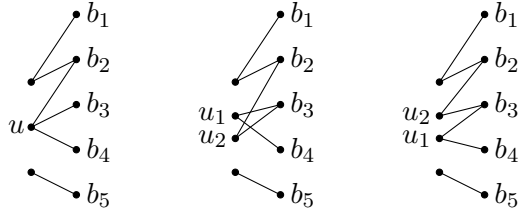

A description of the steps in \algname{Refine} is shown in Figure \ref{untangleex}. Note that \algname{Refine} does not actually check if drawing $u_2$ above $u_1$ results in a planar embedding. This is because planarity of $(T,S,\phi)$ guarantees that at least one of these embeddings will be planar, provided we made appropriate choices at previous refinements. Details are addressed in the proof of the next theorem. Since results from this proof will be relevant for our work, we repeat it below for the convenience of the reader. Note that Lemma \ref{extends} is stated more generally than in \cite{layout}. This extra generality will be useful in the next subsection. 

\begin{definition}
A partial layout $(X,Y)$ for a tanglegram is called \textit{promising} if it can be extended to a planar layout by successively replacing vertices with their children in some order.
\end{definition}

\begin{lemma}[\cite{layout}, Lemma 3]\label{extends}
Let $(X,Y)$ be a promising partial layout of a planar tanglegram $(T,S,\phi)$, and let $E$ be the set of edges on $X\cup Y$ generated using the Boolean table $P$, that is, for all $u\in X$ and $v\in Y$, $(u,v)\in E$ if and only if $P[u,v]=True$. Let $u$ be a vertex of highest degree in the bipartite graph $(X,E,Y)$. 
\begin{enumerate}[label=(\alph*)]
    \item If $\deg(u)= 1$, then replacing $u$ with $u_1u_2$ or $u_2u_1$ results in a promising partial layout.
    \item If $\deg(u)>1$, then either replacing $u$ with $u_1u_2$ or replacing $u$ with $u_2u_1$ results in a promising partial layout, but not both.
\end{enumerate}
In particular, if $(X,Y)$ is promising at the beginning of an iteration of the \textbf{while} loop in \algname{Untangle}, then it is promising at the end.
\end{lemma}

\begin{proof}
Without loss of generality, we will assume $u\in T$, as the result when $u\in S$ is done similarly. We let $(X_1,Y)$ and $(X_2,Y)$ be the partial layouts obtained by replacing $u$ with $u_1u_2$ or $u_2u_1$, respectively. Since $(X,Y)$ is promising, at least one of these partial layouts must be promising, so assume that $(X_1,Y)$ is promising. 

First, suppose $\deg(u)=1$ in $(X,E,Y)$. Since $u$ is a vertex of maximum degree, the unique neighbor of $u$, denoted $v$, also has degree 1. 
Notice that since $u$ and $v$ have degree 1, $\leaf(u)$ and $\leaf(v)$ must be matched by $\phi$. Extend $(X_1,Y)$ to a planar layout $(X',Y')$ of $(T,S,\phi)$, where $\leaf(u_1)$ appears before $\leaf(u_2)$. If we perform a flip at $u$ and a flip at $v$, then we obtain a layout where $\leaf(u_2)$ appears before $\leaf(u_1)$, as shown in Figure \ref{bothextend}. Notice that this is a planar layout that can be obtained by replacing $u$ with $u_2u_1$ instead, implying $(X_2,Y)$ is also promising. We see that regardless of how we replace $u$ with $u_1$ and $u_2$, the resulting partial layout is promising. 

\begin{figure}[h]
    \centering
    \begin{tikzpicture}
        \filldraw[fill=black,draw=black] (-3,0) circle (1pt);
        \filldraw[fill=black,draw=black] (-2,1) circle (1pt);
        \filldraw[fill=black,draw=black] (-2,-1) circle (1pt);
        \filldraw[fill=black,draw=black] (-1.5,1.5) circle (1pt);
        \filldraw[fill=black,draw=black] (-1.5,0.5) circle (1pt);
        \filldraw[fill=black,draw=black] (2,0) circle (1pt);
        \filldraw[fill=black,draw=black] (1,1) circle (1pt);
        \filldraw[fill=black,draw=black] (1,-1) circle (1pt);
        \draw (2,0) -- (1,1) -- (0.1,1.9) -- (0.1,0.1) -- (1,1);
        \draw (2,0) -- (1,-1) -- (0.1,-1.9) -- (0.1,-0.1) -- (1,-1);
        \draw (-3,0) -- (-2,-1) -- (-1.1,-1.9) -- (-1.1,-0.1) -- (-2,-1);
        \draw (-3,0) -- (-1.1,1.9) -- (-1.1,1.1) -- (-1.5,1.5);
        \draw (-2,1) -- (-1.1,0.1) -- (-1.1,0.9) -- (-1.5,0.5);
        \node at (-2,1.25) {$u$};
        \node at (1,1.25) {$v$};
        \node at (-1.625,1.75) {$u_1$};
        \node at (-1.625,0.25) {$u_2$};
        \node at (-0.875,1.75) {\small{$t_1$}};
        \node at (-0.875,1.05) {{$\vdots$}};
        \node at (-0.875,0.25) {\small{$t_k$}};
        \node at (-0.125,1.75) {\small{$s_1$}};
        \node at (-0.125,1.05) {{$\vdots$}};
        \node at (-0.125,0.25) {\small{$s_k$}};
        \draw[dashed,color=black!25!white] (-1.1,1.75) -- (0.1,1.75);
        \draw[dashed,color=black!25!white] (-1.1,1.5) -- (0.1,1.5);
        \draw[dashed,color=black!25!white] (-1.1,1.25) -- (0.1,1.25);
        \draw[dashed,color=black!25!white] (-1.1,.75) -- (0.1,.75);
        \draw[dashed,color=black!25!white] (-1.1,.5) -- (0.1,.5);
        \draw[dashed,color=black!25!white] (-1.1,.25) -- (0.1,.25);
        \draw[dashed,color=black!25!white] (-1.1,-1.75) -- (0.1,-1.75);
        \draw[dashed,color=black!25!white] (-1.1,-1.5) -- (0.1,-1.5);
        \draw[dashed,color=black!25!white] (-1.1,-1.25) -- (0.1,-1.25);
        \draw[dashed,color=black!25!white] (-1.1,-1) -- (0.1,-1);
        \draw[dashed,color=black!25!white] (-1.1,-.75) -- (0.1,-.75);
        \draw[dashed,color=black!25!white] (-1.1,-.5) -- (0.1,-.5);
        \draw[dashed,color=black!25!white] (-1.1,-.25) -- (0.1,-.25);
    \end{tikzpicture}
    \quad
    \begin{tikzpicture}
        \filldraw[fill=black,draw=black] (-3,0) circle (1pt);
        \filldraw[fill=black,draw=black] (-2,1) circle (1pt);
        \filldraw[fill=black,draw=black] (-2,-1) circle (1pt);
        \filldraw[fill=black,draw=black] (-1.5,1.5) circle (1pt);
        \filldraw[fill=black,draw=black] (-1.5,0.5) circle (1pt);
        \filldraw[fill=black,draw=black] (2,0) circle (1pt);
        \filldraw[fill=black,draw=black] (1,1) circle (1pt);
        \filldraw[fill=black,draw=black] (1,-1) circle (1pt);
        \draw (2,0) -- (1,1) -- (0.1,1.9) -- (0.1,0.1) -- (1,1);
        \draw (2,0) -- (1,-1) -- (0.1,-1.9) -- (0.1,-0.1) -- (1,-1);
        \draw (-3,0) -- (-2,-1) -- (-1.1,-1.9) -- (-1.1,-0.1) -- (-2,-1);
        \draw (-3,0) -- (-1.1,1.9) -- (-1.1,1.1) -- (-1.5,1.5);
        \draw (-2,1) -- (-1.1,0.1) -- (-1.1,0.9) -- (-1.5,0.5);
        \node at (-2,1.25) {$u$};
        \node at (1,1.25) {$v$};
        \node at (-1.625,1.75) {$u_2$};
        \node at (-1.625,0.25) {$u_1$};
        \node at (-0.875,1.75) {\small{$t_k$}};
        \node at (-0.875,1.05) {{$\vdots$}};
        \node at (-0.875,0.25) {\small{$t_1$}};
        \node at (-0.125,1.75) {\small{$s_k$}};
        \node at (-0.125,1.05) {{$\vdots$}};
        \node at (-0.125,0.25) {\small{$s_1$}};
        \draw[dashed,color=black!25!white] (-1.1,1.75) -- (0.1,1.75);
        \draw[dashed,color=black!25!white] (-1.1,1.5) -- (0.1,1.5);
        \draw[dashed,color=black!25!white] (-1.1,1.25) -- (0.1,1.25);
        \draw[dashed,color=black!25!white] (-1.1,.75) -- (0.1,.75);
        \draw[dashed,color=black!25!white] (-1.1,.5) -- (0.1,.5);
        \draw[dashed,color=black!25!white] (-1.1,.25) -- (0.1,.25);
        \draw[dashed,color=black!25!white] (-1.1,-1.75) -- (0.1,-1.75);
        \draw[dashed,color=black!25!white] (-1.1,-1.5) -- (0.1,-1.5);
        \draw[dashed,color=black!25!white] (-1.1,-1.25) -- (0.1,-1.25);
        \draw[dashed,color=black!25!white] (-1.1,-1) -- (0.1,-1);
        \draw[dashed,color=black!25!white] (-1.1,-.75) -- (0.1,-.75);
        \draw[dashed,color=black!25!white] (-1.1,-.5) -- (0.1,-.5);
        \draw[dashed,color=black!25!white] (-1.1,-.25) -- (0.1,-.25);
    \end{tikzpicture}
    \caption{Starting with the layout $(X',Y')$ on the left with $\{t_i\}_{i=1}^k$ and $\{s_i\}_{i=1}^k$ denoting leaves, performing flips at $u$ and $v$ produces another planar layout. }
    \label{bothextend}
\end{figure}
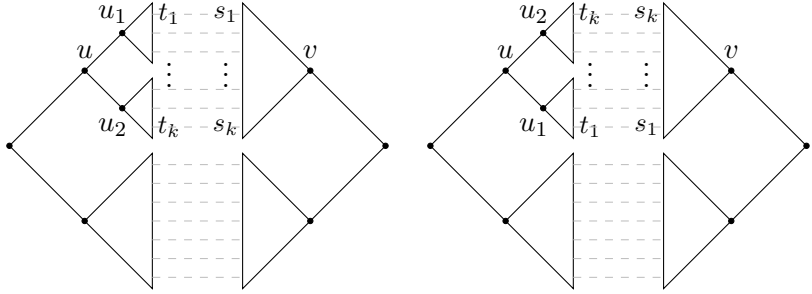

Next, suppose that $\deg(u)>1$. As before, we suppose $(X_1,Y)$ is promising. Let $E'$ be the edges on $X_1\cup Y$ constructed using the Boolean table $P$, and for $i=1,2$, we let $N(u_i)$ denote the set of neighbors of $u_i$ in $(X_1,E',Y)$. We claim that $N(u_1)\Delta N(u_2)\neq \emptyset$, where $\Delta$ denotes the symmetric difference of two sets. To see this, suppose that $N(u_1)=N(u_2)$. If $|N(u_1)|=|N(u_2)|=1$, then this would imply $\deg(u)=1$ in $(X,E,Y)$, which by assumption cannot be the case. Hence, $|N(u_1)|=|N(u_2)|\geq 2$. Then there exists some pair of vertices $v_1,v_2\in N(u_1)=N(u_2)$ that are each adjacent to both $u_1$ and $u_2$. However, this implies a crossing occurs in both $(X_1,E',Y)$ and $(X_2,E',Y)$, as shown in Figure \ref{neighbor2}. Then by construction of the Boolean table $P$, there exist some $t_{1,1},t_{1,2}\in \leaf(u_1)$ and $t_{2,1},t_{2,2}\in \leaf(u_2)$ where each $t_{i,j}\in \leaf(u_i)$ is matched to some $s_{i,j}\in \leaf(v_j)$. Regardless of any refinements of $(X,Y)$, the resulting layout will have either the crossing $(t_{1,2},s_{1,2})$ and $(t_{2,1},s_{2,1})$, or the crossing $(t_{1,1},s_{1,1})$ and $(t_{2,2},s_{2,2})$. Since $(X_1,Y)$ is assumed to be promising, it must be that $N(u_1)\neq N(u_2)$.

\begin{figure}[h]
    \centering
    \begin{tikzpicture}[scale=0.8]
\filldraw[fill=black,draw=black] (0,0) circle (2pt);
\filldraw[fill=black,draw=black] (1,0.5) circle (2pt);
\filldraw[fill=black,draw=black] (0,-0.75) circle (2pt);
\filldraw[fill=black,draw=black] (1,-1) circle (2pt);
\draw (0,0) -- (1,0.5);
\draw (0,-0.75) -- (1,-1);
\draw (0,0) -- (1,-1);
\draw (1,0.5) -- (0,-0.75);
\node at (-0.325,0) {$u_1$};
\node at (-0.325,-0.75) {$u_2$};
\node at (1.325,0.5) {$v_1$};
\node at (1.325,-1) {$v_2$};
\end{tikzpicture}
\quad 
    \begin{tikzpicture}[scale=0.8]
\filldraw[fill=black,draw=black] (0,0) circle (2pt);
\filldraw[fill=black,draw=black] (1,0.5) circle (2pt);
\filldraw[fill=black,draw=black] (0,-0.75) circle (2pt);
\filldraw[fill=black,draw=black] (1,-1) circle (2pt);
\draw (0,0) -- (1,0.5);
\draw (0,-0.75) -- (1,-1);
\draw (0,0) -- (1,-1);
\draw (1,0.5) -- (0,-0.75);
\node at (-0.325,0) {$u_2$};
\node at (-0.325,-0.75) {$u_1$};
\node at (1.325,0.5) {$v_1$};
\node at (1.325,-1) {$v_2$};
\end{tikzpicture}
\caption{Arrangements of the vertices $u_1,u_2,v_1$ and $v_2$ when $N(u_1)= N(u_2)$ and $\deg(u)\geq 2$.}\label{neighbor2}
\end{figure}
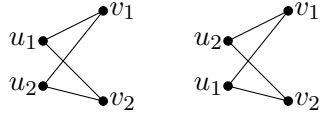

With $N(u_1)\Delta N(u_2)\neq \emptyset$ established, we let $v\in N(u_1)\Delta N(u_2)$. First, we assume that $v=v_1\in N(u_1)\setminus N(u_2)$. Since $(X_1,Y)$ is promising, it must be that drawing $(X_1,E',Y)$ with vertices in the orders indicated by $X_1$ and $Y$ produces no crossings. Then $v_1$ must appear before any $v_2\in N(u_2)$ in the list $Y$. Furthermore, for any $v_2\in N(u_2)$, if we interchange $u_1$ and $u_2$, then the edges $(u_1,v_1)$ and $(u_2,v_2)$ will intersect. Regardless of our future refinements, there will exist some leaves $t_i\in \leaf(u_1),t_j\in \leaf(u_2),s_{\phi(i)}\in \leaf(v_1),s_{\phi(j)}\in \leaf(v_2)$ such that $(t_i,s_{\phi(i)})$ and $(t_j,s_{\phi(j)})$ intersect, which implies $(X_2,Y)$ cannot be promising. A similar argument applies when $v=v_2\in N(u_2)\setminus N(u_1)$. 

Now consider the \textbf{while} loop in \algname{Untangle}. The algorithm uses \algname{Refine} on a vertex $u$ of highest degree in $(X,E,Y)$, which replaces $u$ with $u_1$ and $u_2$. If $\deg(u)=1$, then (a) shows that $(X,Y)$ is promising regardless of the choice at $u$. If $\deg(u)>1$, \algname{Refine} replaces $(X,Y)$ with $(X_1,Y)$ or $(X_2,Y)$ based on whichever bipartite graph $(X_1,E',Y)$ or $(X_2,E',Y)$ does not have crossings, and our proof of (b) shows that this results in a promising partial layout.
\end{proof}

\begin{theorem}\label{ModifiedUntangle} 
\cite[Theorem 1]{layout} For any planar tanglegram $(T,S,\phi)$, the 
\algname{Untangle} algorithm terminates in a planar layout $(X,Y)$.
\end{theorem}
\begin{proof}
If $(T,S,\phi)$ is planar, then $(X,Y)=(\rut(T),\rut(S))$ in line \ref{rootlayout} is promising.
By Lemma \ref{extends}, $(X,Y)$ is promising after each iteration of the \textbf{while} loop. This loop terminates when $(X,Y)$ contains only leaves of $T$ and $S$, which must then be a planar layout.
\end{proof}

\begin{remark}
In the proofs of Lemma \ref{extends} and Theorem \ref{ModifiedUntangle}, the arguments hold regardless of which vertex of highest degree $u$ is selected. In fact, we do not even need to select the vertex of highest degree at every step in \algname{Untangle}. We specified a vertex of highest degree for simplicity. As long as \algname{Untangle} does not use \algname{Refine} on a vertex in $(X,E,Y)$ with degree one while its neighbor has degree more than one, the algorithm will still output a planar layout for a planar tanglegram.
\end{remark}

\begin{lemma}\label{untanglespeed}\cite[Lemma 4]{layout}
\algname{Untangle} runs in $O(n^2)$ time and space, where $n$ is the size of the tanglegram.
\end{lemma}

\begin{remark}\label{modifieduntangletime}
The additional steps in $\algname{ModifiedUntangle}$ involve the set $L$, which has size at most $n-1$. Hence, for a planar tanglegram of size $n$, Theorem \ref{ModifiedUntangle} and Lemma \ref{untanglespeed} also imply that \algname{ModifiedUntangle} terminates in a planar layout $(X,Y)$ and runs in $O(n^2)$ time and space.
\end{remark}

\subsection{Leaf-matched pairs and paired flips}

We now consider the additional steps involving the set $L$ in \algname{ModifiedUntangle} algorithm. While the \algname{Untangle} algorithm produces a planar layout for planar tanglegrams, one might ask how to generate all of them. As noted in the proof of Lemma \ref{extends}, if $(X,Y)$ is a planar layout of $(T,S,\phi)$, then one method to generate additional planar layouts is using flips at vertices $u\in T$ and $v\in S$ where $\leaf(u)$ and $\leaf(v)$ are matched by $\phi$. We give a name for these pairs of vertices and the operation involving a flip at both vertices, followed by an example in Figure \ref{pairedflipexample}. Then we show that \algname{ModifiedUntangle} identifies these pairs in the set $L$.

\begin{definition}
Let $(T,S,\phi)$ be a tanglegram with layout $(X,Y)$. A pair of internal vertices $(u,v)$ with $u\in T$ and $v\in S$ is a \textit{leaf-matched pair} of $(T,S,\phi)$ if $\leaf (u)$ and $\leaf(v)$ are matched by $\phi$. A \textit{paired flip} at $(u,v)$ is the operation on $(T,S,\phi)$ corresponding to a flip at $u$ and a flip at $v$. This maps $(X,Y)$ to the layout $(X',Y')$, where $X'$ is the image of $X$ after a flip at $u$ and $Y'$ is the image of $Y$ after a flip at $v$.
\end{definition}

\begin{figure}[h]
\begin{center}
\begin{tikzpicture}[scale=0.5]
\filldraw[fill=black,draw=black] (1,4) circle (2pt);
\filldraw[fill=black,draw=black] (1,3) circle (2pt);
\filldraw[fill=black,draw=black] (1,2) circle (2pt);
\filldraw[fill=black,draw=black] (1,1) circle (2pt);
\filldraw[fill=black,draw=black] (1,0) circle (2pt);
\filldraw[fill=black,draw=black] (1,-1) circle (2pt);
\filldraw[fill=black,draw=black] (1,-2) circle (2pt);
\filldraw[fill=black,draw=black] (1,-3) circle (2pt);
\filldraw[fill=black,draw=black] (-1,4) circle (2pt);
\filldraw[fill=black,draw=black] (-1,3) circle (2pt);
\filldraw[fill=black,draw=black] (-1,2) circle (2pt);
\filldraw[fill=black,draw=black] (-1,1) circle (2pt);
\filldraw[fill=black,draw=black] (-1,0) circle (2pt);
\filldraw[fill=black,draw=black] (-1,-1) circle (2pt);
\filldraw[fill=black,draw=black] (-1,-2) circle (2pt);
\filldraw[fill=black,draw=black] (-1,-3) circle (2pt);
\draw[dashed] (-1,4) -- (1,4);
\draw[dashed] (-1,3) -- (1,3);
\draw[dashed] (-1,2) -- (1,2);
\draw[dashed] (-1,1) -- (1,1);
\draw[dashed] (-1,0) -- (1,0);
\draw[dashed] (-1,-1) -- (1,-1);
\draw[dashed] (-1,-2) -- (1,-2);
\draw[dashed] (-1,-3) -- (1,-3);
\draw (-1,4) -- (-1.5,3.5) -- (-1,3) -- (-1.5,3.5) -- (-2.5,2.5) -- (-1.5,1.5) -- (-1,2) -- (-1.5,1.5) -- (-1,1) -- (-2.5,2.5) -- (-4.5,0.5);
\draw (-1,-3) -- (-1.5,-2.5) -- (-1,-2) -- (-1.5,-2.5) -- (-2,-2) -- (-1,-1) -- (-2,-2) -- (-2.5,-1.5) -- (-1,0) -- (-2.5,-1.5) -- (-4.5,0.5);
\draw (1,4) -- (1.5,3.5) -- (1,3) -- (1.5,3.5) -- (2,3) -- (1,2) -- (2,3) -- (2.5,2.5) -- (1,1) -- (2.5,2.5) -- (4.5,0.5) ; 
\draw (1,-3) -- (1.5,-2.5) -- (1,-2) -- (1.5,-2.5) -- (4.5,0.5);
\draw (1,0) -- (1.5,-0.5) -- (1,-1) -- (1.5,-0.5) --(3.5,1.5);
\filldraw[fill=black,draw=black] (-1.5,1.5) circle (2pt);
\filldraw[fill=black,draw=black] (-2.5,-1.5) circle (2pt);
\filldraw[fill=black,draw=black] (-2,-2) circle (2pt);
\filldraw[fill=black,draw=black] (2,3) circle (2pt);
\filldraw[fill=black,draw=black] (1.5,-0.5) circle (2pt);
\filldraw[fill=black,draw=black] (3.5,1.5) circle (2pt);
\filldraw[fill=black,draw=black] (-1.5,3.5) circle (2pt);
\filldraw[fill=black,draw=black] (1.5,3.5) circle (2pt);
\filldraw[fill=black,draw=black] (2.5,2.5) circle (2pt);
\filldraw[fill=black,draw=black] (-2.5,2.5) circle (2pt);
\filldraw[fill=black,draw=black] (4.5,0.5) circle (2pt);
\filldraw[fill=black,draw=black] (-4.5,0.5) circle (2pt);
\node at (-2.75,2.75) {\small{$u$}};
\node at (2.75,2.75) {\small{$v$}};
\node at (-0.75,4.325) {\small{$t_1$}};
\node at (-0.75,3.325) {\small{$t_2$}};
\node at (-0.75,2.325) {\small{$t_3$}};
\node at (-0.75,1.325) {\small{$t_4$}};
\node at (-0.75,0.325) {\small{$t_5$}};
\node at (-0.75,-0.675) {\small{$t_6$}};
\node at (-0.75,-1.675) {\small{$t_7$}};
\node at (-0.75,-2.675) {\small{$t_8$}};
\node at (0.75,4.325) {\small{$s_1$}};
\node at (0.75,3.325) {\small{$s_2$}};
\node at (0.75,2.325) {\small{$s_3$}};
\node at (0.75,1.325) {\small{$s_4$}};
\node at (0.75,0.325) {\small{$s_5$}};
\node at (0.75,-0.675) {\small{$s_6$}};
\node at (0.75,-1.675) {\small{$s_7$}};
\node at (0.75,-2.675) {\small{$s_8$}};
\filldraw[fill=black,draw=black] (1.5,-2.5) circle (2pt);
\filldraw[fill=black,draw=black] (-1.5,-2.5) circle (2pt);
\end{tikzpicture}
\quad
\begin{tikzpicture}[scale=0.5]
\filldraw[fill=black,draw=black] (1,4) circle (2pt);
\filldraw[fill=black,draw=black] (1,3) circle (2pt);
\filldraw[fill=black,draw=black] (1,2) circle (2pt);
\filldraw[fill=black,draw=black] (1,1) circle (2pt);
\filldraw[fill=black,draw=black] (1,0) circle (2pt);
\filldraw[fill=black,draw=black] (1,-1) circle (2pt);
\filldraw[fill=black,draw=black] (1,-2) circle (2pt);
\filldraw[fill=black,draw=black] (1,-3) circle (2pt);
\filldraw[fill=black,draw=black] (-1,4) circle (2pt);
\filldraw[fill=black,draw=black] (-1,3) circle (2pt);
\filldraw[fill=black,draw=black] (-1,2) circle (2pt);
\filldraw[fill=black,draw=black] (-1,1) circle (2pt);
\filldraw[fill=black,draw=black] (-1,0) circle (2pt);
\filldraw[fill=black,draw=black] (-1,-1) circle (2pt);
\filldraw[fill=black,draw=black] (-1,-2) circle (2pt);
\filldraw[fill=black,draw=black] (-1,-3) circle (2pt);
\draw[dashed] (-1,4) -- (1,4);
\draw[dashed] (-1,3) -- (1,3);
\draw[dashed] (-1,2) -- (1,2);
\draw[dashed] (-1,1) -- (1,1);
\draw[dashed] (-1,0) -- (1,0);
\draw[dashed] (-1,-1) -- (1,-1);
\draw[dashed] (-1,-2) -- (1,-2);
\draw[dashed] (-1,-3) -- (1,-3);
\draw (-1,4) -- (-1.5,3.5) -- (-1,3) -- (-1.5,3.5) -- (-2.5,2.5) -- (-1.5,1.5) -- (-1,2) -- (-1.5,1.5) -- (-1,1) -- (-2.5,2.5) -- (-4.5,0.5);
\draw (-1,-3) -- (-1.5,-2.5) -- (-1,-2) -- (-1.5,-2.5) -- (-2,-2) -- (-1,-1) -- (-2,-2) -- (-2.5,-1.5) -- (-1,0) -- (-2.5,-1.5) -- (-4.5,0.5);
\draw (1,1) -- (1.5,1.5) -- (1,2) -- (1.5,1.5) -- (2,2) -- (1,3) -- (2,2) -- (2.5,2.5) -- (1,4)  -- (4.5,0.5) ; 
\draw (1,-3) -- (1.5,-2.5) -- (1,-2) -- (1.5,-2.5) -- (4.5,0.5);
\draw (1,0) -- (1.5,-0.5) -- (1,-1) -- (1.5,-0.5) --(3.5,1.5);
\filldraw[fill=black,draw=black] (-1.5,1.5) circle (2pt);
\filldraw[fill=black,draw=black] (-2.5,-1.5) circle (2pt);
\filldraw[fill=black,draw=black] (-2,-2) circle (2pt);
\filldraw[fill=black,draw=black] (2,2) circle (2pt);
\filldraw[fill=black,draw=black] (1.5,-0.5) circle (2pt);
\filldraw[fill=black,draw=black] (3.5,1.5) circle (2pt);
\filldraw[fill=black,draw=black] (-1.5,3.5) circle (2pt);
\filldraw[fill=black,draw=black] (1.5,1.5) circle (2pt);
\filldraw[fill=black,draw=black] (2.5,2.5) circle (2pt);
\filldraw[fill=black,draw=black] (-2.5,2.5) circle (2pt);
\filldraw[fill=black,draw=black] (4.5,0.5) circle (2pt);
\filldraw[fill=black,draw=black] (-4.5,0.5) circle (2pt);
\filldraw[fill=black,draw=black] (1.5,-2.5) circle (2pt);
\filldraw[fill=black,draw=black] (-1.5,-2.5) circle (2pt);
\node at (-0.75,4.325) {\small{$t_4$}};
\node at (-0.75,3.325) {\small{$t_3$}};
\node at (-0.75,2.325) {\small{$t_2$}};
\node at (-0.75,1.325) {\small{$t_1$}};
\node at (-0.75,0.325) {\small{$t_5$}};
\node at (-0.75,-0.675) {\small{$t_6$}};
\node at (-0.75,-1.675) {\small{$t_7$}};
\node at (-0.75,-2.675) {\small{$t_8$}};
\node at (0.75,4.325) {\small{$s_4$}};
\node at (0.75,3.325) {\small{$s_3$}};
\node at (0.75,2.325) {\small{$s_2$}};
\node at (0.75,1.325) {\small{$s_1$}};
\node at (0.75,0.325) {\small{$s_5$}};
\node at (0.75,-0.675) {\small{$s_6$}};
\node at (0.75,-1.675) {\small{$s_7$}};
\node at (0.75,-2.675) {\small{$s_8$}};
\node at (-2.75,2.75) {\small{$u$}};
\node at (2.75,2.75) {\small{$v$}};
\end{tikzpicture}
\caption{A paired flip at $(u,v)$ maps each layout to the other one.}
\label{pairedflipexample}
\end{center}
\end{figure}
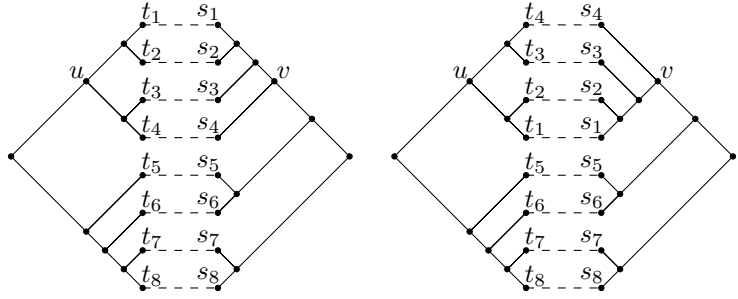

\begin{lemma}\label{degree1}
Let $(T,S,\phi)$ be a planar tanglegram. A pair of internal vertices $(u,v)$ is a leaf-matched pair of a planar tanglegram $(T,S,\phi)$ if and only if at some step of the \algname{ModifiedUntangle} algorithm, the internal vertices $u\in T$ and $v\in S$ appear as adjacent degree one vertices in $(X,E,Y)$. Hence, the set $L$ returned by \algname{ModifiedUntangle} is the set of leaf-matched pairs of $(T,S,\phi)$.
\end{lemma}
\begin{proof}
Suppose the internal vertices $u\in T$ and $v\in S$ appear as adjacent degree one vertices in $(X,E,Y)$ during some step of the \algname{ModifiedUntangle} algorithm. Since both vertices have degree one, the construction of the edges in $(X,E,Y)$ using the Boolean table $P$ in line \ref{edges} of \algname{Refine} implies that $\leaf(u)$ and $\leaf(v)$ are matched under $\phi$, and thus $(u,v)$ is a leaf-matched pair. 

Conversely, suppose $(u,v)$ is a leaf-matched pair. If $u$ and $v$ are the root vertices of $T$ and $S$, then these appear as degree one vertices at the first step of \algname{ModifiedUntangle}. Otherwise, at some step of the algorithm, either $u$ or $v$ will appear for the first time in a partial layout $(X,Y)$, and without loss of generality, we assume it is $u\in X$. Since $v$ has not appeared in a partial layout yet, there is some vertex $v'\in Y$ that is an ancestor of $v$. Then $\leaf(u)$ is matched with a proper subset of $\leaf(v')$ in the tanglegram $(T,S,\phi)$, so $\deg(v')>1$ in $(X,E,Y)$. From line \ref{highestdegree} of \algname{ModifiedUntangle}, we see that $v'$ will be replaced with its children before $u$ is. Repeating this argument, $v$ will eventually appear before we use \algname{Refine} on $u$, and at this time, $u$ and $v$ will be adjacent degree one vertices since they are a leaf-matched pair of $(T,S,\phi)$. 
\end{proof}

We know that given a planar layout $(X,Y)$ of $(T,S,\phi)$, paired flips will generate additional planar layouts, but we do not yet know that they generate all possible planar layouts. It may be possible that some appropriate choice of flips not equivalent to a sequence of paired flips also results in a planar layout. We will show that this is in fact not the case. Our proof for this uses Lemma \ref{extends}, which holds for arbitrary promising partial layouts, not just ones considered in \algname{ModifiedUntangle}.

First, notice that for a tanglegram $(T,S,\phi)$ of size $n$, \algname{ModifiedUntangle} starts with a promising partial layout $(X_1,Y_1)=(\rut(T),\rut(S))$. At each step, it replaces an internal vertex of $T$ or $S$ using \algname{Refine}. Since a tree on $n$ leaves has $n-1$ internal vertices, the algorithm uses \algname{Refine} a total of $2n-2$ times. Thus, it produces a sequence of promising partial layouts $\{(X_k,Y_k)\}_{k=1}^{2n-1}$ that terminates at $(X,Y)=(X_{2n-1},Y_{2n-1})$.
We give a name for such a sequence.

\begin{definition}\label{promisingsequence}
Let $(T,S,\phi)$ be a tanglegram with layout $(X,Y)$. We call $\{(X_k,Y_k)\}_{k=1}^{2n-1}$ a \textit{partial sequence for $(X,Y)$} if
\begin{itemize}
    \item for $k=1,2,\ldots,2n-1$, $(X_k,Y_k)$ is a partial layout,
    \item $(X_1,Y_1)=(\rut(T),\rut(S))$,
    \item $(X_{2n-1},Y_{2n-1})=(X,Y)$, and
    \item for $k=1,2,...,2n-2$, the partial layout $(X_{k+1},Y_{k+1})$ is obtained from $(X_k,Y_k)$ by refining some vertex $u\in X_k\cup Y_k$, that is, replacing $u$ with its children in an appropriate order.
\end{itemize}
A partial sequence is \textit{promising} if all $(X_k,Y_k)$ are promising partial layouts, or equivalently, if $(X,Y)=(X_{2n-1},Y_{2n-1})$ is a planar layout.
\end{definition}

If $(X',Y')$ is another planar layout of $(T,S,\phi)$, we can use $\{(X_k,Y_k)\}_{k=1}^{2n-1}$ to find a promising partial sequence $\{(X_k',Y_k')\}_{k=1}^{2n-1}$ for $(X',Y')$. We do this by constructing each $(X_k',Y_k')$ as follows.
\begin{itemize}
    \item Draw the trees $T$ and $S$ with leaves from top to bottom in the order described by $(X',Y')$. 
    \item For each $u\in X_k$, contract the subtree of $T$ rooted at $u$ to the vertex $u$ itself. Do the same for each $v\in Y_k$, and call the resulting trees $T_k$ and $S_k$.
    \item Let $X_k'$ be the leaves of $T_k$ listed from top to bottom, and similarly for $Y_k'$ and $S_k$. 
\end{itemize}
An example of these steps is shown below in Figure \ref{contraction}. Note that by construction, $(X_k,Y_k)$ and $(X_k',Y_k')$ contain the same vertices, but possibly in different orders. We now show that $\{(X_k',Y_k')\}_{k=1}^{2n-1}$ constructed in this manner is a promising partial sequence for $(X',Y')$ and then use this to establish Theorem \ref{planarlayouts}.

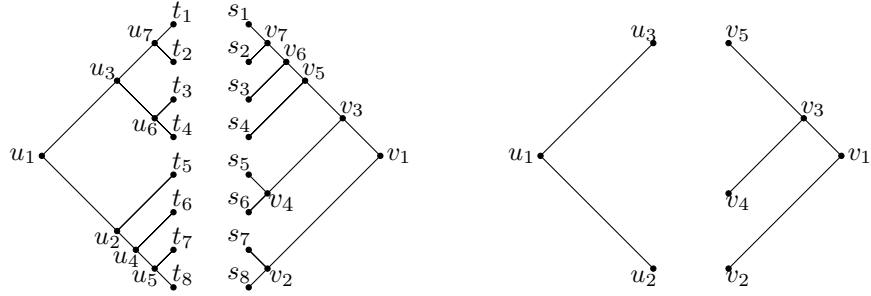
\begin{figure}[h]
    \centering
\begin{tikzpicture}[scale=0.5]
\filldraw[fill=black,draw=black] (1,4) circle (2pt);
\filldraw[fill=black,draw=black] (1,3) circle (2pt);
\filldraw[fill=black,draw=black] (1,2) circle (2pt);
\filldraw[fill=black,draw=black] (1,1) circle (2pt);
\filldraw[fill=black,draw=black] (1,0) circle (2pt);
\filldraw[fill=black,draw=black] (1,-1) circle (2pt);
\filldraw[fill=black,draw=black] (1,-2) circle (2pt);
\filldraw[fill=black,draw=black] (1,-3) circle (2pt);
\filldraw[fill=black,draw=black] (-1,4) circle (2pt);
\filldraw[fill=black,draw=black] (-1,3) circle (2pt);
\filldraw[fill=black,draw=black] (-1,2) circle (2pt);
\filldraw[fill=black,draw=black] (-1,1) circle (2pt);
\filldraw[fill=black,draw=black] (-1,0) circle (2pt);
\filldraw[fill=black,draw=black] (-1,-1) circle (2pt);
\filldraw[fill=black,draw=black] (-1,-2) circle (2pt);
\filldraw[fill=black,draw=black] (-1,-3) circle (2pt);
\draw (-1,4) -- (-1.5,3.5) -- (-1,3) -- (-1.5,3.5) -- (-2.5,2.5) -- (-1.5,1.5) -- (-1,2) -- (-1.5,1.5) -- (-1,1) -- (-2.5,2.5) -- (-4.5,0.5);
\draw (-1,-3) -- (-1.5,-2.5) -- (-1,-2) -- (-1.5,-2.5) -- (-2,-2) -- (-1,-1) -- (-2,-2) -- (-2.5,-1.5) -- (-1,0) -- (-2.5,-1.5) -- (-4.5,0.5);
\draw (1,4) -- (1.5,3.5) -- (1,3) -- (1.5,3.5) -- (2,3) -- (1,2) -- (2,3) -- (2.5,2.5) -- (1,1) -- (2.5,2.5) -- (4.5,0.5) ; 
\draw (1,-3) -- (1.5,-2.5) -- (1,-2) -- (1.5,-2.5) -- (4.5,0.5);
\draw (1,0) -- (1.5,-0.5) -- (1,-1) -- (1.5,-0.5) --(3.5,1.5);
\filldraw[fill=black,draw=black] (-1.5,1.5) circle (2pt);
\filldraw[fill=black,draw=black] (-2.5,-1.5) circle (2pt);
\filldraw[fill=black,draw=black] (-2,-2) circle (2pt);
\filldraw[fill=black,draw=black] (2,3) circle (2pt);
\filldraw[fill=black,draw=black] (1.5,-0.5) circle (2pt);
\filldraw[fill=black,draw=black] (3.5,1.5) circle (2pt);
\filldraw[fill=black,draw=black] (-1.5,3.5) circle (2pt);
\filldraw[fill=black,draw=black] (1.5,3.5) circle (2pt);
\filldraw[fill=black,draw=black] (2.5,2.5) circle (2pt);
\filldraw[fill=black,draw=black] (-2.5,2.5) circle (2pt);
\filldraw[fill=black,draw=black] (4.5,0.5) circle (2pt);
\filldraw[fill=black,draw=black] (-4.5,0.5) circle (2pt);
\node at (-5,0.5) {\small{$u_1$}};
\node at (-2.75,-1.75) {\small{$u_2$}};
\node at (-2.875,2.75) {\small{$u_3$}};
\node at (-2.25,-2.25) {\small{$u_4$}};
\node at (-1.75,-2.75) {\small{$u_5$}};
\node at (-1.875,3.75) {\small{$u_7$}};
\node at (-1.75,1.25) {\small{$u_6$}};
\node at (5,0.5) {\small{$v_1$}};
\node at (3.75,1.75) {\small{$v_3$}};
\node at (2.75,2.75) {\small{$v_5$}};
\node at (2.25,3.25) {\small{$v_6$}};
\node at (1.875,-2.75) {\small{$v_2$}};
\node at (1.75,3.75) {\small{$v_7$}};
\node at (1.875,-0.75) {$v_4$};
\node at (-0.75,4.325) {\small{$t_1$}};
\node at (-0.75,3.325) {\small{$t_2$}};
\node at (-0.75,2.325) {\small{$t_3$}};
\node at (-0.75,1.325) {\small{$t_4$}};
\node at (-0.75,0.325) {\small{$t_5$}};
\node at (-0.75,-0.675) {\small{$t_6$}};
\node at (-0.75,-1.675) {\small{$t_7$}};
\node at (-0.75,-2.675) {\small{$t_8$}};
\node at (0.75,4.325) {\small{$s_1$}};
\node at (0.75,3.325) {\small{$s_2$}};
\node at (0.75,2.325) {\small{$s_3$}};
\node at (0.75,1.325) {\small{$s_4$}};
\node at (0.75,0.325) {\small{$s_5$}};
\node at (0.75,-0.675) {\small{$s_6$}};
\node at (0.75,-1.675) {\small{$s_7$}};
\node at (0.75,-2.675) {\small{$s_8$}};
\filldraw[fill=black,draw=black] (1.5,-2.5) circle (2pt);
\filldraw[fill=black,draw=black] (-1.5,-2.5) circle (2pt);
\end{tikzpicture}
    $\qquad$ 
\begin{tikzpicture}[scale=0.5]
\draw (-1.5,3.5) -- (-4.5,0.5);
\draw  (-1.5,-2.5) -- (-4.5,0.5);
\draw (3.5,0.5) -- (0.5,3.5);
\draw (3.5,0.5) -- (0.5,-2.5);
\draw (2.5,1.5) -- (0.5,-0.5);
\filldraw[fill=black,draw=black] (-1.5,-2.5) circle (2pt);
\filldraw[fill=black,draw=black] (0.5,3.5) circle (2pt);
\filldraw[fill=black,draw=black] (-1.5,3.5) circle (2pt);
\filldraw[fill=black,draw=black] (3.5,0.5) circle (2pt);
\filldraw[fill=black,draw=black] (-4.5,0.5) circle (2pt);
\filldraw[fill=black,draw=black] (2.5,1.5) circle (2pt);
\filldraw[fill=black,draw=black] (0.5,-0.5) circle (2pt);
\node at (-5,0.5) {\small{$u_1$}};
\node at (-1.75,-2.75) {\small{$u_2$}};
\node at (-1.75,3.75) {\small{$u_3$}};
\node at (4,0.5) {\small{$v_1$}};
\node at (2.75,1.75) {\small{$v_3$}};
\node at (0.75,3.75) {\small{$v_5$}};
\node at (0.75,-2.75) {\small{$v_2$}};
\node at (0.75,-0.75) {\small{$v_4$}};
\filldraw[fill=black,draw=black] (0.5,-2.5) circle (2pt);
\end{tikzpicture}
    \caption{Starting with drawings of $T$ and $S$ corresponding to $(X',Y')$, we use $(X_k,Y_k)=(u_2u_3,v_2v_4v_5)$ to form the contracted trees $T_k$ and $S_k$ shown on the right. Listing leaves from top to bottom gives us the partial layout $(X'_k,Y'_k)=(u_3u_2,v_5v_4v_2)$.}
    \label{contraction}
\end{figure}

\begin{lemma}\label{sequence}
The sequence $\{(X_k',Y_k')\}_{k=1}^{2n-1}$ is a promising partial sequence for $(X',Y')$.
\end{lemma}
\begin{proof}
 By construction, notice that each $(X_k',Y_k')$ is a partial layout, $(X_1',Y_1')=(\rut(T),\rut(S))$, and $(X_{2n-1}',Y_{2n-1}')=(X',Y')$. It remains to show that refining a vertex $u\in X_k'\cup Y_k'$ produces $(X_{k+1}',Y_{k+1}')$. Denote the vertex refined in $(X_k,Y_k)$ to obtain $(X_{k+1},Y_{k+1})$ as $u_k$, and without loss of generality, we assume $u_k\in X_k$. This implies that $(X_{k}',Y_{k}')$ and $(X_{k+1}',Y_{k+1}')$ have almost the same vertices, except $(X_{k}',Y_{k}')$ contains $u_k$, while $(X_{k+1}',Y_{k+1}')$ contains its two children $u_{k,1}$ and $u_{k,2}$. By our construction using contractions, the tree $T_{k}$ can be obtained from $T_{k+1}$ by contracting the two children of $u_k$ onto the vertex itself. Thus, we can obtain $(X_{k}',Y_{k}')$ from $(X_{k+1}',Y_{k+1}')$ by replacing the adjacent children of $u_k$ with $u_k$ itself. Then we can also obtain $(X_{k+1}',Y_{k+1}')$ from $(X_k',Y_k')$ by refining $u_k$. Thus, all conditions in Definition \ref{promisingsequence} are satisfied, so we see that $\{(X_k',Y_k')\}_{k=1}^{2n-1}$ is a promising partial sequence for $(X',Y')$.
\end{proof}

\begin{proof}[Proof of Theorem \ref{planarlayouts}]
By Lemma \ref{degree1}, $L$ is the set of leaf-matched pairs of $(T,S,\phi)$. It is clear that if $(u,v)\in L$, then starting with $(X,Y)$ and performing a paired flip at $(u,v)$ produces another layout in $ \mathscr{P}(T,S,\phi)$, so the same is true if we perform any sequence of paired flips starting with $(X,Y)$. We show that all planar layouts can  be obtained this way.

Let $(X',Y')\in \mathscr{P}(T,S,\phi)$ be distinct from $(X,Y)$. Let $\{(X_k,Y_k)\}_{k=1}^{2n-1}$ be the promising partial sequence for $(X,Y)$ produced in \algname{ModifiedUntangle}, with corresponding bipartite graphs $\{(X_k,E_k,Y_k)\}_{k=1}^{2n-1}$. By Lemma \ref{sequence}, we can use this sequence for $(X,Y)$ to construct a corresponding promising partial sequence $\{(X_k',Y_k')\}_{k=1}^{2n-1}$ for $(X',Y')$ where $(X_k,Y_k)$ and $(X_k',Y_k')$ contain the same vertices, though possibly in different orders.
Thus, if we use the Boolean table $P$ to construct edges $E_k'$ on the vertices in $X_k'\cup Y_k'$, then $E_k'=E_k$. Since $(X',Y')\neq (X,Y)$, there must be some minimal $m$ such that $(X_{m+1},Y_{m+1})\neq (X_{m+1}',Y_{m+1}')$. Without loss of generality, we assume the refined vertex at this step was $u\in X_m$ and that \algname{Refine} replaced $u$ with $u_{1}u_{2}$ to obtain $(X_{m+1},Y_{m+1})$, while $(X_{m+1}',Y_{m+1}')$ requires replacing $u$ with $u_{2}u_{1}$.

Consider the degree of $u$ in the bipartite graph $(X_m, E_m,Y_m)=(X_m',E_m,Y_m')$. If $\deg(u)\geq 2$, then Lemma \ref{extends} implies that $(X_{m+1}',Y_{m+1}')$ is not promising, which is not the case since $\{(X_k',Y_k')\}_{k=1}^{2n-1}$ is a promising partial sequence for $(X',Y')$. Thus, it must be that $\deg(u)=1$. Since \algname{ModifiedUntangle} always refines a vertex of highest degree, this implies that all vertices in $(X_m, E_m,Y_m)=(X_m',E_m,Y_m')$ must have degree 1.

Let $v\in Y_m$ be the unique neighbor of $u$ in $(X_m,E_m,Y_m)$. Once we replace $u$ with its children, notice that $v$ will be the unique vertex in $(X_{m+1},E_{m+1},Y_{m+1})$ with $\deg(v)\geq 2$. Thus, after \algname{ModifiedUntangle} replaces $u$ with $u_{1}u_{2}$ to obtain $(X_{m+1},Y_{m+1})$, it will replace $v$ with its children in some order to obtain $(X_{m+2},Y_{m+2})$. 

Lemma \ref{degree1} implies that $(u,v)\in L$, so we let $(\widetilde{X},\widetilde{Y})\in \mathscr{P}(T,S,\phi)$ be $(X,Y)$ after a paired flip at $(u,v)$. Using $\{(X_k,Y_k)\}_{k=1}^{2n-1}$, we again use Lemma \ref{sequence} to construct a promising partial sequence $\{(\widetilde{X}_k,\widetilde{Y}_k)\}_{k=1}^{2n-1}$ for $(\widetilde{X},\widetilde{Y})$. By construction, $(\widetilde{X}_k,\widetilde{Y}_k)=(X_k,Y_k)$ for all $k\leq m$, and because of the paired flip at $(u,v)$, we see that $(\widetilde{X}_k,\widetilde{Y}_k)=(X'_k,Y'_k)$ for all $k\leq m+1$. Furthermore, the preceding paragraph implies that $(\widetilde{X}_{m+2},\widetilde{Y}_{m+2})$ is obtained from $(\widetilde{X}_{m+1},\widetilde{Y}_{m+1})$ by refining the vertex $v$. Since $\deg(v)=2$ in $(\widetilde{X}_{m+1},E_{m+1},\widetilde{Y}_{m+1})$, Lemma \ref{extends} implies that a unique choice for the children of $v$ results in a promising partial layout, and thus it must be that $(\widetilde{X}_{m+2},\widetilde{Y}_{m+2})=(X'_{m+2},Y'_{m+2})$. 

If $(\widetilde{X},\widetilde{Y})\neq (X',Y')$, we can repeat the above argument. Eventually, this process terminates in a planar layout $(\widetilde{X},\widetilde{Y})$ obtained from a sequence of paired flips starting at $(X,Y)$, where  $(\widetilde{X}_k,\widetilde{Y}_k)=(X_k',Y'_k)$ for all $k$. Hence, we see that $(X',Y')=(\widetilde{X},\widetilde{Y})$, and any $(X',Y')\in\mathscr{P}(T,S,\phi)$ can be obtained using a sequence of paired flips starting with $(X,Y)$. 
\end{proof}

\begin{remark}\label{irreducible}
A tanglegram is \textit{irreducible} if its only leaf matched pair is given by the roots of the two trees. For irreducible tanglegrams $(T,S,\phi)$ of size at least three, Theorem \ref{planarlayouts} implies that there are exactly two planar layouts, which must be mirror images of one another. This specific case was established and used by Ralaivaosaona, Ravelomanana, and Wagner to enumerate planar tanglegrams \cite[Proposition 5]{countingplanar}.
\end{remark}

We now define an undirected graph called the flip graph of a planar tanglegram. Theorem \ref{planarlayouts} gives us a corollary about this graph. By using the outputted layout from \algname{ModifiedUntangle} and considering all subsets of $L$, one could in principle find all planar layouts of a planar tanglegram $(T,S,\phi)$ and produce the flip graph of a tanglegram, though there may be exponentially many planar layouts.

\begin{definition}
Let $(T,S,\phi)$ be a planar tanglegram. Define the \textit{flip graph of $(T,S,\phi)$} as $\Gamma(T,S,\phi)=(V,E)$ with vertices $v_{(X,Y)}\in V$ corresponding to planar layouts $(X,Y)$, and edges $(v_{(X,Y)},v_{(X',Y')})\in E$ if $(X',Y')$ can be obtained from $(X,Y)$ by a paired flip at some leaf-matched pair $(u,v)$ of $(T,S,\phi)$. 
\end{definition}

\begin{corollary}
The flip graph of a planar tanglegram is connected.
\end{corollary}

\subsection{Enumeration of planar tanglegrams by number of leaf-matched pairs}

In this subsection, we show an enumerative result for the number of planar tanglegrams of size $n$ with $k$ leaf-matched pairs. Recall the generating functions $F(x,q)$ and $H(x)$ from (\ref{Fxq}) and (\ref{Hx}). We first give a definition, and then establish our generalization of \cite[Theorem 1]{countingplanar}.

\begin{definition}
The \textit{irreducible component} of a tanglegram $(T,S,\phi)$, denoted $\text{Irr}(T,S,\phi)$, is the irreducible tanglegram formed by contracting each non-root leaf-matched pair of $(T,S,\phi)$ to a single pair of matched leaves.
\end{definition}

\begin{proof}[Proof of Theorem \ref{Fxqrelation}]
Equation (\ref{Fxqequation}) is equivalent to
\begin{equation}\label{Fxqequation2}
    F(x,q) =x+q\cdot \left( H(F(x,q))-\frac{F(x,q)^2}{2} \right)+q\cdot \frac{ F(x,q)^2+F(x^2,q^2)}{2},
\end{equation}
so we establish this relation instead.
The term $x$ accounts for the unique tanglegram of size 1, which has no leaf-matched pairs. For the remaining tanglegrams, we can form each tanglegram $(T,S,\phi)$ by starting with its irreducible component $\text{Irr}(T,S,\phi)$ and replacing matched leaves with planar tanglegrams (possibly of size 1). We consider two cases depending on the size of $\text{Irr}(T,S,\phi)$. 

First, consider tanglegrams with $\text{size}(\text{Irr}(T,S,\phi))\geq 3$. As noted in the proof of \cite[Theorem 1]{countingplanar}, these tanglegrams do not have any symmetry. The generating function $q\cdot [H(x)-x^2/2]$ counts irreducible planar tanglegrams of size $n\geq 3$. A term $qx^n$ corresponds to a planar irreducible tanglegram of size $n$, and replacing a pair of matched leaves with a planar tanglegram corresponds to replacing $x$ with $F(x,q)$. Hence, $q\cdot [H(F(x,q))-F(x,q)^2/2]$ enumerates tanglegrams with irreducible component of size $n \geq 3$.

Second, suppose $\text{size}(\text{Irr}(T,S,\phi))=2$. These tanglegrams are formed by starting with the unique planar tanglegram of size two corresponding to the term $qx^2$ and replacing the two pairs of leaves with two planar tanglegrams $\{(T_1,S_1,\phi_1),(T_2,S_2,\phi_2)\}$, where the order is not relevant. The generating function $F(x,q)^2$ would count ordered pairs of planar tanglegrams. This correctly counts the case when $(T_1,S_1,\phi_1)$ and $(T_2,S_2,\phi_2)$ are the same, but counts all other cases twice. To remedy this over-counting, we can add $F(x^2,q^2)$, which counts the pairs where $(T_1,S_1,\phi_1)$ and $(T_2,S_2,\phi_2)$ are the same, and then divide the result by two to account for the order not being relevant. Hence, $q\cdot \frac{F(x,q)^2+F(x^2,q^2)}{2}$ enumerates tanglegrams with irreducible component of size two. Combined, we obtain (\ref{Fxqequation2}).
\end{proof}

Note that substituting $q=1$ results in the original relation given in \cite[Theorem 1]{countingplanar}. Using this result and the coefficients of $H(x)$ from \cite{countingplanar}, it takes a computer-algebra system only a few seconds to generate several coefficients of $F(x,q)$. We collect some of these coefficients in Table \ref{leafmatchedtable}. See \cite[A349409]{oeis} for more terms. The corresponding planar tanglegrams for $n=4$ are shown in Figure \ref{size4tanglegrams}.

\begin{table}[h]
    \centering
    \begin{tabular}{c | c |  c | c| c| c |c |  c}
        $n,k$  & 1 &  2 & 3 & 4 & 5 & 6 & total \\  \hline 
        2 & 1 & & & & & &   1\\  \hline 
        3 & 1 & 1 & & & & &  2\\  \hline 
        4 & 5 & 4 & 2 & & & &  11\\  \hline 
        5 & 34 & 28 & 11 & 3  & & &  76\\  \hline 
        6 & 273 & 239 & 102 & 29 & 6 & &  649\\  \hline 
        7 & 2436 & 2283 & 1045 & 325 & 73 & 11 &  6173\\  
    \end{tabular}
    \caption{Coefficients of $x^nq^k$ in $F(x,q)$ for $2\leq n\leq 7$.}
    \label{leafmatchedtable}
\end{table}

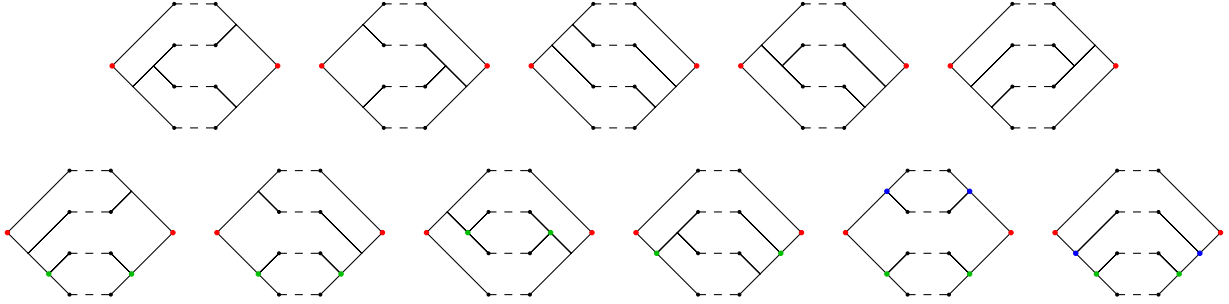
\begin{figure}[h]
    \centering
    \begin{tikzpicture}[scale=0.55]
    \filldraw[fill=black,draw=black] (0,3) circle (1pt);
    \filldraw[fill=black,draw=black] (0,2) circle (1pt);
    \filldraw[fill=black,draw=black] (0,1) circle (1pt);
    \filldraw[fill=black,draw=black] (0,0) circle (1pt);
    \filldraw[fill=black,draw=black] (1,3) circle (1pt);
    \filldraw[fill=black,draw=black] (1,2) circle (1pt);
    \filldraw[fill=black,draw=black] (1,1) circle (1pt);
    \filldraw[fill=black,draw=black] (1,0) circle (1pt);
    \draw[dashed] (0,3) -- (1,3);
    \draw[dashed] (0,2) -- (1,2);
    \draw[dashed] (0,1) -- (1,1);
    \draw[dashed] (0,0) -- (1,0);
    \draw (0,3) -- (-1.5,1.5) -- (-1,1) -- (0,2) -- (-0.5,1.5) -- (0,1) -- (-0.5,1.5) -- (-1,1) -- (0,0);
    \draw (1,3) -- (1.5,2.5) -- (1,2) -- (1.5,2.5) -- (2.5,1.5) -- (1.5,0.5) -- (1,1) -- (1.5,0.5) -- (1,0);
    \filldraw[fill=red,draw=red] (2.5,1.5) circle (1.5pt);
    \filldraw[fill=red,draw=red] (-1.5,1.5) circle (1.5pt);
    \end{tikzpicture}
    \quad 
        \begin{tikzpicture}[scale=0.55]
    \filldraw[fill=black,draw=black] (0,3) circle (1pt);
    \filldraw[fill=black,draw=black] (0,2) circle (1pt);
    \filldraw[fill=black,draw=black] (0,1) circle (1pt);
    \filldraw[fill=black,draw=black] (0,0) circle (1pt);
    \filldraw[fill=black,draw=black] (1,3) circle (1pt);
    \filldraw[fill=black,draw=black] (1,2) circle (1pt);
    \filldraw[fill=black,draw=black] (1,1) circle (1pt);
    \filldraw[fill=black,draw=black] (1,0) circle (1pt);
    \draw[dashed] (0,3) -- (1,3);
    \draw[dashed] (0,2) -- (1,2);
    \draw[dashed] (0,1) -- (1,1);
    \draw[dashed] (0,0) -- (1,0);
    \draw (0,3) -- (-0.5,2.5) -- (0,2) -- (-0.5,2.5)-- (-1.5,1.5) -- (-0.5,0.5) -- (0,1) -- (-0.5,0.5) -- (0,0);
    \draw (1,3) -- (2.5,1.5) -- (2,1) -- (1,2) -- (1.5,1.5) -- (1,1) -- (1.5,1.5) -- (2,1) -- (1,0);
    \filldraw[fill=red,draw=red] (2.5,1.5) circle (1.5pt);
    \filldraw[fill=red,draw=red] (-1.5,1.5) circle (1.5pt);
    \end{tikzpicture}
        \quad 
        \begin{tikzpicture}[scale=0.55]
    \filldraw[fill=black,draw=black] (0,3) circle (1pt);
    \filldraw[fill=black,draw=black] (0,2) circle (1pt);
    \filldraw[fill=black,draw=black] (0,1) circle (1pt);
    \filldraw[fill=black,draw=black] (0,0) circle (1pt);
    \filldraw[fill=black,draw=black] (1,3) circle (1pt);
    \filldraw[fill=black,draw=black] (1,2) circle (1pt);
    \filldraw[fill=black,draw=black] (1,1) circle (1pt);
    \filldraw[fill=black,draw=black] (1,0) circle (1pt);
    \draw[dashed] (0,3) -- (1,3);
    \draw[dashed] (0,2) -- (1,2);
    \draw[dashed] (0,1) -- (1,1);
    \draw[dashed] (0,0) -- (1,0);
    \draw (0,3) -- (-0.5,2.5) -- (0,2) -- (-0.5,2.5) -- (-1,2) -- (0,1) -- (-1,2) -- (-1.5,1.5) -- (0,0);
    \draw (1,3) -- (2.5,1.5) -- (2,1) -- (1,2) -- (2,1) -- (1.5,0.5) -- (1,1) -- (1.5,0.5) -- (1,0);
    \filldraw[fill=red,draw=red] (2.5,1.5) circle (1.5pt);
    \filldraw[fill=red,draw=red] (-1.5,1.5) circle (1.5pt);
    \end{tikzpicture}
            \quad 
        \begin{tikzpicture}[scale=0.55]
    \filldraw[fill=black,draw=black] (0,3) circle (1pt);
    \filldraw[fill=black,draw=black] (0,2) circle (1pt);
    \filldraw[fill=black,draw=black] (0,1) circle (1pt);
    \filldraw[fill=black,draw=black] (0,0) circle (1pt);
    \filldraw[fill=black,draw=black] (1,3) circle (1pt);
    \filldraw[fill=black,draw=black] (1,2) circle (1pt);
    \filldraw[fill=black,draw=black] (1,1) circle (1pt);
    \filldraw[fill=black,draw=black] (1,0) circle (1pt);
    \draw[dashed] (0,3) -- (1,3);
    \draw[dashed] (0,2) -- (1,2);
    \draw[dashed] (0,1) -- (1,1);
    \draw[dashed] (0,0) -- (1,0);
    \draw (0,3) -- (-1,2) -- (-0.5,1.5) -- (0,2) -- (-0.5,1.5) -- (0,1) -- (-1,2) -- (-1.5,1.5) -- (0,0);
    \draw (1,3) -- (2.5,1.5) -- (2,1) -- (1,2) -- (2,1) -- (1.5,0.5) -- (1,1) -- (1.5,0.5) -- (1,0);
    \filldraw[fill=red,draw=red] (2.5,1.5) circle (1.5pt);
    \filldraw[fill=red,draw=red] (-1.5,1.5) circle (1.5pt);
    \end{tikzpicture}
    \quad 
    \begin{tikzpicture}[scale=0.55]
    \filldraw[fill=black,draw=black] (0,3) circle (1pt);
    \filldraw[fill=black,draw=black] (0,2) circle (1pt);
    \filldraw[fill=black,draw=black] (0,1) circle (1pt);
    \filldraw[fill=black,draw=black] (0,0) circle (1pt);
    \filldraw[fill=black,draw=black] (1,3) circle (1pt);
    \filldraw[fill=black,draw=black] (1,2) circle (1pt);
    \filldraw[fill=black,draw=black] (1,1) circle (1pt);
    \filldraw[fill=black,draw=black] (1,0) circle (1pt);
    \draw[dashed] (0,3) -- (1,3);
    \draw[dashed] (0,2) -- (1,2);
    \draw[dashed] (0,1) -- (1,1);
    \draw[dashed] (0,0) -- (1,0);
    \draw (0,3) -- (-1.5,1.5) -- (-1,1) -- (0,2) -- (-1,1) -- (-0.5,0.5) -- (0,1) -- (-0.5,0.5) -- (0,0);
    \draw (1,3) -- (2,2) -- (1.5,1.5) -- (1,2) -- (1.5,1.5) -- (1,1) -- (2,2) -- (2.5,1.5) -- (1,0);
    \filldraw[fill=red,draw=red] (2.5,1.5) circle (1.5pt);
    \filldraw[fill=red,draw=red] (-1.5,1.5) circle (1.5pt);
    \end{tikzpicture} \\
    \phantom{-}\\
    \begin{tikzpicture}[scale=0.55]
    \filldraw[fill=black,draw=black] (0,3) circle (1pt);
    \filldraw[fill=black,draw=black] (0,2) circle (1pt);
    \filldraw[fill=black,draw=black] (0,1) circle (1pt);
    \filldraw[fill=black,draw=black] (0,0) circle (1pt);
    \filldraw[fill=black,draw=black] (1,3) circle (1pt);
    \filldraw[fill=black,draw=black] (1,2) circle (1pt);
    \filldraw[fill=black,draw=black] (1,1) circle (1pt);
    \filldraw[fill=black,draw=black] (1,0) circle (1pt);
    \draw[dashed] (0,3) -- (1,3);
    \draw[dashed] (0,2) -- (1,2);
    \draw[dashed] (0,1) -- (1,1);
    \draw[dashed] (0,0) -- (1,0);
        \draw (0,3) -- (-1.5,1.5) -- (-1,1) -- (0,2) -- (-1,1) -- (-0.5,0.5) -- (0,1) -- (-0.5,0.5) -- (0,0);
 \draw (1,3) -- (1.5,2.5) -- (1,2) -- (1.5,2.5) -- (2.5,1.5) -- (1.5,0.5) -- (1,1) -- (1.5,0.5) -- (1,0);
 \filldraw[fill=red,draw=red] (2.5,1.5) circle (1.5pt);
    \filldraw[fill=red,draw=red] (-1.5,1.5) circle (1.5pt);
    \filldraw[fill=green!75!black,draw=green!75!black] (1.5,0.5) circle (1.5pt);
    \filldraw[fill=green!75!black,draw=green!75!black] (-0.5,0.5) circle (1.5pt);
    \end{tikzpicture}
    \quad
        \begin{tikzpicture}[scale=0.55]
    \filldraw[fill=black,draw=black] (0,3) circle (1pt);
    \filldraw[fill=black,draw=black] (0,2) circle (1pt);
    \filldraw[fill=black,draw=black] (0,1) circle (1pt);
    \filldraw[fill=black,draw=black] (0,0) circle (1pt);
    \filldraw[fill=black,draw=black] (1,3) circle (1pt);
    \filldraw[fill=black,draw=black] (1,2) circle (1pt);
    \filldraw[fill=black,draw=black] (1,1) circle (1pt);
    \filldraw[fill=black,draw=black] (1,0) circle (1pt);
    \draw[dashed] (0,3) -- (1,3);
    \draw[dashed] (0,2) -- (1,2);
    \draw[dashed] (0,1) -- (1,1);
    \draw[dashed] (0,0) -- (1,0);
    \draw (0,3) -- (-0.5,2.5) -- (0,2) -- (-0.5,2.5)-- (-1.5,1.5) -- (-0.5,0.5) -- (0,1) -- (-0.5,0.5) -- (0,0);
    \draw (1,3) -- (2.5,1.5) -- (2,1) -- (1,2) -- (2,1) -- (1.5,0.5) -- (1,1) -- (1.5,0.5) -- (1,0);
    \filldraw[fill=red,draw=red] (2.5,1.5) circle (1.5pt);
    \filldraw[fill=red,draw=red] (-1.5,1.5) circle (1.5pt);
    \filldraw[fill=green!75!black,draw=green!75!black] (1.5,0.5) circle (1.5pt);
    \filldraw[fill=green!75!black,draw=green!75!black] (-0.5,0.5) circle (1.5pt);
    \end{tikzpicture}
    \quad 
    \begin{tikzpicture}[scale=0.55]
    \filldraw[fill=black,draw=black] (0,3) circle (1pt);
    \filldraw[fill=black,draw=black] (0,2) circle (1pt);
    \filldraw[fill=black,draw=black] (0,1) circle (1pt);
    \filldraw[fill=black,draw=black] (0,0) circle (1pt);
    \filldraw[fill=black,draw=black] (1,3) circle (1pt);
    \filldraw[fill=black,draw=black] (1,2) circle (1pt);
    \filldraw[fill=black,draw=black] (1,1) circle (1pt);
    \filldraw[fill=black,draw=black] (1,0) circle (1pt);
    \draw[dashed] (0,3) -- (1,3);
    \draw[dashed] (0,2) -- (1,2);
    \draw[dashed] (0,1) -- (1,1);
    \draw[dashed] (0,0) -- (1,0);
    \draw (0,3) -- (-1,2) -- (-0.5,1.5) -- (0,2) -- (-0.5,1.5) -- (0,1) -- (-1,2) -- (-1.5,1.5) -- (0,0);
    \draw (1,3) -- (2.5,1.5) -- (2,1) -- (1,2) -- (1.5,1.5) -- (1,1) -- (1.5,1.5) -- (2,1) -- (1,0);
    \filldraw[fill=red,draw=red] (2.5,1.5) circle (1.5pt);
    \filldraw[fill=red,draw=red] (-1.5,1.5) circle (1.5pt);
    \filldraw[fill=green!75!black,draw=green!75!black] (1.5,1.5) circle (1.5pt);
    \filldraw[fill=green!75!black,draw=green!75!black] (-0.5,1.5) circle (1.5pt);
    \end{tikzpicture}
    \quad 
    \begin{tikzpicture}[scale=0.55]
    \filldraw[fill=black,draw=black] (0,3) circle (1pt);
    \filldraw[fill=black,draw=black] (0,2) circle (1pt);
    \filldraw[fill=black,draw=black] (0,1) circle (1pt);
    \filldraw[fill=black,draw=black] (0,0) circle (1pt);
    \filldraw[fill=black,draw=black] (1,3) circle (1pt);
    \filldraw[fill=black,draw=black] (1,2) circle (1pt);
    \filldraw[fill=black,draw=black] (1,1) circle (1pt);
    \filldraw[fill=black,draw=black] (1,0) circle (1pt);
    \draw[dashed] (0,3) -- (1,3);
    \draw[dashed] (0,2) -- (1,2);
    \draw[dashed] (0,1) -- (1,1);
    \draw[dashed] (0,0) -- (1,0);
    \draw (0,3) -- (-1.5,1.5) -- (-1,1) -- (0,2) -- (-0.5,1.5) -- (0,1) -- (-0.5,1.5) -- (-1,1) -- (0,0);
   \draw (1,3) -- (2.5,1.5) -- (2,1) -- (1,2) -- (2,1) -- (1.5,0.5) -- (1,1) -- (1.5,0.5) -- (1,0);
   \filldraw[fill=red,draw=red] (2.5,1.5) circle (1.5pt);
    \filldraw[fill=red,draw=red] (-1.5,1.5) circle (1.5pt);
    \filldraw[fill=green!75!black,draw=green!75!black] (2,1) circle (1.5pt);
    \filldraw[fill=green!75!black,draw=green!75!black] (-1,1) circle (1.5pt);
    \end{tikzpicture}
    \quad 
    \begin{tikzpicture}[scale=0.55]
    \filldraw[fill=black,draw=black] (0,3) circle (1pt);
    \filldraw[fill=black,draw=black] (0,2) circle (1pt);
    \filldraw[fill=black,draw=black] (0,1) circle (1pt);
    \filldraw[fill=black,draw=black] (0,0) circle (1pt);
    \filldraw[fill=black,draw=black] (1,3) circle (1pt);
    \filldraw[fill=black,draw=black] (1,2) circle (1pt);
    \filldraw[fill=black,draw=black] (1,1) circle (1pt);
    \filldraw[fill=black,draw=black] (1,0) circle (1pt);
    \draw[dashed] (0,3) -- (1,3);
    \draw[dashed] (0,2) -- (1,2);
    \draw[dashed] (0,1) -- (1,1);
    \draw[dashed] (0,0) -- (1,0);
    \draw (0,3) -- (-0.5,2.5) -- (0,2) -- (-0.5,2.5)-- (-1.5,1.5) -- (-0.5,0.5) -- (0,1) -- (-0.5,0.5) -- (0,0);
    \draw (1,3) -- (1.5,2.5) -- (1,2) -- (1.5,2.5) -- (2.5,1.5) -- (1.5,0.5) -- (1,1) -- (1.5,0.5) -- (1,0);
    \filldraw[fill=red,draw=red] (2.5,1.5) circle (1.5pt);
    \filldraw[fill=red,draw=red] (-1.5,1.5) circle (1.5pt);
    \filldraw[fill=green!75!black,draw=green!75!black] (1.5,0.5) circle (1.5pt);
    \filldraw[fill=green!75!black,draw=green!75!black] (-0.5,0.5) circle (1.5pt);
    \filldraw[fill=blue,draw=blue] (1.5,2.5) circle (1.5pt);
    \filldraw[fill=blue,draw=blue] (-0.5,2.5) circle (1.5pt);
    \end{tikzpicture}
    \quad 
    \begin{tikzpicture}[scale=0.55]
    \filldraw[fill=black,draw=black] (0,3) circle (1pt);
    \filldraw[fill=black,draw=black] (0,2) circle (1pt);
    \filldraw[fill=black,draw=black] (0,1) circle (1pt);
    \filldraw[fill=black,draw=black] (0,0) circle (1pt);
    \filldraw[fill=black,draw=black] (1,3) circle (1pt);
    \filldraw[fill=black,draw=black] (1,2) circle (1pt);
    \filldraw[fill=black,draw=black] (1,1) circle (1pt);
    \filldraw[fill=black,draw=black] (1,0) circle (1pt);
    \draw[dashed] (0,3) -- (1,3);
    \draw[dashed] (0,2) -- (1,2);
    \draw[dashed] (0,1) -- (1,1);
    \draw[dashed] (0,0) -- (1,0);
        \draw (0,3) -- (-1.5,1.5) -- (-1,1) -- (0,2) -- (-1,1) -- (-0.5,0.5) -- (0,1) -- (-0.5,0.5) -- (0,0);
    \draw (1,3) -- (2.5,1.5) -- (2,1) -- (1,2) -- (2,1) -- (1.5,0.5) -- (1,1) -- (1.5,0.5) -- (1,0);
    \filldraw[fill=red,draw=red] (2.5,1.5) circle (1.5pt);
    \filldraw[fill=red,draw=red] (-1.5,1.5) circle (1.5pt);
    \filldraw[fill=green!75!black,draw=green!75!black] (1.5,0.5) circle (1.5pt);
    \filldraw[fill=green!75!black,draw=green!75!black] (-0.5,0.5) circle (1.5pt);
    \filldraw[fill=blue,draw=blue] (2,1) circle (1.5pt);
    \filldraw[fill=blue,draw=blue] (-1,1) circle (1.5pt);
    \end{tikzpicture}
    \caption{The 11 planar tanglegrams of size 4. The first five tanglegrams are irreducible, the next four have two leaf-matched pairs, and the final two have three pairs.} 
    \label{size4tanglegrams}
\end{figure}

\section{The Tanglegram  Single Edge Insertion Problem}\label{TEI}

In this section, we solve the Tanglegram Single Edge Insertion Problem. For convenience of the reader, we restate the problem below.

\begin{problem*}[\textbf{Tanglegram  Single Edge Insertion}]
Given a tanglegram $(T,S,\phi)$ and a planar subtanglegram $(T_I,S_{\phi(I)},\phi|_I)$ induced by $I=[n]\setminus \{i\}$ for $i\in [n]$, find a layout of $(T,S,\phi)$ that restricts to a planar layout of $(T_I,S_{\phi(I)},\phi|_I)$ and has the minimal number of crossings possible.
\end{problem*}

Since planar layouts of $(T_I,S_{\phi(I)},\phi|_I)$ are relevant for this problem, we can apply our work from Section \ref{planar}. We use the notation $(T,S,\phi|_I)=(T,S,\phi)\setminus \{(t_j,s_{\phi(j)})\}_{j\notin I}$ for a tanglegram with some between-tree edges removed. Note that $(T,S,\phi|_I)$ is well-defined as an input into \algname{ModifiedUntangle}, motivating this notation. Letting $(X,Y),L=\algname{ModifiedUntangle}(T,S,\phi|_I)$, we will show that $(X,Y)$ does restrict to a planar layout of $(T_I,S_{\phi(I)},\phi|_I)$ whenever it is a planar subtanglegram. However, $L$ may contain pairs $(u,v)$ where $u\in T$ or $v\in S$ are not vertices in $(T_I,S_{\phi(I)},\phi|_I)$. One simple example of this situation is shown in Figure \ref{smoothing}.


\begin{figure}[h]
    \centering
\begin{tikzpicture}[scale=0.6]
\filldraw[fill=black,draw=black] (-1,4) circle (1pt);
\filldraw[fill=black,draw=black] (-1,3) circle (1pt);
\filldraw[fill=black,draw=black] (-1,2) circle (1pt);
\filldraw[fill=black,draw=black] (-1,1) circle (1pt);
\filldraw[fill=black,draw=black] (1,4) circle (1pt);
\filldraw[fill=black,draw=black] (1,3) circle (1pt);
\filldraw[fill=black,draw=black] (1,2) circle (1pt);
\filldraw[fill=black,draw=black] (1,1) circle (1pt);
\draw (-1,4) -- (-1.5,3.5) -- (-1,3) -- (-1.5,3.5) -- (-2,3) -- (-1,2) -- (-2,3) -- (-2.5,2.5) -- (-1,1);
\draw (1,4) -- (1.5,3.5) -- (1,3) -- (1.5,3.5) -- (2,3) -- (1,2) -- (2,3) -- (2.5,2.5) -- (1,1);
\filldraw[fill=black,draw=black] (-1.5,3.5) circle (1pt);
\filldraw[fill=black,draw=black] (-2,3) circle (1pt);
\filldraw[fill=black,draw=black] (-2.5,2.5) circle (1pt);
\filldraw[fill=black,draw=black] (1.5,3.5) circle (1pt);
\filldraw[fill=black,draw=black] (2,3) circle (1pt);
\filldraw[fill=black,draw=black] (2.5,2.5) circle (1pt);
\draw[dashed] (-1,4) -- (1,4);
\draw[dashed] (-1,2) -- (1,1);
\draw[dashed] (-1,3) -- (1,2);
\node at (-0.875,4.375) {$t_1$};
\node at (-0.875,3.375) {$t_2$};
\node at (-0.875,2.375) {$t_3$};
\node at (-0.875,1.375) {$t_4$};
\node at (0.875,4.375) {$s_1$};
\node at (0.875,3.375) {$s_2$};
\node at (0.875,2.375) {$s_3$};
\node at (0.875,1.375) {$s_4$};
\node at (-2.75,2.75) {$u$};
\node at (2.75,2.75) {$v$};

\filldraw[fill=black,draw=black] (6,3.5) circle (1pt);
\filldraw[fill=black,draw=black] (6,2.5) circle (1pt);
\filldraw[fill=black,draw=black] (6,1.5) circle (1pt);
\filldraw[fill=black,draw=black] (8,3.5) circle (1pt);
\filldraw[fill=black,draw=black] (8,2.5) circle (1pt);
\filldraw[fill=black,draw=black] (8,1.5) circle (1pt);
\draw[dashed] (6,3.5) -- (8,3.5);
\draw[dashed] (6,2.5) -- (8,2.5);
\draw[dashed] (6,1.5) -- (8,1.5);
\draw (6,3.5) -- (5.5,3) -- (6,2.5) -- (5.5,3) -- (5,2.5) -- (6,1.5);
\draw (8,3.5) -- (8.5,3) -- (8,2.5) -- (8.5,3) -- (9,2.5) -- (8,1.5);
\node at (6.125,3.875) {$t_1$};
\node at (6.125,2.875) {$t_2$};
\node at (6.125,1.875) {$t_3$};
\node at (7.875,3.875) {$s_1$};
\node at (7.875,2.875) {$s_3$};
\node at (7.875,1.875) {$s_4$};
\node at (9.25,2.75) {$v$};
\filldraw[fill=black,draw=black] (5.5,3) circle (1pt);
\filldraw[fill=black,draw=black] (6,2.5) circle (1pt);
\filldraw[fill=black,draw=black] (8.5,3) circle (1pt);
\filldraw[fill=black,draw=black] (9,2.5) circle (1pt);
\filldraw[fill=black,draw=black] (5,2.5) circle (1pt);
\end{tikzpicture}
\caption{When used on $(T,S,\phi|_{\{1,2,3\}})$ shown on the left, {\normalfont{\algname{ModifiedUntangle}}} will add $(u,v)$ to $L$ on the first step of the algorithm. However, $u$ is not a vertex of the subtanglegram $(T_{\{1,2,3\}},S_{\{1,3,4\}},\phi|_{\{1,2,3\}})$, shown on the right.}
\label{smoothing}
\end{figure}
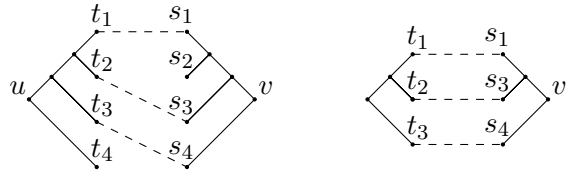

To find the leaf-matched pairs of the subtanglegram, one could input $(T_I,S_{\phi(I)},\phi|_I)$ itself into \algname{ModifiedUntangle}, and then relate the outputs to $(T,S,\phi)$, though this would require extending the outputted layout of $(T_I,S_{\phi(I)},\phi|_I)$ to a layout for $(T,S,\phi)$. Alternatively, we can directly construct the set of leaf-matched pairs of $(T_I,S_{\phi(I)},\phi|_I)$ using $L$. 

\begin{definition}\label{Lsubtanglegram}
\label{L(I)}
Let $(T,S,\phi)$ be a tanglegram of size $n$, and let $I\subseteq [n]$. For any pair of vertices $u\in T$, $v\in S$  such that $\leaf(u)\cap T_I$ is paired with $\leaf(v)\cap S_{\phi(I)}$ by $\phi|_I$, define their \textit{reduced} vertices to be
$$u_{red}=\text{minimal $u'\leq_T u$ such that $\leaf(u)\cap T_I=\leaf(u')\cap T_I$}$$
$$v_{red}=\text{minimal $v'\leq_S v$ such that $\leaf(v)\cap S_{\phi(I)}=\leaf(v')\cap S_{\phi(I)}$}.$$
Given the set $L$ from \algname{ModifiedUntangle}$(T,S,\phi|_I)$, define $L(I)$ to be the set of all $(u_{red},v_{red})$ such that $(u,v)\in L$ and $|\leaf(u)\cap T_I|=|\leaf(v)\cap S_{\phi(I)}|>1$.
\end{definition}

\begin{lemma}
\label{Lrestrict}
Suppose $(T,S,\phi)$ is a tanglegram with planar subtanglegram $(T_I,S_{\phi(I)},\phi|_I)$ induced by $I\subseteq [n]$, and let $(X,Y),L=\algname{ModifiedUntangle}(T,S,\phi|_I)$. If $(u,v)\in L$, then $\leaf(u)\cap T_I$ and $\leaf(v)\cap S_{\phi(I)}$ are matched by $\phi|_I$. 
\end{lemma}
\begin{proof}
Since $(u,v)\in L$, we know $u$ and $v$ appeared as adjacent degree one vertices in a bipartite graph $(X_j,E_j,Y_j)$ during some step of $\algname{ModifiedUntangle}$. The edges $E_j$ are constructed using the Boolean table $P$, and because of the restriction $\phi|_I$, $P[u,v]=True$ if and only if some element of $\leaf(u)\cap T_I$ is matched with an element of $\leaf(v)\cap S_{\phi(I)}$ by $\phi|_I$. From this, we see that $u$ and $v$ being adjacent degree one vertices in $(X_j,E_j,Y_j)$ implies $\leaf(u)\cap T_I$ and $\leaf(v)\cap S_{\phi(I)}$ are matched by $\phi|_I$.
\end{proof}

\begin{lemma}\label{modified}
Suppose $(T,S,\phi)$ is a tanglegram with planar subtanglegram $(T_I,S_{\phi(I)},\phi|_I)$ induced by $I\subseteq [n]$, and let $(X,Y),L=\algname{ModifiedUntangle}(T,S,\phi|_I)$. 
\begin{enumerate}[label=(\alph*)]
    \item The layout $(X,Y)$ restricts to a planar layout of $(T_I,S_{\phi(I)},\phi|_I)$. 
    \item The set $L(I)$ is the set of leaf-matched pairs of $(T_I,S_{\phi(I)},\phi|_I)$. 
\end{enumerate}

\end{lemma}

\begin{proof}
To prove (a), we use a similar argument as in Theorem \ref{ModifiedUntangle}. Call $(X,Y)$ promising if it extends to a planar drawing of $(T,S,\phi|_I)$. The initial partial layout of roots is promising by assumption. We claim that $(X,Y)$ is still promising after an iteration of \algname{Refine}. Suppose we are refining a vertex of highest degree $u\in X$. It is possible that when we use \algname{Refine} on $u$ to replace it with its children $u_1$ and $u_2$, the Boolean table values $P[u_i,v]$ are false for all $v\in Y$, and this occurs precisely when all between-tree edges incident to leaves in $\leaf(u_i)$ have been removed. If this occurs, we assume without loss of generality that it occurs for $u_1$. Since $(X,Y)$ was promising, it extends to some planar drawing, and if we interchange the subtrees rooted at $u_1$ and $u_2$, we obtain another planar drawing since leaves in $\leaf(u_1)$ are not incident to any of the between-tree edges. 
Thus, replacing $u$ with $u_1u_2$ or $u_2u_1$ both result in a promising partial layout. For all other cases, the same arguments from Lemma \ref{extends} apply, and thus \algname{ModifiedUntangle} terminates in a drawing of $(T,S,\phi|_I)$ that is planar. Restricting to just the leaves in $I$, this must be a planar layout of $(T_I,S_{\phi(I)},\phi|_I)$. 

For (b), we let $(X',Y'),L'=\algname{ModifiedUntangle}(T_I,S_{\phi(I)},\phi|_I)$, so we must show that $L(I)=L'$. For the remainder of this proof, we will use $\leaf(u)$ to denote the leaves of $u$ in $T$ and use $\leaf(u)\cap T_I$ to denote the leaves of $u$ in $T_I$. We use a corresponding notation for vertices $v\in S$.  

We claim that $L(I)\subseteq L'$. Let $(u_{red},v_{red})\in L(I)$ for some $(u,v)\in L$. By Definition \ref{L(I)} and Lemma \ref{Lrestrict}, $\leaf(u)\cap T_I=\leaf(u_{red})\cap T_I$ is matched with $\leaf(v)\cap S_{\phi(I)}=\leaf(v_{red})\cap S_{\phi(I)}$ by $\phi|_I$ and $|\leaf(u_{red})\cap T_I|=|\leaf(v_{red})\cap S_{\phi(I)}|>1$.
Also recall from Definition \ref{L(I)} that $u_{red}$ is the minimal $u'\leq_T u$ with $\leaf(u)\cap T_I=\leaf(u')\cap T_I$. Note that this $u_{red}$ cannot be a leaf in $T$ since we know $|\leaf(u)\cap T_I)|=|\leaf(u_{red})\cap T_I|>1$. Now let $T_1$ and $T_2$ be the subtrees of $T$ such that $\rut(T_1)$ and $\rut(T_2)$ are the children of $u_{red}$. Then the definition of $u_{red}$ implies that $\leaf(\rut(T_1))\cap T_I$ and $\leaf(\rut(T_2))\cap T_I$ are both proper, nonempty subsets of $\leaf(u_{red})\cap T_I$, and from this, we conclude that nonempty subtrees $T_1'\subseteq T_1$ and $T_2'\subseteq T_2$ appear in the minimal subtree containing $\{t_i\}_{i\in I}$. This implies that $u_{red}$ must appear in the minimal subtree containing $\{t_i\}_{i\in I}$, as it is the minimal vertex that is an ancestor of the vertices in $T_1'\cup T_2'$. Furthermore, this also implies that $u_{red}$ has a child in $T_1'$ and a child in $T_2'$. Since $u_{red}$ has two children in the minimal subtree containing $\{t_i\}_{i\in I}$, it will not be suppressed when forming $T_I$. We conclude that $u_{red}$ is an internal vertex of $T_I$, and a similar argument shows $v_{red}$ is an internal vertex of $S_{\phi(I)}$. Combined, we see that $u_{red}$ and $v_{red}$ are internal vertices of $(T_I,S_{\phi(I)},\phi|_I)$, and $\leaf(u_{red})\cap T_I$ is matched with $\leaf(v_{red})\cap S_{\phi(I)}$ by $\phi|_I$. Thus, $(u_{red},v_{red})\in L'$, and we conclude that $L(I)\subseteq L'$.

To finish proving (b), we must show that $L(I)\supset L'$. Consider any  leaf-matched pair $(u',v')\in L'$. Observe that $u'=u'_{red}$ and $v'=v'_{red}$ since $u'$ and $v'$ are vertices in $(T_I,S_{\phi(I)},\phi|_I)$. During $\algname{ModifiedUntangle}(T,S,\phi|_I)$, either $u'$ or $v'$ appears first in some partial layout $(X_j,Y_j)$, so we assume without loss of generality that it is $u'\in X_j$. Since $v'$ has not appeared yet in a partial layout, some vertex $y'\geq_S v'$ must be in $Y_j$. Since $\leaf(u')\cap T_I$ is matched with a subset of $\leaf(y')\cap S_{\phi(I)}$ by $\phi|_I$, we know that $y'$ is the unique neighbor of $u'$ in $(X_j,E_j,Y_j)$, and if $\deg(y')\geq 1$ in $(X_j,E_j,Y_j)$, then $y'$ will be refined before we refine $u'$. Eventually, we will obtain some partial layout $(X_k,Y_k)$ where some $w'\geq_S v'$ will appear in $Y_k$ with $\deg(w')=1$ in $(X_k,E_k,Y_k)$, and hence $u'$ and $w'$ are unique neighbors of one another in $(X_k,E_k,Y_k)$. After this occurs, \algname{ModifiedUntangle} will add $(u',w')$ to $L$ before it refines $u'$ or $w'$. Since $(u',v')\in L'$, we know $\leaf(u')\cap T_I$ is matched with $\leaf(v')\cap S_{\phi(I)}$. Since $u'$ and $w'$ had degree one in $(X_k,E_k,Y_k)$, Lemma \ref{Lrestrict} implies $\leaf(u')\cap T_I$ is also matched with $\leaf(w')\cap S_{\phi(I)}$ by $\phi|_I$, and therefore $\leaf(v')\cap S_{\phi(I)}=\leaf(w')\cap S_{\phi(I)}$. Since $v'_{red}=v'$, we conclude that $w'_{red}=v'_{red}=v'$, and thus $(u',v')=(u_{red}',w'_{red})\in L(I)$.  
\end{proof}

\subsection{Preserving subtanglegram planarity and reducing crossings}\label{4-1}

Throughout this subsection, fix $(T,S,\phi)$ as a tanglegram of size $n$, and fix $I=[n]\setminus \{i\}$ for some $i\in [n]$. Let $L(I)$ be the leaf-matched pairs of $(T_I,S_{\phi(I)},\phi|_I)$, and let $(X,Y)$ be a layout of $(T,S,\phi)$ that restricts to a planar layout of $(T_I,S_{\phi(I)},\phi|_I)$. We also use the notation
\begin{equation}
u_0=\text{parent of $t_i\in T$} \quad \text{ and } \quad  v_0=\text{parent of $s_{\phi(i)}\in S$},
\end{equation}
which are also the unique internal vertices in $(T,S,\phi)$ that are not in $(T_I,S_{\phi(I)},\phi|_I)$.
To solve the Tanglegram Single Edge Insertion Problem, we pursue the following questions:
\begin{enumerate}[label=(\arabic*)]
    \item What operations on $(X,Y)$ produce another layout $(X',Y')$ that is also planar when restricted to $(T_I,S_{\phi(I)},\phi|_I)$?
    \item How can we efficiently find the operation(s) that correspond to a solution to the Tanglegram Single Edge Insertion Problem?
\end{enumerate}
In this section, we answer (1) and establish some results that will lead to (2). A complete answer to (2) requires some cases, which we detail in Section \ref{4-2}. Before we answer (1), we start with a definition. While a special case of this definition is sufficient for answering (1), the extra generality will be useful later in Section \ref{MEI}.

\begin{definition}
Let $T$ be a tree, let $u\in T$ be an internal vertex, and let $u_1,u_2$ be the children of $u$. A \textit{subtree switch} (sometimes abbreviated \textit{switch}) at $u$ is the operation on $T$ that interchanges the two subtrees rooted at $u_1$ and $u_2$, while maintaining the relative order of all leaves within each subtree. Note that a subtree switch is equivalent to a flip at $u$ and then a flip at each child of $u$ that is not a leaf.
\end{definition}

\begin{lemma}\label{planarsubtanglegram}
Suppose $(X,Y)$ and $(X',Y')$ are both layouts of $(T,S,\phi)$ that restrict to a planar layout of $(T_I,S_{\phi(I)},\phi|_I)$. Then $(X',Y')$ can be obtained from $(X,Y)$ using a sequence of the following operations:
\begin{itemize}
    \item paired flips at $(u,v)\in L(I)$, and 
    \item subtree switches at $u_0$ and $v_0$.
\end{itemize}
\end{lemma}

\begin{proof}
Note that flips at internal vertices $u\in T$ and $v\in S$ generate all trees isomorphic to $T$ and $S$, and all of these flips commute with one another. Hence, sequences of flips at internal vertices generate all layouts of a tanglegram. From this, we know that some sequence of flips $f_1,\ldots,f_m$ at vertices in $T$ and $S$ maps $(X,Y)$ to $(X',Y')$, and all of these flips commute with one another. We can assume that we do not flip at any vertex twice, as all flips commute and have order two. 

Recall that $u_0$ is the parent of $t_i$ and $v_0$ is the parent of $s_{\phi(i)}$. Let $u'$ be the child of $u_0$ that is not $t_i$, and let $v'$ be the child of $v_0$ that is not $s_{\phi(i)}$. Note that these may or may not be internal vertices. If $u'$ is an internal vertex, define $g$ to be a flip at $u'$. Otherwise, define $g$ as the identity map. Similarly, define $h$ to be a flip at $v'$ or the identity map, respectively corresponding to when $v'$ is an internal vertex or a leaf.

First, suppose that none of the flips $f_1,\ldots,f_m$ are flips at $u_0$ and $v_0$. Then we can restrict the layouts $(X,Y)$ and $(X',Y')$ to $(T_I,S_{\phi(I)},\phi|_I)$ and consider $f_1,\ldots,f_m$ as a sequence of flips in the subtanglegram $(T_I,S_{\phi(I)},\phi|_I)$. Since the restrictions of $(X,Y)$ and $(X',Y')$ are planar layouts of $(T_I,S_{\phi(I)},\phi|_I)$, Theorem \ref{planarlayouts} implies that $f_1,\ldots,f_m$ must be equivalent to a sequence of paired flips at elements in $L(I)$. 

Now suppose that some $f_j$ corresponds to a flip at $u_0$, but no flip at $v_0$ occurs. Since all flips commute, we can assume without loss of generality that this is $f_1$. Then another sequence of flips that maps $(X,Y)$ to $(X',Y')$ is
$$f_1,g,g,f_2,\ldots,f_m,$$
as all flips have order 2. The composition $g\circ f_1$ is a subtree switch at $u_0$ that maps $(X,Y)$ to a layout $(X'',Y'')$ that is also planar when restricted to $(T_I,S_{\phi(I)},\phi|_I)$. The sequence $g,f_2,\ldots,f_m$ maps $(X'',Y'')$ to $(X',Y')$, and none of the flips involve the vertices $u_0$ and $v_0$. By the preceding paragraph, this sequence must be equivalent to a sequence of paired flips at elements in $L(I)$. We can use a similar argument when some $f_j$ corresponds to a flip at $v_0$ and no flip at $u_0$ occurs.

Finally, if flips at both $u_0$ and $v_0$ occur, then we assume without loss of generality that these flips are $f_1$ and $f_2$, respectively. We then consider the sequence
$$f_1,g,f_2,h,g,h,f_3,\ldots,f_m,$$
which is equivalent to $f_1,\ldots,f_m$ since all flips commute and have order 2. Using similar reasoning, $g\circ f_1$ and $h\circ f_2$ are subtree switches, and $g,h,f_3,\ldots,f_m$ correspond to a sequence of paired flips at elements in $L(I)$. 
\end{proof}

\begin{lemma}\label{automorphisms}
Let $(u,v)\in L(I)$, and let $(X',Y')$ be the image of $(X,Y)$ after a paired flip at $(u,v)$.
\begin{enumerate}[label=(\alph*)]
    \item If $u\not>_T t_i$ and $v\not>_S s_{\phi(i)}$, then $(X,Y)$ and $(X',Y')$ have the exact same crossings.
    \item If $u>_T t_i$ and $v>_S s_{\phi(i)}$, then $(X,Y)$ and $(X',Y')$ have the exact same crossings.
\end{enumerate}
\end{lemma}

\begin{proof}
For (a), we suppose $u\not>_T t_i$ and $v\not>_S s_{\phi(i)}$. Suppose $(t_i,s_{\phi(i)})$ crosses some but not all of the edges between $\leaf(u)$ and $\leaf (v)$. Then $t_i$ appears in the middle of the leaves in $\leaf(u)$, or $s_{\phi(i)}$ appears in the middle of the leaves in $\leaf(v)$. These respectively contradict the assumptions $u\not>_T t_i$ and $v\not>_S s_{\phi(i)}$, as the leaves in $\leaf(u)$ and $\leaf(v)$ must appear consecutively in any layout. From this, we conclude that $(t_i,s_{\phi(i)})$ either crosses all or none of the edges between $\leaf(u)$ and $\leaf(v)$. Then a paired flip at $(u,v)$ does not affect any crossings. 

For (b), suppose that for the pair $(u,v)$, both $u>_T t_i$ and $v>_S s_{\phi(i)}$. Then all crossings of $(T,S,\phi)$ are contained in the subtanglegram with trees rooted at $u$ and $v$. A paired flip at $(u,v)$ reflects this subtanglegram, preserving all of the crossings. Visualizations of these arguments are shown in Figure \ref{badpairs}.
\end{proof}

\begin{figure}[h]
    \centering
\begin{tikzpicture}[scale=0.3]
\filldraw[fill=blue,draw=blue] (-2,5) circle (3pt);
\filldraw[fill=blue,draw=blue] (-2,1) circle (3pt);
\filldraw[fill=black,draw=black] (-4,3) circle (3pt);
\filldraw[fill=blue,draw=blue] (-7,0) circle (3pt);
\filldraw[fill=blue,draw=blue] (-3,-4) circle (3pt);
\filldraw[fill=blue,draw=blue] (2,5) circle (3pt);
\filldraw[fill=blue,draw=blue] (2,1) circle (3pt);
\filldraw[fill=black,draw=black] (4,3) circle (3pt);
\filldraw[fill=blue,draw=blue] (7,0) circle (3pt);
\filldraw[fill=blue,draw=blue] (3,-4) circle (3pt);
\draw (-1,6) -- (-2,5) -- (-1,4) -- (-1,6);
\draw (-1,2)--(-2,1) -- (-1,0) -- (-1,2);
\draw (-2,5) -- (-4,3) -- (-2,1) -- (-4,3) -- (-7,0) -- (-1,-6) -- (-1,-2) -- (-3,-4);
\draw (1,6) -- (2,5) -- (1,4) -- (1,6);
\draw (1,2)--(2,1) -- (1,0) -- (1,2);
\draw (2,5) -- (4,3) -- (2,1) -- (4,3) -- (7,0) -- (1,-6) -- (1,-2) -- (3,-4);
\draw[dashed,color=black!50!white] (-1,5.5) -- (1,5.5);
\draw[dashed,color=black!50!white] (-1,5) -- (1,5);
\draw[dashed,color=black!50!white] (-1,4.5) -- (1,4.5);
\draw[dashed,color=black!50!white] (-1,1.5) -- (1,1.5);
\draw[dashed,color=black!50!white] (-1,1) -- (1,1);
\draw[dashed,color=black!50!white] (-1,0.5) -- (1,0.5);
\draw[dashed,color=black!50!white] (-1,-5) -- (1,-5);
\draw[dashed,color=black!50!white] (-1,-3) -- (1,-3);
\draw[dashed,color=black!50!white] (-1,-4) -- (1,-4);

\draw[color=red] (-2.5,1.5) -- (-1,3) -- (1,-3.5);
\filldraw[fill=red,draw=red] (-2.5,1.5) circle (3pt);
\filldraw[fill=red,draw=red] (-1,3) circle (3pt);
\filldraw[fill=red,draw=red] (1,-3.5) circle (3pt);

\node[color=red] at (-0.25,3) {\small{$t_i$}};
\node[color=red] at (0.25,-4) {\small{$s_{\phi(i)}$}};
\node[color=blue] at (-2.5,5.5) {\small{$b$}};
\node[color=blue] at (2.5,5.5) {\small{$x$}};
\node[color=blue] at (-2.5,0.5) {\small{$c$}};
\node[color=blue] at (2.5,0.5) {\small{$y$}};
\node[color=blue] at (-7.5,.5) {\small{$a$}};
\node[color=blue] at (7.5,.5) {\small{$w$}};
\node[color=blue] at (-3.5,-4.5) {\small{$d$}};
\node[color=blue] at (3.5,-4.5) {\small{$z$}};
\node[color=red] at (-3.25,1.25) {\small{$u_0$}};

\end{tikzpicture}
    \caption{Lemma \ref{automorphisms} shows that paired flips at $(a,w),(v,x),(c,y)$ do not change any crossings. Notice that a paired flip at $(d,z)$ affects crossings involving the edges between $\leaf(d)$ and $\leaf(z)$.}
    \label{badpairs}
\end{figure}
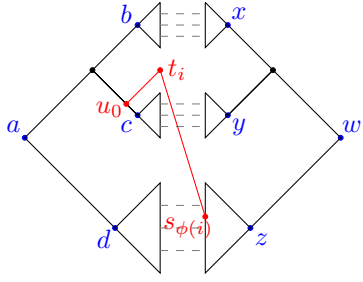


The preceding lemmas suggest that we should focus on subtree switches at $u_0$ and $v_0$, as well as the leaf-matched pairs $(u,v)\in L(I)$ where exactly one of $u>_Tt_i$ or $v>_Ss_{\phi(i)}$ is true. Motivated by these observations, we define the following sets:
\begin{equation}
\begin{split}\label{maximums}
L(I)_T=\{(u,v)\in L(I):u>_T t_i,v\not>_S s_{\phi(i)}\},\\
L(I)_S=\{(u,v)\in L(I):u\not>_T t_i,v>_S s_{\phi(i)}\}.
\end{split}
\end{equation}
The two preceding lemmas then imply the following result concerning solutions to the Single Edge Insertion Problem.

\begin{corollary}\label{existence}
A solution to the Tanglegram Single Edge Insertion Problem can be obtained by starting at $(X,Y)$ and performing a sequence of the following operations:
\begin{itemize}
    \item paired flips at $(u,v)\in L(I)_T\cup L(I)_S$, and
    \item subtree switches at $u_0$ and $v_0$.
\end{itemize}
\end{corollary}

\begin{proof}
Let $(X_{min},Y_{min})$ correspond to a solution of the Tanglegram Single Edge Insertion Problem. By Lemma \ref{planarsubtanglegram}, we can obtain $(X_{min},Y_{min})$ from $(X,Y)$ using a sequence $f_1,f_2,\ldots,f_m$ of paired flips at $(u,v)\in L(I)$ and subtree switches at $u_0$ and $v_0$. Some of these $f_j$ may correspond to paired flips at $(u,v)\in L(I)$ where both or neither of $u>_T t_i$, $v>_S s_{\phi(i)}$ are true. We can assume these paired flips correspond to $f_{k+1},\ldots ,f_m$ for some $0\leq k \leq m$, as all flips commuting implies that paired flips and subtree switches also commute. By Lemma \ref{automorphisms} with induction, if we only perform $f_1,f_2,\ldots,f_k$, then we obtain another layout $(X_{min}',Y_{min}')$ that has the exact same crossings as $(X_{min},Y_{min})$. Thus, $(X_{min}',Y_{min}')$ also solves the Tanglegram Single Edge Insertion Problem, and the result follows. 
\end{proof}

From Corollary \ref{existence}, we see that solving the Tanglegram Single Edge Insertion Problem reduces to finding specific sequences of paired flips at $(u,v)\in L(I)_T\cup L(I)_S$ and subtree switches at $u_0,v_0$. Identifying the correct sequences will require some cases. Before we consider these cases, note that the sets $\{u\in T:(u,v)\in L(I)_T \text{ for some $v\in S$}\}$ and $\{v\in S:(u,v)\in L(I)_S \text{ for some $u\in T$}\}$ are linearly ordered since they form subsets of the ancestors of $t_i$ and $s_{\phi(i)}$, respectively. In particular, this implies that $u>_Tu_0$ for all $(u,v)\in L(I)_T$, and $v>_Sv_0$ for all $(u,v)\in L(I)_S$ since $u_0$ and $v_0$ are respectively the parents of $t_i$ and $s_{\phi(i)}$. Additionally, whenever $L(I)_T$ and $L(I)_S$ are nonempty, each set has a unique ``maximal" element, which we denote
\begin{equation}
\begin{split}
(u_{Tmax},v_{Tmax})=\text{unique $(u,v)\in L(I)_T$ such that $u\geq_T u'$ for all $(u',v')\in L(I)_T$},\\
(u_{Smax},v_{Smax})=\text{unique $(u,v)\in L(I)_S$ such that $v\geq_S v'$ for all $(u',v')\in L(I)_S$}.
\end{split}
\end{equation}
Some examples are shown in Figure \ref{insertioncases}. 
We conclude this section with one final lemma describing the implications of ancestry relations between $u_0$ and $u_{Smax}$, or $v_0$ and $v_{Tmax}$.

\begin{figure}[h]
    \centering
    \begin{tikzpicture}[scale=0.55]
\filldraw[fill=black,draw=black] (1,4) circle (2pt);
\filldraw[fill=black,draw=black] (1,3) circle (2pt);
\filldraw[fill=black,draw=black] (1,2) circle (2pt);
\filldraw[fill=black,draw=black] (1,1) circle (2pt);
\filldraw[fill=black,draw=black] (1,0) circle (2pt);
\filldraw[fill=black,draw=black] (1,-1) circle (2pt);
\filldraw[fill=black,draw=black] (1,-2) circle (2pt);
\filldraw[fill=black,draw=black] (1,-3) circle (2pt);
\filldraw[fill=black,draw=black] (-1,4) circle (2pt);
\filldraw[fill=black,draw=black] (-1,3) circle (2pt);
\filldraw[fill=black,draw=black] (-1,2) circle (2pt);
\filldraw[fill=black,draw=black] (-1,1) circle (2pt);
\filldraw[fill=black,draw=black] (-1,0) circle (2pt);
\filldraw[fill=black,draw=black] (-1,-1) circle (2pt);
\filldraw[fill=black,draw=black] (-1,-2) circle (2pt);
\filldraw[fill=black,draw=black] (-1,-3) circle (2pt);
\draw[dashed] (-1,4) -- (1,4);
\draw[dashed] (-1,3) -- (1,3);
\draw[dashed] (-1,2) -- (1,2);
\draw[dashed] (-1,1) -- (1,1);
\draw[dashed] (-1,0) -- (1,0);
\draw[dashed] (-1,-1) -- (1,-1);
\draw[dashed] (-1,-2) -- (1,-2);
\draw[dashed] (-1,-3) -- (1,-3);
\draw (-1,4) -- (-1.5,3.5) -- (-1,3) -- (-1.5,3.5) -- (-2.5,2.5) -- (-1.5,1.5) -- (-1,2) -- (-1.5,1.5) -- (-1,1) -- (-2.5,2.5) -- (-4.5,0.5);
\draw (-1,-3) -- (-1.5,-2.5) -- (-1,-2) -- (-1.5,-2.5) -- (-2,-2) -- (-1,-1) -- (-2,-2) -- (-2.5,-1.5) -- (-1,0) -- (-2.5,-1.5) -- (-4.5,0.5);
\draw (1,4) -- (1.5,3.5) -- (1,3) -- (1.5,3.5) -- (2,3) -- (1,2) -- (2,3) -- (2.5,2.5) -- (1,1) -- (2.5,2.5) -- (4.5,0.5) ; 
\draw (1,-3) -- (1.5,-2.5) -- (1,-2) -- (1.5,-2.5) -- (4.5,0.5);
\draw (1,0) -- (1.5,-0.5) -- (1,-1) -- (1.5,-0.5) --(3.5,1.5);
\filldraw[fill=black,draw=black] (-1.5,1.5) circle (2pt);
\filldraw[fill=black,draw=black] (-2.5,-1.5) circle (2pt);
\filldraw[fill=black,draw=black] (-2,-2) circle (2pt);
\filldraw[fill=black,draw=black] (2,3) circle (2pt);
\filldraw[fill=black,draw=black] (1.5,-0.5) circle (2pt);
\filldraw[fill=black,draw=black] (3.5,1.5) circle (2pt);
\filldraw[fill=black,draw=black] (-1.5,3.5) circle (2pt);
\filldraw[fill=black,draw=black] (1.5,3.5) circle (2pt);
\filldraw[fill=black,draw=black] (2.5,2.5) circle (2pt);
\filldraw[fill=black,draw=black] (-2.5,2.5) circle (2pt);
\filldraw[fill=black,draw=black] (4.5,0.5) circle (2pt);
\filldraw[fill=black,draw=black] (-4.5,0.5) circle (2pt);
\draw[color=red] (-2.75,-1.25) -- (-1,0.5) -- (1,-2.5) -- (1.25,-2.25);
\filldraw[fill=red,draw=red] (-2.75,-1.25) circle (2pt);
\node[color=red] at (-3.125,-1.625) {$u_0$};
\filldraw[fill=red,draw=red] (-1,0.5) circle (2pt);
\filldraw[fill=red,draw=red] (1,-2.5) circle (2pt);
\filldraw[fill=red,draw=red] (1.25,-2.25) circle (2pt);
\node[color=green!75!black] at (-2.5,-2.75) {$u_{Smax}$};
\filldraw[fill=green!75!black,draw=green!75!black] (1.5,-2.5) circle (2pt);
\filldraw[fill=green!75!black,draw=green!75!black] (-1.5,-2.5) circle (2pt);
\node[color=red] at (-0.5,0.5) {$t_i$};
\node[color=red] at (0.25,-2.5) {$s_{\phi(i)}$};
\end{tikzpicture}
\text{ }
    \begin{tikzpicture}[scale=0.55]
\filldraw[fill=black,draw=black] (1,4) circle (2pt);
\filldraw[fill=black,draw=black] (1,3) circle (2pt);
\filldraw[fill=black,draw=black] (1,2) circle (2pt);
\filldraw[fill=black,draw=black] (1,1) circle (2pt);
\filldraw[fill=black,draw=black] (1,0) circle (2pt);
\filldraw[fill=black,draw=black] (1,-1) circle (2pt);
\filldraw[fill=black,draw=black] (1,-2) circle (2pt);
\filldraw[fill=black,draw=black] (1,-3) circle (2pt);
\filldraw[fill=black,draw=black] (-1,4) circle (2pt);
\filldraw[fill=black,draw=black] (-1,3) circle (2pt);
\filldraw[fill=black,draw=black] (-1,2) circle (2pt);
\filldraw[fill=black,draw=black] (-1,1) circle (2pt);
\filldraw[fill=black,draw=black] (-1,0) circle (2pt);
\filldraw[fill=black,draw=black] (-1,-1) circle (2pt);
\filldraw[fill=black,draw=black] (-1,-2) circle (2pt);
\filldraw[fill=black,draw=black] (-1,-3) circle (2pt);
\draw[dashed] (-1,4) -- (1,4);
\draw[dashed] (-1,3) -- (1,3);
\draw[dashed] (-1,2) -- (1,2);
\draw[dashed] (-1,1) -- (1,1);
\draw[dashed] (-1,0) -- (1,0);
\draw[dashed] (-1,-1) -- (1,-1);
\draw[dashed] (-1,-2) -- (1,-2);
\draw[dashed] (-1,-3) -- (1,-3);
\draw (-1,4) -- (-1.5,3.5) -- (-1,3) -- (-1.5,3.5) -- (-2.5,2.5) -- (-1.5,1.5) -- (-1,2) -- (-1.5,1.5) -- (-1,1) -- (-2.5,2.5) -- (-4.5,0.5);
\draw (-1,-3) -- (-1.5,-2.5) -- (-1,-2) -- (-1.5,-2.5) -- (-2,-2) -- (-1,-1) -- (-2,-2) -- (-2.5,-1.5) -- (-1,0) -- (-2.5,-1.5) -- (-4.5,0.5);
\draw (1,4) -- (1.5,3.5) -- (1,3) -- (1.5,3.5) -- (2,3) -- (1,2) -- (2,3) -- (2.5,2.5) -- (1,1) -- (2.5,2.5) -- (4.5,0.5) ; 
\draw (1,-3) -- (1.5,-2.5) -- (1,-2) -- (1.5,-2.5) -- (4.5,0.5);
\draw (1,0) -- (1.5,-0.5) -- (1,-1) -- (1.5,-0.5) --(3.5,1.5);
\filldraw[fill=black,draw=black] (-1.5,1.5) circle (2pt);
\filldraw[fill=black,draw=black] (-2.5,-1.5) circle (2pt);
\filldraw[fill=black,draw=black] (-2,-2) circle (2pt);
\filldraw[fill=black,draw=black] (2,3) circle (2pt);
\filldraw[fill=black,draw=black] (1.5,-0.5) circle (2pt);
\filldraw[fill=black,draw=black] (3.5,1.5) circle (2pt);
\filldraw[fill=green!75!black,draw=green!75!black] (-1.5,3.5) circle (2pt);
\filldraw[fill=green!75!black,draw=green!75!black] (1.5,3.5) circle (2pt);
\filldraw[fill=green!75!black,draw=green!75!black] (2.5,2.5) circle (2pt);
\filldraw[fill=green!75!black,draw=green!75!black] (-2.5,2.5) circle (2pt);
\filldraw[fill=black,draw=black] (4.5,0.5) circle (2pt);
\filldraw[fill=black,draw=black] (-4.5,0.5) circle (2pt);
\filldraw[fill=black,draw=black] (1.5,-2.5) circle (2pt);
\filldraw[fill=black,draw=black] (-1.5,-2.5) circle (2pt);
\draw[color=red] (-1.25,3.25) -- (-1,3.5) -- (1,0.5) -- (2.75,2.25);
\filldraw[fill=red,draw=red] (-1.25,3.25) circle (2pt);
\filldraw[fill=red,draw=red] (-1,3.5) circle (2pt);
\filldraw[fill=red,draw=red] (1,0.5) circle (2pt);
\filldraw[fill=red,draw=red] (2.75,2.25) circle (2pt);
\node[color=red] at (2.75,1.75) {$v_0$};
\node[color=green!75!black] at (3.5,2.625) {$v_{Tmax}$};
\node[color=red] at (-0.5,3.5) {$t_i$};
\node[color=red] at (0.25,0.5) {$s_{\phi(i)}$};
\end{tikzpicture}
\text{ }
    \begin{tikzpicture}[scale=0.55]
\filldraw[fill=black,draw=black] (1,4) circle (2pt);
\filldraw[fill=black,draw=black] (1,3) circle (2pt);
\filldraw[fill=black,draw=black] (1,2) circle (2pt);
\filldraw[fill=black,draw=black] (1,1) circle (2pt);
\filldraw[fill=black,draw=black] (1,0) circle (2pt);
\filldraw[fill=black,draw=black] (1,-1) circle (2pt);
\filldraw[fill=black,draw=black] (1,-2) circle (2pt);
\filldraw[fill=black,draw=black] (1,-3) circle (2pt);
\filldraw[fill=black,draw=black] (-1,4) circle (2pt);
\filldraw[fill=black,draw=black] (-1,3) circle (2pt);
\filldraw[fill=black,draw=black] (-1,2) circle (2pt);
\filldraw[fill=black,draw=black] (-1,1) circle (2pt);
\filldraw[fill=black,draw=black] (-1,0) circle (2pt);
\filldraw[fill=black,draw=black] (-1,-1) circle (2pt);
\filldraw[fill=black,draw=black] (-1,-2) circle (2pt);
\filldraw[fill=black,draw=black] (-1,-3) circle (2pt);
\draw[dashed] (-1,4) -- (1,4);
\draw[dashed] (-1,3) -- (1,3);
\draw[dashed] (-1,2) -- (1,2);
\draw[dashed] (-1,1) -- (1,1);
\draw[dashed] (-1,0) -- (1,0);
\draw[dashed] (-1,-1) -- (1,-1);
\draw[dashed] (-1,-2) -- (1,-2);
\draw[dashed] (-1,-3) -- (1,-3);
\draw (-1,4) -- (-1.5,3.5) -- (-1,3) -- (-1.5,3.5) -- (-2.5,2.5) -- (-1.5,1.5) -- (-1,2) -- (-1.5,1.5) -- (-1,1) -- (-2.5,2.5) -- (-4.5,0.5);
\draw (-1,-3) -- (-1.5,-2.5) -- (-1,-2) -- (-1.5,-2.5) -- (-2,-2) -- (-1,-1) -- (-2,-2) -- (-2.5,-1.5) -- (-1,0) -- (-2.5,-1.5) -- (-4.5,0.5);
\draw (1,4) -- (1.5,3.5) -- (1,3) -- (1.5,3.5) -- (2,3) -- (1,2) -- (2,3) -- (2.5,2.5) -- (1,1) -- (2.5,2.5) -- (4.5,0.5) ; 
\draw (1,-3) -- (1.5,-2.5) -- (1,-2) -- (1.5,-2.5) -- (4.5,0.5);
\draw (1,0) -- (1.5,-0.5) -- (1,-1) -- (1.5,-0.5) --(3.5,1.5);
\filldraw[fill=black,draw=black] (-1.5,1.5) circle (2pt);
\filldraw[fill=black,draw=black] (-2.5,-1.5) circle (2pt);
\filldraw[fill=black,draw=black] (-2,-2) circle (2pt);
\filldraw[fill=black,draw=black] (2,3) circle (2pt);
\filldraw[fill=black,draw=black] (1.5,-0.5) circle (2pt);
\filldraw[fill=black,draw=black] (3.5,1.5) circle (2pt);
\filldraw[fill=black,draw=black] (-1.5,3.5) circle (2pt);
\filldraw[fill=black,draw=black] (1.5,3.5) circle (2pt);
\filldraw[fill=green!75!black,draw=green!75!black] (2.5,2.5) circle (2pt);
\filldraw[fill=green!75!black,draw=green!75!black] (-2.5,2.5) circle (2pt);
\filldraw[fill=black,draw=black] (4.5,0.5) circle (2pt);
\filldraw[fill=black,draw=black] (-4.5,0.5) circle (2pt);
\filldraw[fill=green!75!black,draw=green!75!black] (1.5,-2.5) circle (2pt);
\filldraw[fill=green!75!black,draw=green!75!black] (-1.5,-2.5) circle (2pt);
\draw[color=red] (-1.75,1.75) -- (-1,2.5) -- (1,-2.5) -- (1.25,-2.25);
\filldraw[fill=red,draw=red] (-1.75,1.75) circle (2pt);
\filldraw[fill=red,draw=red] (-1,2.5) circle (2pt);
\filldraw[fill=red,draw=red] (1,-2.5) circle (2pt);
\filldraw[fill=red,draw=red] (1.25,-2.25) circle (2pt);
\node[color=red] at (-2,1.5) {$u_0$};
\node[color=red] at (1.5,-2) {$v_0$};
\node[color=green!75!black] at (-2.5,-2.75) {$u_{Smax}$};
\node[color=green!75!black] at (3.5,2.75) {$v_{Tmax}$};
\node[color=red] at (-0.5,2.5) {$t_i$};
\node[color=red] at (0.25,-2.5) {$s_{\phi(i)}$};
\end{tikzpicture}
    \caption{Examples of $u_{Smax}$ and $v_{Tmax}$.}
    \label{insertioncases} 
\end{figure}
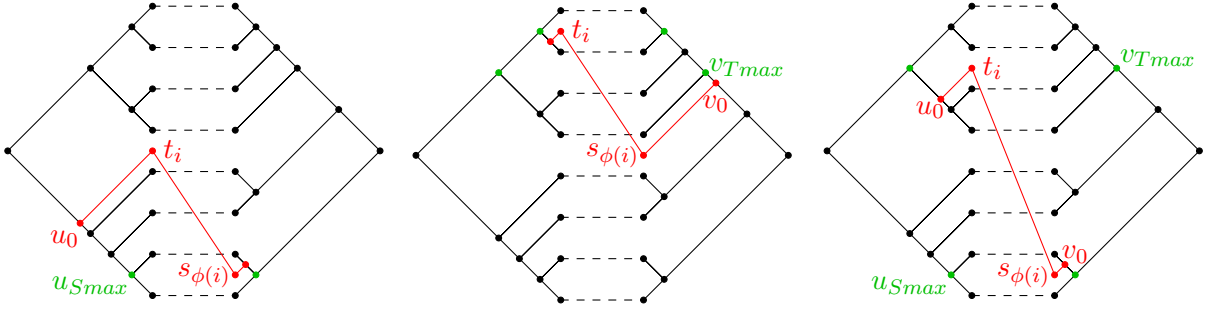

\begin{lemma}\label{case1order}  
The  following properties hold when the appropriate vertex $u_{Smax}$ or $v_{Tmax}$ exists.
\begin{enumerate}[label=(\roman*)]
    \item Suppose $u_0>_T u_{Smax}$, and let $(u,v)\in L(I)$. If $u>_T t_i$, then $v>_S s_{\phi(i)}$. Hence, $L(I)_T=\emptyset$.
    \item Suppose $u_0\not>_T u_{Smax}$. Then $u_0\not>_T u$ for any $(u,v)\in L(I)_S$.
    \item Suppose $v_0>_S v_{Tmax}$, and let $(u,v)\in L(I)$. If $v>_S s_{\phi(i)}$, then $u>_T t_i$. Hence, $L(I)_S=\emptyset$.
    \item Suppose $v_0\not>_S v_{Tmax}$. Then $v_0\not>_S v$ for any $(u,v)\in L(I)_T$.
\end{enumerate}
\end{lemma}
\begin{proof}
For (i), suppose $u_0>_T u_{Smax}$, and consider $(u,v)\in L(I)$ with $u>_T t_i$. Since $u_0$ is the parent of $t_i$, this implies that $u>_T u_0$. Combined with the assumption, we conclude $$u>_T u_0>_T u_{Smax}.$$
Then in the induced subtree $T_I$, we know $u>_{T_I} u_{Smax}$. This implies a corresponding relation for their leaf-matched vertices $v$ and $v_{Smax}$ in the induced subtree $S_{\phi(I)}$, so $v>_{S_{\phi(I)}} v_{Smax}$. Since the ancestry relations on $S_{\phi(I)}$ are restrictions of the relations on $S$, this implies $v>_S v_{Smax}$. We know $v_{Smax}>_S s_{\phi(i)}$ by the definition of $L(I)_S$, so we conclude that $$v>_S v_{Smax}>_S s_{\phi(i)}.$$ 

For (ii), suppose $u_0\not>_T u_{Smax}$. Assume by contradiction that $u_0>_T u$ for some $(u,v)\in L(I)_S$. Then by the definition of $(u_{Smax},v_{Smax})$, we know that $v_{Smax}>_S v$, and therefore $u_{Smax}>_T u$ by the same argument involving $T_I$ and $S_{\phi(I)}$ in the proof of (i) above. We see that $u_{Smax}$ and $u_0$ are both ancestors of $u$, and the assumption $u_0\not>_T u_{Smax}$ then implies $$u_{Smax}>_T u_0>_T u$$
since the ancestors of $u$ are linearly ordered. However, $u_0$ is the parent of $t_i$, so this implies $u_{Smax}>_T t_i$, contradicting $(u_{Smax},v_{Smax})\in L(I)_S$. Results (iii) and (iv) follow by similar arguments.
\end{proof}

\subsection{Insertion Algorithm}\label{4-2}

In the previous section, we established that a solution to the Single Edge Insertion Problem only requires us to consider $L(I)_T,L(I)_S,u_0$, and $v_0$. Now we wish to efficiently find a sequence of paired flips and subtree switches that solves the Tanglegram Single Edge Insertion Problem. Our approach for finding such a sequence will be different depending on where $u_0$ and $v_0$ are inserted. An example of each of the three cases we will consider was previously shown in Figure \ref{insertioncases}. In this subsection, we solve the three different cases, and then we will combine our results to form the \algname{Insertion} algorithm that solves the Tanglegram Single Edge Insertion Problem.

We start with the case $u_0>_T u_{Smax}$. Lemma \ref{case1order} implies $L(I)_T$ is empty, so we can focus our attention on $v_0$, $u_0$, and $L(I)_S=\{(u_j,v_j)\}_{j=1}^m$, where elements have been indexed so that $v_1<_S v_2<_S \ldots  <_S v_m$. Below we give an algorithm for this first case of single edge insertion. Notice that we do not consider a flip or switch at any vertex until we have considered all ancestors of that vertex, and we make these flip and switch choices based on edges in a set we call $E(u_0),E(v_0)$, and $E(v_j)$ for $j=1,2,\ldots,m$. These sets track the crossings that can be affected by an operation at that vertex and cannot be affected by operations at descendants of that vertex. An example of the algorithm is shown in Figure \ref{insertioncase1ex}.

\begin{algorithm}[H]\DontPrintSemicolon
\caption{\algname{Insertion Case $u_0>_T u_{Smax}$}}\label{case1algorithm}
\KwIn{tanglegram $(T,S,\phi)$, index $i$ such that $(T_I,S_{\phi(I)},\phi|_I)$ is planar for $I=[n]\setminus \{i\}$}
\KwOut{layout of $(T,S,\phi)$ that restricts to a planar layout of $(T_I,S_{\phi(I)},\phi|_I)$}
\tcp*{Step 1: initialize the algorithm.}
 $(X,Y),L\coloneqq\algname{ModifiedUntangle}(T,S,\phi|_I)$\\
construct $L(I)$ from $L$ using Definition \ref{Lsubtanglegram}\\
$u_0\coloneqq$ parent of $t_i$, $v_0\coloneqq$ parent of $s_{\phi(i)}$\\
$L(I)_S\coloneqq\{(u,v)\in L(I):u\not>_T t_i,v>_S s_{\phi(i)}\}$\\
\tcp*{Step 2: construct edge sets.}
linearly order $L(I)_S=\{(u_j,v_j)\}_{j=1}^m$ so that $v_1<_S v_2<_S \ldots <_S v_m$\\
$E(u_0) \coloneqq$ between-tree edges with an endpoint in $\leaf(u_0)\setminus \leaf(u_m)$\\
$E(v_0) \coloneqq$ between-tree edges with an endpoint in $\leaf(v_0)\setminus \{s_{\phi(i)}\}$\\
\For{$j=1,2,\ldots ,m$,}
    {
    $E(v_j) \coloneqq$ between-tree edges with an endpoint in $\leaf(v_j)\setminus \leaf(v_{j-1})$
    }
\tcp*{Step 3: use paired flips and subtree switches to change crossings.}
\If{$(t_i,s_{\phi(i)})$ crosses more than half of the edges in $E(u_0)$ in the layout $(X,Y)$, \label{u0start}} 
    {
    update $X \coloneqq \algname{SubtreeSwitch}(X,u_0)$ \label{u0end}
    }
\For{$j=m,\ldots ,2,1$,} 
    {
    \If{$(t_i,s_{\phi(i)})$ crosses more than half of the edges in $E(v_j)$ in the layout $(X,Y)$,\label{vjstart}}  
        {
        update $(X,Y)\coloneqq \algname{PairedFlip}((X,Y),(u_j,v_j))$ \label{vjend}
        }
    }
\If{$(t_i,s_{\phi(i)})$ crosses more than half of the edges in $E(v_0)$ in the layout $(X,Y)$,}
    {
    update $Y\coloneqq \algname{SubtreeSwitch}(Y,v_0)$
    }
\Return $(X,Y)$
\end{algorithm}

\begin{figure}[h]
\begin{center}
    \begin{tikzpicture}[scale=0.55]
\filldraw[fill=black,draw=black] (1,4) circle (2pt);
\filldraw[fill=black,draw=black] (1,3) circle (2pt);
\filldraw[fill=black,draw=black] (1,2) circle (2pt);
\filldraw[fill=black,draw=black] (1,1) circle (2pt);
\filldraw[fill=black,draw=black] (1,0) circle (2pt);
\filldraw[fill=black,draw=black] (1,-1) circle (2pt);
\filldraw[fill=black,draw=black] (1,-2) circle (2pt);
\filldraw[fill=black,draw=black] (1,-3) circle (2pt);
\filldraw[fill=black,draw=black] (-1,4) circle (2pt);
\filldraw[fill=black,draw=black] (-1,3) circle (2pt);
\filldraw[fill=black,draw=black] (-1,2) circle (2pt);
\filldraw[fill=black,draw=black] (-1,1) circle (2pt);
\filldraw[fill=black,draw=black] (-1,0) circle (2pt);
\filldraw[fill=black,draw=black] (-1,-1) circle (2pt);
\filldraw[fill=black,draw=black] (-1,-2) circle (2pt);
\filldraw[fill=black,draw=black] (-1,-3) circle (2pt);
\draw[dashed] (-1,4) -- (1,4);
\draw[dashed] (-1,3) -- (1,3);
\draw[dashed] (-1,2) -- (1,2);
\draw[dashed] (-1,1) -- (1,1);
\draw[dashed] (-1,0) -- (1,0);
\draw[dashed] (-1,-1) -- (1,-1);
\draw[dashed] (-1,-2) -- (1,-2);
\draw[dashed] (-1,-3) -- (1,-3);
\draw (-1,4) -- (-1.5,3.5)-- (-1,3) -- (-1.5,3.5) -- (-4.5,0.5) -- (-3.5,-0.5) -- (-1,2) -- (-1.5,1.5) -- (-1,1) -- (-1.5,1.5) -- (-3.5,-0.5) -- (-2.5,-1.5) -- (-1,0) -- (-2,-1) -- (-1.5,-1.5) -- (-1,-1) -- (-1.5,-1.5) -- (-1,-2) -- (-2,-1) -- (-2.5,-1.5) -- (-1,-3);
\draw (1,4) -- (1.5,3.5) -- (1,3) -- (1.5,3.5) -- (2,3) -- (1,2) -- (2,3) -- (2.5,2.5) -- (1,1) -- (2.5,2.5) -- (4.5,0.5) -- (2.5,-1.5) -- (1,0) -- (2.5,-1.5) -- (1,-3) -- (2,-2) -- (1,-1) -- (1.5,-1.5) -- (1,-2);
\filldraw[fill=black,draw=black] (-1.5,1.5) circle (2pt);
\filldraw[fill=black,draw=black] (-2.5,-1.5) circle (2pt);
\filldraw[fill=black,draw=black] (-2,-1) circle (2pt);
\filldraw[fill=black,draw=black] (2,3) circle (2pt);
\filldraw[fill=black,draw=black] (2,-2) circle (2pt);
\filldraw[fill=black,draw=black] (2.5,-1.5) circle (2pt);
\filldraw[fill=black,draw=black] (-1.5,3.5) circle (2pt);
\filldraw[fill=black,draw=black] (1.5,3.5) circle (2pt);
\filldraw[fill=black,draw=black] (2.5,2.5) circle (2pt);
\filldraw[fill=black,draw=black] (-2,-1) circle (2pt);
\filldraw[fill=black,draw=black] (4.5,0.5) circle (2pt);
\filldraw[fill=black,draw=black] (-4.5,0.5) circle (2pt);
\filldraw[fill=black,draw=black] (-3.5,-0.5) circle (2pt);
\draw[color=red] (-3.75,-0.25) -- (-1,2.5) -- (1,-0.5) -- (1.25,-1.25);
\filldraw[fill=red,draw=red] (-3.75,-0.25) circle (2pt);
\filldraw[fill=red,draw=red] (-1,2.5) circle (2pt);
\filldraw[fill=red,draw=red] (1,-0.5) circle (2pt);
\filldraw[fill=red,draw=red] (1.25,-1.25) circle (2pt);
\node[color=red] at (-0.5,2.625) {\small{$t_i$}};
\node[color=red] at (0.25,-0.5) {\small{$s_{\phi(i)}$}};
\filldraw[fill=black,draw=black] (1.5,-1.5) circle (2pt);
\filldraw[fill=black,draw=black] (-1.5,-1.5) circle (2pt);
\node[color=white] at (-0.5,-3.625) {\small{$t_i$}};
\node[color=red] at (-4,-0.5) {\small{$u_0$}};
\node[color=red] at (1.5,-1) {\small{$v_0$}};
\node at (-1.875,-1.5) {\small{$u_1$}};
\node at (1.875,-1.5) {\small{$v_1$}};
\node at (-3,-1.5) {\small{$u_2$}};
\node at (3,-1.5) {\small{$v_2$}};
\draw (1.2,2.2) -- (-1.2,2.2) -- (-1.2,0.8) -- (1.2,0.8) -- (1.2,2.2);
\node at (0,1.5) {\small{$E(u_0)$}};
\end{tikzpicture}
\phantom{-}
\begin{tikzpicture}[scale=0.55]
\filldraw[fill=black,draw=black] (1,4) circle (2pt);
\filldraw[fill=black,draw=black] (1,3) circle (2pt);
\filldraw[fill=black,draw=black] (1,2) circle (2pt);
\filldraw[fill=black,draw=black] (1,1) circle (2pt);
\filldraw[fill=black,draw=black] (1,0) circle (2pt);
\filldraw[fill=black,draw=black] (1,-1) circle (2pt);
\filldraw[fill=black,draw=black] (1,-2) circle (2pt);
\filldraw[fill=black,draw=black] (1,-3) circle (2pt);
\filldraw[fill=black,draw=black] (-1,4) circle (2pt);
\filldraw[fill=black,draw=black] (-1,3) circle (2pt);
\filldraw[fill=black,draw=black] (-1,2) circle (2pt);
\filldraw[fill=black,draw=black] (-1,1) circle (2pt);
\filldraw[fill=black,draw=black] (-1,0) circle (2pt);
\filldraw[fill=black,draw=black] (-1,-1) circle (2pt);
\filldraw[fill=black,draw=black] (-1,-2) circle (2pt);
\filldraw[fill=black,draw=black] (-1,-3) circle (2pt);
\draw[dashed] (-1,4) -- (1,4);
\draw[dashed] (-1,3) -- (1,3);
\draw[dashed] (-1,2) -- (1,2);
\draw[dashed] (-1,1) -- (1,1);
\draw[dashed] (-1,0) -- (1,0);
\draw[dashed] (-1,-1) -- (1,-1);
\draw[dashed] (-1,-2) -- (1,-2);
\draw[dashed] (-1,-3) -- (1,-3);
\draw (-1,4) -- (-1.5,3.5)-- (-1,3) -- (-1.5,3.5) -- (-4.5,0.5) -- (-3.5,-0.5) -- (-1,2) -- (-1.5,1.5) -- (-1,1) -- (-1.5,1.5) -- (-3.5,-0.5) -- (-2.5,-1.5) -- (-1,0) -- (-2,-1) -- (-1.5,-1.5) -- (-1,-1) -- (-1.5,-1.5) -- (-1,-2) -- (-2,-1) -- (-2.5,-1.5) -- (-1,-3);
\draw (1,4) -- (1.5,3.5) -- (1,3) -- (1.5,3.5) -- (2,3) -- (1,2) -- (2,3) -- (2.5,2.5) -- (1,1) -- (2.5,2.5) -- (4.5,0.5) -- (2.5,-1.5) -- (1,0) -- (2.5,-1.5) -- (1,-3) -- (2,-2) -- (1,-1) -- (1.5,-1.5) -- (1,-2);
\filldraw[fill=black,draw=black] (-1.5,1.5) circle (2pt);
\filldraw[fill=black,draw=black] (-2.5,-1.5) circle (2pt);
\filldraw[fill=black,draw=black] (-2,-1) circle (2pt);
\filldraw[fill=black,draw=black] (2,3) circle (2pt);
\filldraw[fill=black,draw=black] (2,-2) circle (2pt);
\filldraw[fill=black,draw=black] (2.5,-1.5) circle (2pt);
\filldraw[fill=black,draw=black] (-1.5,3.5) circle (2pt);
\filldraw[fill=black,draw=black] (1.5,3.5) circle (2pt);
\filldraw[fill=black,draw=black] (2.5,2.5) circle (2pt);
\filldraw[fill=black,draw=black] (-2,-1) circle (2pt);
\filldraw[fill=black,draw=black] (4.5,0.5) circle (2pt);
\filldraw[fill=black,draw=black] (-4.5,0.5) circle (2pt);
\filldraw[fill=black,draw=black] (-3.5,-0.5) circle (2pt);
\draw[color=red] (-3.75,-0.25) -- (-1,-3.5) -- (1,-0.5) -- (1.25,-1.25);
\filldraw[fill=red,draw=red] (-3.75,-0.25) circle (2pt);
\filldraw[fill=red,draw=red] (-1,-3.5) circle (2pt);
\filldraw[fill=red,draw=red] (1,-0.5) circle (2pt);
\filldraw[fill=red,draw=red] (1.25,-1.25) circle (2pt);
\node[color=red]  at (-0.5,-3.625) {\small{$t_i$}};
\node[color=red] at (0.125,-0.625) {\small{$s_{\phi(i)}$}};
\filldraw[fill=black,draw=black] (1.5,-1.5) circle (2pt);
\filldraw[fill=black,draw=black] (-1.5,-1.5) circle (2pt);
\node[color=red] at (-4,-0.5) {\small{$u_0$}};
\node[color=red] at (1.5,-1) {\small{$v_0$}};
\node at (-1.875,-1.5) {\small{$u_1$}};
\node at (1.875,-1.5) {\small{$v_1$}};
\node at (-3,-1.5) {\small{$u_2$}};
\node at (3,-1.5) {\small{$v_2$}};
\node at (0,0.5) {\small{$E(v_2)$}};
\draw (1.2,0.2) -- (-1.2,0.2) -- (-1.2,-0.2) -- (1.2,-0.2) -- (1.2,0.2);
\draw (1.2,-2.8) -- (-1.2,-2.8) -- (-1.2,-3.2) -- (1.2,-3.2) -- (1.2,-2.8);
\end{tikzpicture}
\phantom{-}
\begin{tikzpicture}[scale=0.55]
\filldraw[fill=black,draw=black] (1,4) circle (2pt);
\filldraw[fill=black,draw=black] (1,3) circle (2pt);
\filldraw[fill=black,draw=black] (1,2) circle (2pt);
\filldraw[fill=black,draw=black] (1,1) circle (2pt);
\filldraw[fill=black,draw=black] (1,0) circle (2pt);
\filldraw[fill=black,draw=black] (1,-1) circle (2pt);
\filldraw[fill=black,draw=black] (1,-2) circle (2pt);
\filldraw[fill=black,draw=black] (1,-3) circle (2pt);
\filldraw[fill=black,draw=black] (-1,4) circle (2pt);
\filldraw[fill=black,draw=black] (-1,3) circle (2pt);
\filldraw[fill=black,draw=black] (-1,2) circle (2pt);
\filldraw[fill=black,draw=black] (-1,1) circle (2pt);
\filldraw[fill=black,draw=black] (-1,0) circle (2pt);
\filldraw[fill=black,draw=black] (-1,-1) circle (2pt);
\filldraw[fill=black,draw=black] (-1,-2) circle (2pt);
\filldraw[fill=black,draw=black] (-1,-3) circle (2pt);
\draw[dashed] (-1,4) -- (1,4);
\draw[dashed] (-1,3) -- (1,3);
\draw[dashed] (-1,2) -- (1,2);
\draw[dashed] (-1,1) -- (1,1);
\draw[dashed] (-1,0) -- (1,0);
\draw[dashed] (-1,-1) -- (1,-1);
\draw[dashed] (-1,-2) -- (1,-2);
\draw[dashed] (-1,-3) -- (1,-3);
\draw (-1,4) -- (-1.5,3.5)-- (-1,3) -- (-1.5,3.5) -- (-4.5,0.5) -- (-3.5,-0.5) -- (-1,2) -- (-1.5,1.5) -- (-1,1) -- (-1.5,1.5) -- (-3.5,-0.5) -- (-2.5,-1.5) -- (-1,0) -- (-2,-1) -- (-1.5,-1.5) -- (-1,-1) -- (-1.5,-1.5) -- (-1,-2) -- (-2,-1) -- (-2.5,-1.5) -- (-1,-3);
\draw (1,4) -- (1.5,3.5) -- (1,3) -- (1.5,3.5) -- (2,3) -- (1,2) -- (2,3) -- (2.5,2.5) -- (1,1) -- (2.5,2.5) -- (4.5,0.5) -- (2.5,-1.5) -- (1,0) -- (2.5,-1.5) -- (1,-3) -- (2,-2) -- (1,-1) -- (1.5,-1.5) -- (1,-2);
\filldraw[fill=black,draw=black] (-1.5,1.5) circle (2pt);
\filldraw[fill=black,draw=black] (-2.5,-1.5) circle (2pt);
\filldraw[fill=black,draw=black] (-2,-1) circle (2pt);
\filldraw[fill=black,draw=black] (2,3) circle (2pt);
\filldraw[fill=black,draw=black] (2,-2) circle (2pt);
\filldraw[fill=black,draw=black] (2.5,-1.5) circle (2pt);
\filldraw[fill=black,draw=black] (-1.5,3.5) circle (2pt);
\filldraw[fill=black,draw=black] (1.5,3.5) circle (2pt);
\filldraw[fill=black,draw=black] (2.5,2.5) circle (2pt);
\filldraw[fill=black,draw=black] (-2,-1) circle (2pt);
\filldraw[fill=black,draw=black] (4.5,0.5) circle (2pt);
\filldraw[fill=black,draw=black] (-4.5,0.5) circle (2pt);
\filldraw[fill=black,draw=black] (-3.5,-0.5) circle (2pt);
\draw[color=red] (-3.75,-0.25) -- (-1,-3.5) -- (1,-2.5) -- (1.25,-1.75);
\filldraw[fill=red,draw=red] (-3.75,-0.25) circle (2pt);
\filldraw[fill=red,draw=red] (-1,-3.5) circle (2pt);
\filldraw[fill=red,draw=red] (1,-2.5) circle (2pt);
\filldraw[fill=red,draw=red] (1.25,-1.75) circle (2pt);
\node[color=red] at (-0.5,-3.625) {\small{$t_i$}};
\node[color=red] at (0.25,-2.375) {\small{$s_{\phi(i)}$}};
\filldraw[fill=black,draw=black] (1.5,-1.5) circle (2pt);
\filldraw[fill=black,draw=black] (-1.5,-1.5) circle (2pt);
\node[color=red] at (-4,-0.5) {\small{$u_0$}};
\node[color=red] at (1.5,-2.125) {\small{$v_0$}};
\node at (-1.875,-1.5) {\small{$u_1$}};
\node at (1.875,-1.5) {\small{$v_1$}};
\node at (-3,-1.5) {\small{$u_2$}};
\node at (3,-1.5) {\small{$v_2$}};
\end{tikzpicture}
\caption{If {\normalfont{\algname{ModifiedUntangle}}$(T,S,\phi|_I)$} returns the layout on the left, then Algorithm \ref{case1algorithm} would first perform a subtree switch at $u_0$ to obtain the layout in the middle. Then it would not perform a paired flip at $(u_2,v_2)$, would perform a paired flip at $(u_1,v_1)$, and would not perform a subtree switch at $v_0$, returning the layout on the right with one crossing.}
\label{insertioncase1ex}
\end{center}
\end{figure}
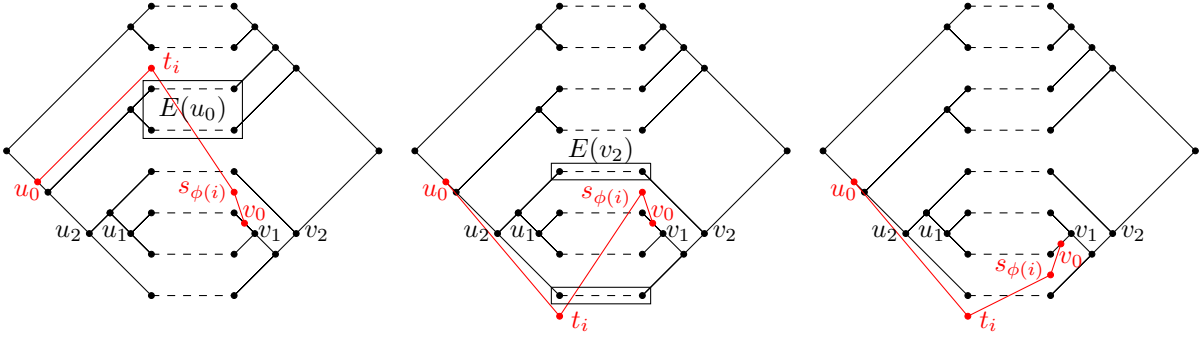

\begin{lemma}\label{insertioncase1}
Algorithm \ref{case1algorithm} solves the Tanglegram Single Edge Insertion Problem when $L(I)_S\neq \emptyset$ and $u_0>_T u_{Smax}$.
\end{lemma}
\begin{proof}
By Lemma \ref{modified}, the initial layout $(X,Y)$ of $(T,S,\phi)$ in Step 1 of Algorithm \ref{case1algorithm} restricts to a planar layout of $(T_I,S_{\phi(I)},\phi|_I)$, and $L(I)$ is the set of leaf-matched pairs of $(T_I,S_{\phi(I)},\phi|_I)$. The assumption $u_0>_T u_{Smax}$ with Lemma \ref{case1order} implies $L(I)_T=\emptyset$. Corollary \ref{existence} with $L(I)_T=\emptyset$ implies that starting with $(X,Y)$, some sequence of paired flips at $(u_j,v_j)\in L(I)_S$ and subtree switches at $u_0$ and $v_0$ solves the Tanglegram Single Edge Insertion Problem. 
The algorithm makes choices in the order $u_0,(u_m,v_m),\ldots ,(u_1,v_1),v_0$, so we will show that our choice at each step can extend to a solution to the  Single Edge Insertion Problem, and thus the algorithm terminates at a solution.

The algorithm starts with $u_0$. Some operations at $u_0,(u_j,v_j)\in L(I)_S$, and $v_0$ produces a solution $(X_{min},Y_{min})$ to the Single Edge Insertion Problem. This layout contains sublists for the elements in $\leaf(u_0)$ and the leaves paired with them, which have the form either $(t_iX_1UX_2,Y_1VY_2)$ or $(X_1UX_2t_i,Y_1VY_2)$, where $U,V$ is an ordering of $\leaf(u_m),\leaf(v_m)$ and $X_1,X_2,Y_1,Y_2$ order the remaining leaves. We focus on the case $(t_iX_1UX_2,Y_1VY_2)$ illustrated in Figure \ref{case1drawing1}, as the other case follows by similar reasoning. Notice that $E(u_0)=\op{Edges}(X_1)\cup \op{Edges}(X_2)$, where $\op{Edges}(X_i)$ is the set of between-tree edges with a leaf in $X_i$.

\begin{figure}[h]
    \centering
    \begin{tikzpicture}[scale=.5]
    \draw (-1,7) -- (-2,6) -- (-1,5) -- (-1,7) -- (-4.5,3.5) -- (-3,2) -- (-1,4) -- (-1,0) -- (-4.5,3.5) -- (-6,2) -- (-1,-3) -- (-1,-1) -- (-2,-2);
    \draw[color=blue!50!white] (-2,2) -- (-1,3) -- (-1,1) -- (-2,2);
    \draw (1,7) -- (2,6) -- (1,5) -- (1,7) -- (6,2) -- (4.5,0.5) -- (1,4) -- (1,-3)-- (6,2);
    \draw[color=blue!50!white] (2,2) -- (1,3) -- (1,1) -- (2,2);
    \draw[dashed] (-1,6.5) -- (1,6.5);
    \draw[dashed] (-1,5.5) -- (1,5.5);
    \draw[dashed] (-1,6) -- (1,6);
    \draw[dashed] (-1,-2) -- (1,-2);
    \draw[dashed] (-1,-1.5) -- (1,-1.5);
    \draw[dashed] (-1,-2.5) -- (1,-2.5);
    \draw[color=red] (1,2.25) -- (-1,4.5)-- (-3.25,2.25);
    \draw[color=black!25!white,dashed] (-1,.5) -- (1,.5);
    \draw[color=black!25!white,dashed] (-1,2.5) -- (1,2.5);
    \draw[color=black!25!white,dashed] (-1,1.5) -- (1,1.5);
    \draw[color=black!25!white,dashed] (-1,3.5) -- (1,3.5);
    \draw[color=black!25!white,dashed] (-1,3) -- (1,3);
    \draw[color=black!25!white,dashed] (-1,2) -- (1,2);
    \draw[color=black!25!white,dashed] (-1,1) -- (1,1);
    \filldraw[fill=black,draw=black] (-6,2) circle (2pt);
    \filldraw[fill=black,draw=black] (-4.5,3.5) circle (2pt);
    \filldraw[fill=black,draw=black] (-2,6) circle (2pt);
    \filldraw[fill=black,draw=black] (-3,2) circle (2pt);
    \filldraw[fill=blue,draw=blue] (-2,2) circle (2pt);
    \filldraw[fill=black,draw=black] (-2,-2) circle (2pt);
    \filldraw[fill=black,draw=black] (6,2) circle (2pt);
    \filldraw[fill=black,draw=black] (4.5,0.5) circle (2pt);
    \filldraw[fill=black,draw=black] (2,6) circle (2pt);
    \filldraw[fill=blue,draw=blue] (2,2) circle (2pt);
    \filldraw[fill=red,draw=red] (1,2.25) circle (2pt);
    \filldraw[fill=red,draw=red] (-1,4.5) circle (2pt);
    \filldraw[fill=red,draw=red] (-3.25,2.25) circle (2pt);
    \node[color=blue] at (-2.25,1.75) {\small{$u_{m}$}};
    \node[color=blue] at (2.375,1.75) {\small{$v_{m}$}};
    \node[color=red] at (1.75,2.5) {\small{$s_{\phi(i)}$}};
    \node[color=red] at (-0.375,4.5) {\small{$t_i$}};
    \node at (-.5,3.5) {$X_1$};
    \node at (-.5,2) {$U$};
    \node at (-.5,0.5) {$X_2$};
    \node at (0.5,3.5) {$Y_1$};
    \node at (0.5,2) {$V$};
    \node at (0.5,0.5) {$Y_2$};
    \node[color=red] at (-3.5,2) {\small{$u_0$}};
    \end{tikzpicture}
    \quad 
    \begin{tikzpicture}[scale=.5]
    \draw (-1,7) -- (-2,6) -- (-1,5) -- (-1,7) -- (-4.5,3.5) -- (-3,2) -- (-1,4) -- (-1,0) -- (-4.5,3.5) -- (-6,2) -- (-1,-3) -- (-1,-1) -- (-2,-2);
    \draw[color=blue!50!white] (-2,2) -- (-1,3) -- (-1,1) -- (-2,2);
    \draw (1,7) -- (2,6) -- (1,5) -- (1,7) -- (6,2) -- (4.5,0.5) -- (1,4) -- (1,-3)-- (6,2);
    \draw[color=blue!50!white] (2,2) -- (1,3) -- (1,1) -- (2,2);
    \draw[dashed] (-1,6.5) -- (1,6.5);
    \draw[dashed] (-1,5.5) -- (1,5.5);
    \draw[dashed] (-1,6) -- (1,6);
    \draw[dashed] (-1,-2) -- (1,-2);
    \draw[dashed] (-1,-1.5) -- (1,-1.5);
    \draw[dashed] (-1,-2.5) -- (1,-2.5);
    \draw[color=red] (1,1.75) -- (-1,-0.5)-- (-3.25,2.25);
    \draw[color=black!25!white,dashed] (-1,.5) -- (1,.5);
    \draw[color=black!25!white,dashed] (-1,2.5) -- (1,2.5);
    \draw[color=black!25!white,dashed] (-1,1.5) -- (1,1.5);
    \draw[color=black!25!white,dashed] (-1,3.5) -- (1,3.5);
    \draw[color=black!25!white,dashed] (-1,3) -- (1,3);
    \draw[color=black!25!white,dashed] (-1,2) -- (1,2);
    \draw[color=black!25!white,dashed] (-1,1) -- (1,1);
    \filldraw[fill=black,draw=black] (-6,2) circle (2pt);
    \filldraw[fill=black,draw=black] (-4.5,3.5) circle (2pt);
    \filldraw[fill=black,draw=black] (-2,6) circle (2pt);
    \filldraw[fill=black,draw=black] (-3,2) circle (2pt);
    \filldraw[fill=blue,draw=blue] (-2,2) circle (2pt);
    \filldraw[fill=black,draw=black] (-2,-2) circle (2pt);
    \filldraw[fill=black,draw=black] (6,2) circle (2pt);
    \filldraw[fill=black,draw=black] (4.5,0.5) circle (2pt);
    \filldraw[fill=black,draw=black] (2,6) circle (2pt);
    \filldraw[fill=blue,draw=blue] (2,2) circle (2pt);
    \filldraw[fill=red,draw=red] (1,1.75) circle (2pt);
    \filldraw[fill=red,draw=red] (-1,-0.5) circle (2pt);
    \filldraw[fill=red,draw=red] (-3.25,2.25) circle (2pt);
    \node[color=blue] at (-2.25,1.75) {\small{$u_{m}$}};
    \node[color=blue] at (2.375,1.75) {\small{$v_{m}$}};
    \node[color=red] at (1.75,1.25) {\small{$s_{\phi(i)}$}};
    \node[color=red] at (-0.375,-.5) {\small{$t_i$}};
    \node at (-.5,3.5) {$X_1$};
    \node at (-.5,2) {$\overline{U}$};
    \node at (-.5,0.5) {$X_2$};
    \node at (0.5,3.5) {$Y_1$};
    \node at (0.5,2) {$\overline{V}$};
    \node at (0.5,0.5) {$Y_2$};
    \node[color=red] at (-3.5,2) {\small{$u_0$}};
    \end{tikzpicture}
    \caption{The effect of a subtree switch at $u_0$ and a paired flip at $(u_m,v_m)$.}
    \label{case1drawing1}
\end{figure}
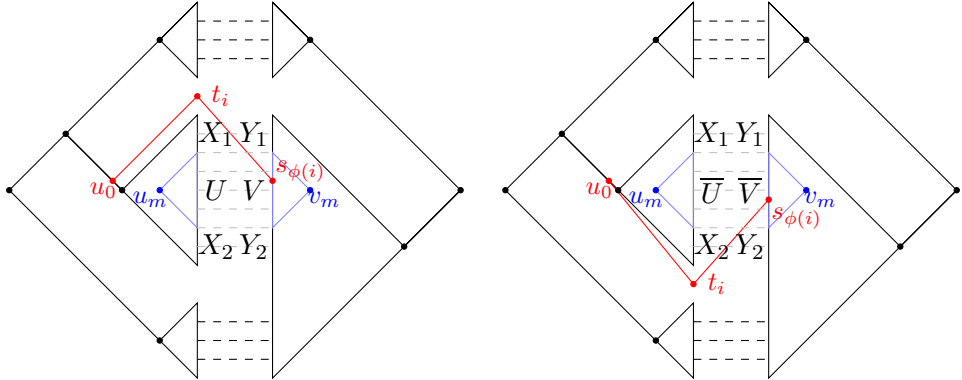

Beginning with $(t_iX_1UX_2,Y_1VY_2)$, if we perform a subtree switch at $u_0$, notice that we can also perform a paired flip at $(u_m,v_m)$ to obtain
$$(t_iX_1UX_2,Y_1VY_2)\xra{\text{subtree switch at $u_0$}}(X_1UX_2t_i,Y_1VY_2)\xra{\text{paired flip at $(u_m,v_m)$}}(X_1\overline{U}X_2t_i,Y_1\overline{V}Y_2),$$
which is also illustrated in  Figure \ref{case1drawing1}. We call this new layout $(X',Y')$. Notice that crossings between $(t_i,s_{\phi(i)})$ and $\op{Edges}(U)$ are the same in $(X_{min},Y_{min})$ and $(X',Y')$, so the choice of subtree switch at $u_0$ does not prevent minimization of crossings in $\op{Edges}(U)$. However, while $(t_i,s_{\phi(i)})$ crosses $\op{Edges}(X_1)$ in $(X_{min},Y_{min})$, it crosses $\op{Edges}(X_2)$ in $(X',Y')$, and these  crossings cannot be affected by choices in $L(I)_S\cup \{v_0\}$. Thus, the choice at $u_0$ that extends to a solution to the Tanglegram Single Edge Insertion Problem must be one that minimizes crossings in $E(u_0)=\op{Edges}(X_1)\cup \op{Edges}(X_2)$, so the algorithm's choice at $u_0$ in lines \ref{u0start}-\ref{u0end} extends to a solution.

Now consider paired flips $(u_j,v_j)\in L(I)_S$. Suppose that we made choices that extend to a solution to the  Single Edge Insertion Problem at $u_0$ and any $(u_{j'},v_{j'})\in L(I)_S$ with $v_{j'}>_S v_j$. Then we know that some choices at $(u_j,v_j),\ldots,(u_1,v_1),$ and $v_0$ lead to a solution to the  Single Edge Insertion Problem. This solution contains an ordering of $\leaf(u_j),\leaf(v_j)$. If $j\geq 2$, denote this ordering as $(X_1UX_2,Y_1VY_2)$ with $U,V$ an ordering of $\leaf(u_{j-1}),\leaf(v_{j-1})$ as shown on the left in Figure \ref{case1drawing2}. If we perform a paired flip at $(u_j,v_j)$, then we can also perform a paired flip at $(u_{j-1},v_{j-1})$ to obtain $$(X_1UX_2,Y_1VY_2)\xra{\text{paired flip at $(u_j,v_j)$}} (\overline{X_2}\,\overline{U}\,\overline{X_1},\overline{Y_2}\,\overline{V}\,\overline{Y_1})\xra{\text{paired flip at $(u_{j-1},v_{j-1})$}}(\overline{X_2}U\overline{X_1},\overline{Y_2}V\overline{Y_1}).$$

\begin{figure}[h]
    \centering
    \begin{tikzpicture}[scale=.5]
    \draw (-1,7) -- (-2,6) -- (-1,5) -- (-1,7) -- (-4.5,3.5) -- (-3,2) -- (-1,4) -- (-1,0) -- (-4.5,3.5) -- (-6,2) -- (-1,-3) -- (-1,-1) -- (-2,-2);
    \draw[color=blue!50!white] (-2,2) -- (-1,3) -- (-1,1) -- (-2,2);
    \draw (1,7) -- (2,6) -- (1,5) -- (1,7) -- (4.5,3.5) -- (3,2) -- (1,4) -- (1,0) -- (4.5,3.5) -- (6,2) -- (1,-3) -- (1,-1) -- (2,-2);
    \draw[color=blue!50!white] (2,2) -- (1,3) -- (1,1) -- (2,2);
    \draw[dashed] (-1,6.5) -- (1,6.5);
    \draw[dashed] (-1,5.5) -- (1,5.5);
    \draw[dashed] (-1,6) -- (1,6);
    \draw[dashed] (-1,-2) -- (1,-2);
    \draw[dashed] (-1,-1.5) -- (1,-1.5);
    \draw[dashed] (-1,-2.5) -- (1,-2.5);
    \draw[color=red] (-4.75,3.25) -- (-1,-0.5) -- (1,2);
    \draw[color=black!25!white,dashed] (-1,.5) -- (1,.5);
    \draw[color=black!25!white,dashed] (-1,2.5) -- (1,2.5);
    \draw[color=black!25!white,dashed] (-1,1.5) -- (1,1.5);
    \draw[color=black!25!white,dashed] (-1,3.5) -- (1,3.5);
    \draw[color=black!25!white,dashed] (-1,3) -- (1,3);
    \draw[color=black!25!white,dashed] (-1,2) -- (1,2);
    \draw[color=black!25!white,dashed] (-1,1) -- (1,1);
    \filldraw[fill=black,draw=black] (-6,2) circle (2pt);
    \filldraw[fill=black,draw=black] (-4.5,3.5) circle (2pt);
    \filldraw[fill=black,draw=black] (-2,6) circle (2pt);
    \filldraw[fill=black,draw=black] (-3,2) circle (2pt);
    \filldraw[fill=blue,draw=blue] (-2,2) circle (2pt);
    \filldraw[fill=black,draw=black] (-2,-2) circle (2pt);
    \filldraw[fill=black,draw=black] (6,2) circle (2pt);
    \filldraw[fill=black,draw=black] (4.5,3.5) circle (2pt);
    \filldraw[fill=black,draw=black] (2,6) circle (2pt);
    \filldraw[fill=black,draw=black] (3,2) circle (2pt);
    \filldraw[fill=blue,draw=blue] (2,2) circle (2pt);
    \filldraw[fill=black,draw=black] (2,-2) circle (2pt);
    \filldraw[fill=red,draw=red] (-1,-0.5) circle (2pt);
    \filldraw[fill=red,draw=red] (1,2) circle (2pt);
    \filldraw[fill=red,draw=red] (-4.75,3.25) circle (2pt);
    \node at (-3,2.5) {\small{$u_j$}};
    \node[color=blue] at (-1.75,1.5) {\small{$u_{j-1}$}};
    \node at (3,2.5) {\small{$v_j$}};
    \node[color=blue] at (2,1.5) {\small{$v_{j-1}$}};
    \node[color=red] at (-0.5,-0.5) {\small{$t_i$}};
    \node[color=red] at (1.75,2.25) {\small{$s_{\phi(i)}$}};
    \node[color=red] at (-5.25,3.5) {\small{$u_0$}};
    \node at (-.5,3.5) {$X_1$};
    \node at (-.5,2) {$U$};
    \node at (-.5,0.5) {$X_2$};
    \node at (0.5,3.5) {$Y_1$};
    \node at (0.5,2) {$V$};
    \node at (0.5,0.5) {$Y_2$};
    \end{tikzpicture}
    \quad 
    \begin{tikzpicture}[scale=.5]
    \draw (-1,7) -- (-2,6) -- (-1,5) -- (-1,7) -- (-4.5,3.5) -- (-3,2) -- (-1,4) -- (-1,0) -- (-4.5,3.5) -- (-6,2) -- (-1,-3) -- (-1,-1) -- (-2,-2);
    \draw[color=blue!50!white] (-2,2) -- (-1,3) -- (-1,1) -- (-2,2);
    \draw (1,7) -- (2,6) -- (1,5) -- (1,7) -- (4.5,3.5) -- (3,2) -- (1,4) -- (1,0) -- (4.5,3.5) -- (6,2) -- (1,-3) -- (1,-1) -- (2,-2);
    \draw[color=blue!50!white] (2,2) -- (1,3) -- (1,1) -- (2,2);
    \draw[dashed] (-1,6.5) -- (1,6.5);
    \draw[dashed] (-1,5.5) -- (1,5.5);
    \draw[dashed] (-1,6) -- (1,6);
    \draw[dashed] (-1,-2) -- (1,-2);
    \draw[dashed] (-1,-1.5) -- (1,-1.5);
    \draw[dashed] (-1,-2.5) -- (1,-2.5);
    \draw[color=red] (-4.75,3.25) -- (-1,-0.5) -- (1,2);
    \draw[color=black!25!white,dashed] (-1,.5) -- (1,.5);
    \draw[color=black!25!white,dashed] (-1,2.5) -- (1,2.5);
    \draw[color=black!25!white,dashed] (-1,1.5) -- (1,1.5);
    \draw[color=black!25!white,dashed] (-1,3.5) -- (1,3.5);
    \draw[color=black!25!white,dashed] (-1,3) -- (1,3);
    \draw[color=black!25!white,dashed] (-1,2) -- (1,2);
    \draw[color=black!25!white,dashed] (-1,1) -- (1,1);
    \filldraw[fill=black,draw=black] (-6,2) circle (2pt);
    \filldraw[fill=black,draw=black] (-4.5,3.5) circle (2pt);
    \filldraw[fill=black,draw=black] (-2,6) circle (2pt);
    \filldraw[fill=black,draw=black] (-3,2) circle (2pt);
    \filldraw[fill=blue,draw=blue] (-2,2) circle (2pt);
    \filldraw[fill=black,draw=black] (-2,-2) circle (2pt);
    \filldraw[fill=black,draw=black] (6,2) circle (2pt);
    \filldraw[fill=black,draw=black] (4.5,3.5) circle (2pt);
    \filldraw[fill=black,draw=black] (2,6) circle (2pt);
    \filldraw[fill=black,draw=black] (3,2) circle (2pt);
    \filldraw[fill=blue,draw=blue] (2,2) circle (2pt);
    \filldraw[fill=black,draw=black] (2,-2) circle (2pt);
    \filldraw[fill=red,draw=red] (-1,-0.5) circle (2pt);
    \filldraw[fill=red,draw=red] (1,2) circle (2pt);
    \filldraw[fill=red,draw=red] (-4.75,3.25) circle (2pt);
    \node at (-3,2.5) {\small{$u_j$}};
    \node[color=blue] at (-1.75,1.5) {\small{$u_{j-1}$}};
    \node at (3,2.5) {\small{$v_j$}};
    \node[color=blue] at (2,1.5) {\small{$v_{j-1}$}};
    \node[color=red] at (-0.5,-0.5) {\small{$t_i$}};
    \node[color=red] at (1.75,2.25) {\small{$s_{\phi(i)}$}};
    \node[color=red] at (-5.25,3.5) {\small{$u_0$}};
    \node at (-.5,3.5) {$\overline{X_2}$};
    \node at (-.5,2) {$U$};
    \node at (-.5,0.5) {$\overline{X_1}$};
    \node at (0.5,3.5) {$\overline{Y_2}$};
    \node at (0.5,2) {$V$};
    \node at (0.5,0.5) {$\overline{Y_1}$};
    \end{tikzpicture}
    \caption{The effect of paired flips at $(u_j,v_j)$ and $(u_{j-1},v_{j-1})$ when $2\leq j\leq m$.}
    \label{case1drawing2}
\end{figure}
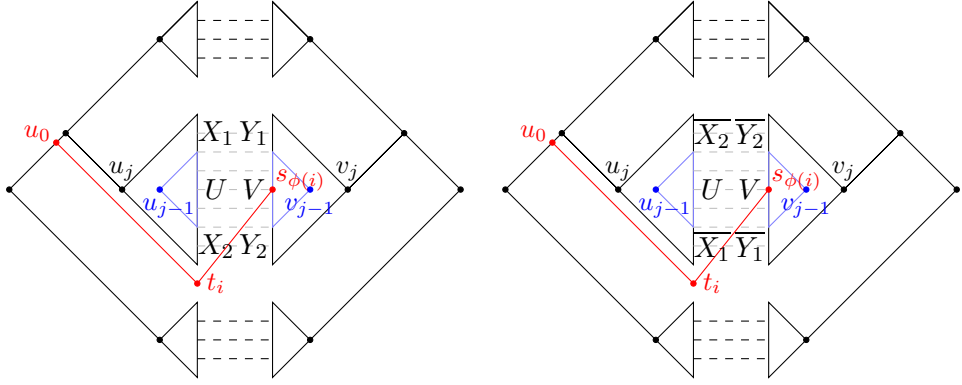

\noindent 
If $j=1$, the ordering of $\leaf(u_1),\leaf(v_1)$ is either $(X_0,Y_1s_{\phi(i)}VY_2)$ or $(X_0,Y_1V s_{\phi(i)}Y_2)$, where $V$ is an ordering of $\leaf(v_0)\setminus \{s_{\phi(i)}\}$. Performing a paired flip at $(u_1,v_1)$ and a subtree switch at $v_0$ respectively results in $(\overline{X}_0,\overline{Y_2}s_{\phi(i)}\overline{V}\,\overline{Y_1})$ or $(\overline{X}_0,\overline{Y_2}\,\overline{V}s_{\phi(i)}\overline{Y_1})$.

In all cases, choosing to perform a paired flip at $(u_j,v_j)$ does not prevent minimization of crossings between $(t_i,s_{\phi(i)})$ and $\op{Edges}(V)$. Thus, the choice at $(u_j,v_j)$ that extends to a solution must be the one that minimizes crossings in $E(v_j)=\op{Edges}(Y_1)\cup \op{Edges}(Y_2)$, as these crossings cannot be affected by operations at $(u_{j-1},v_{j-1}),\ldots,(u_1,v_1)$, and $v_0$. Then the algorithm's choice at $(u_j,v_j)$  in lines \ref{vjstart}-\ref{vjend} extends to a solution. Using induction, we conclude that the sequence of paired flips at elements in $L(I)_S$ extends to a solution.

Finally, for $v_0$, we must minimize crossings involving $(t_i,s_{\phi(i)})$ and between-tree edges with an endpoint in $\leaf(v_0)$. This is precisely what the algorithm does with the set $E(v_0)$. Combined, we conclude that Algorithm \ref{case1algorithm} solves the Tanglegram  Single Edge Insertion Problem when $u_0>_T u_{Smax}$.
\end{proof}

\begin{remark}
When we consider whether or not to perform a subtree switch at $u_0$, it is possible that $|X_1|=|X_2|$, and in this case, both choices at $u_0$ extend to a solution to the Tanglegram  Single Edge Insertion Problem. Similarly, when $|Y_1|=|Y_2|$, both choices of whether or not to 
perform a paired flip at $(u_j,v_j)\in L(I)_S$ extend to a solution. In these situations, we choose not to perform the subtree switch or paired flip for efficiency reasons.
\end{remark}

\begin{remark}\label{insertioncase2}
One can analogously construct an \algname{Insertion Case $v_0>_S v_{Tmax}$} algorithm for the situation $L(I)_T\neq \emptyset$ and $v_0>_S v_{Tmax}$. 
We include the algorithm in the Appendix as Algorithm \ref{case2algorithm}. Proof of its effectiveness follows from a similar argument to the one for Lemma \ref{insertioncase1}. 
\end{remark}

We now consider the remaining cases, where $u_0\not>_T u_{Smax}$ and $v_0\not>_S v_{Tmax}$. 
In these cases, Lemma \ref{case1order} implies $u_0\not>_T u$ for any $(u,v)\in L(I)_S$ and $v_0\not>_S v$ for any $(u,v)\in L(I)_T$. We linearly order $L(I)_S=\{(u_j,v_j)\}_{j=1}^k$ and $L(I)_T=\{(u_j,v_j)\}_{j=k+1}^{k+m}$ so that $v_0<_S v_1<_S v_2 <_S\ldots <_S v_k$ and $u_0<_T u_{k+1}<_T u_{k+2} <_T \ldots <_T u_{k+m}$. Then we define the $E$ sets in a similar way as before:
$$E(v_j)\coloneqq \text{ between-tree edges with an endpoint in }\leaf(v_j)\setminus \leaf(v_{j-1}) \quad \text{ for $j=1,2,\ldots ,k$,}$$
$$E(u_j)\coloneqq \text{ between-tree  edges with an endpoint in }\leaf(u_j)\setminus \leaf(u_{j-1}) \quad \text{ for $j=k+2,k+3,\ldots ,k+m$,}$$
$$E(u_{k+1})\coloneqq \text{ between-tree  edges with an endpoint in }\leaf(u_{k+1})\setminus \leaf(u_{0}),$$
$$E(v_0)\coloneqq \text{ between-tree  edges with an endpoint in }\leaf(v_0)\setminus \{s_{\phi(i)}\},$$
$$E(u_0)\coloneqq \text{ between-tree  edges with an endpoint in }\leaf(u_0)\setminus \{t_i\}.$$
The next lemma gives us a key property concerning these sets. We wish to use the $E$ sets to minimize crossings, and this lemma will show that we do not need to worry about these sets intersecting in most cases.

\begin{lemma}\label{case3order}
Suppose $u_0\not>_T u_{Smax}$ and $v_0\not>_S v_{Tmax}$. 
Then $E(u_j)\cap E(v_\ell)\neq \emptyset$ can only occur when $j=\ell=0$.
\end{lemma}
\begin{proof}
Suppose $E(u_j)\cap E(v_\ell)\neq \emptyset$ for some $j,\ell\neq 0$, that is, there exists $(u_j,v_j)\in L(I)_T$, $(u_{\ell},v_\ell)\in L(I)_S$, and a between-tree edge $(t,s)$ such that $t\in \leaf(u_{j})\setminus \leaf(u_{j-1})$ and $s\in \leaf(v_\ell)\setminus \leaf(v_{\ell-1})$. The assumption $t\in \leaf(u_{j})\setminus \leaf(u_{j-1})$ implies $u_j>_T t$. Since $(u_j,v_j)\in L(I)_T$, the definition of $L(I)_T$ implies that $u_j>_T t_i$ and $v_j \not>_S s_{\phi(i)}$. Furthermore, $(u_j,v_j)$ is a leaf-matched pair of $(T_I,S_{\phi(I)},\phi|_I)$, so $u_j >_T t$ implies $v_{j}>_S s$. Combined, we see that $u_j>_T t$, $u_{j}>_T t_i$, $v_{j}>_S s$, and $v_{j}\not>_S s_{\phi(i)}$.

Using similar reasoning, the assumption $s\in \leaf(v_\ell)\setminus \leaf(v_{\ell-1})$ implies $v_\ell>_S s$, and $(u_\ell,v_\ell)\in L(I)_S$ implies $v_\ell>_S s_{\phi(i)}$.
Since the ancestors of $s$ are linearly ordered, the fact that $v_j$ is an ancestor of both $s$ and $s_{\phi(i)}$ while $v_j$ is only an ancestor of $s$ implies $v_\ell>_S v_j$, giving us the situation illustrated in Figure \ref{disjoint}. This implies $u_\ell>_T u_{j}>_T t_i$, which  contradicts $(u_\ell,v_\ell)\in L(I)_S$.

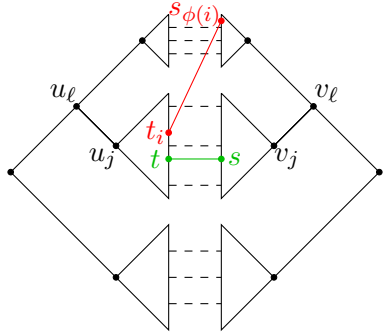
\begin{figure}[h]
    \centering
\begin{tikzpicture}[scale=0.35]
\filldraw[fill=black,draw=black] (-2,5) circle (3pt);
\filldraw[fill=black,draw=black] (-3,1) circle (3pt);
\filldraw[fill=black,draw=black] (-4.5,2.5) circle (3pt);
\filldraw[fill=black,draw=black] (-7,0) circle (3pt);
\filldraw[fill=black,draw=black] (-3,-4) circle (3pt);
\filldraw[fill=black,draw=black] (2,5) circle (3pt);
\filldraw[fill=black,draw=black] (3,1) circle (3pt);
\filldraw[fill=black,draw=black] (4.5,2.5) circle (3pt);
\filldraw[fill=black,draw=black] (7,0) circle (3pt);
\filldraw[fill=black,draw=black] (3,-4) circle (3pt);
\draw (-1,6) -- (-2,5) -- (-1,4) -- (-1,6);
\draw (-1,3)--(-3,1) -- (-1,-1) -- (-1,3);
\draw (-2,5) -- (-4.5,2.5) -- (-3,1) -- (-4.5,2.5) -- (-7,0) -- (-1,-6) -- (-1,-2) -- (-3,-4);
\draw (1,6) -- (2,5) -- (1,4) -- (1,6);
\draw (1,3)--(3,1) -- (1,-1) -- (1,3);
\draw (2,5) -- (4.5,2.5) -- (3,1) -- (4.5,2.5) -- (7,0) -- (1,-6) -- (1,-2) -- (3,-4);
\draw[dashed] (-1,5.5) -- (1,5.5);
\draw[dashed] (-1,5) -- (1,5);
\draw[dashed] (-1,4.5) -- (1,4.5);
\draw[dashed] (-1,2.5) -- (1,2.5);
\draw[dashed] (-1,1) -- (1,1);
\draw[dashed] (-1,-0.5) -- (1,-0.5);
\draw[dashed] (-1,-5) -- (1,-5);
\draw[dashed] (-1,-3) -- (1,-3);
\draw[dashed] (-1,-4) -- (1,-4);

\draw[color=red] (-1,1.5) -- (1,5.75)  ;
\filldraw[fill=red,draw=red] (-1,1.5) circle (3pt);
\filldraw[fill=red,draw=red] (1,5.75) circle (3pt);
\draw[color=green!75!black] (-1,0.5) -- (1,0.5);
\filldraw[fill=green!75!black,draw=green!75!black] (-1,0.5) circle (3pt);
\filldraw[fill=green!75!black,draw=green!75!black] (1,0.5) circle (3pt);
\node at (-5,3) {$u_\ell$};
\node at (5,3) {$v_\ell$};
\node at (-3.5,.5) {$u_j$};
\node at (3.5,.5) {$v_j$};
\node[color=red] at (-1.5,1.5) {$t_i$};
\node[color=red] at (0,6) {$s_{\phi(i)}$};
\node[color=green!75!black] at (-1.5,0.5) {$t$};
\node[color=green!75!black] at (1.5,0.5) {$s$};
\end{tikzpicture}
    \caption{The proof of Lemma \ref{case3order} (i).}
    \label{disjoint}
\end{figure}

Now suppose that $E(u_0)\cap E(v_j)\neq \emptyset$ for some $j\neq 0$, that is,
there exists some $(u_j,v_j)\in L(I)_S$ and some between-tree edge $(t,s)$ with $t\in \leaf(u_0)\setminus \{t_i\}$ and $s\in \leaf(v_j)\setminus \leaf(v_{j-1})$. Since $t\in \leaf(u_0)\setminus \{t_i\}$, we see that $u_0>_T t$. Using the fact that $(u_j,v_j)$ is a leaf-matched pair of $(T_I,S_{\phi(I)},\phi|_I)$, $v_j>_S s$ implies $u_j>_T t$. Since ancestors of $t$ are linearly ordered, either $u_j>_T u_0$ or $u_j<_T u_0$. Using our assumption $u_0\not>_T u_{Smax}$ with Lemma \ref{case1order}, we conclude that $u_{j}>_T u_0>_T t_i$, which contradicts $(u_j,v_j)\in L(I)_S$. The case $E(u_j)\cap E(v_0)\neq \emptyset$ for $j\neq 0$ is ruled out by similar reasoning.
\end{proof}

We now define an algorithm to solve the remaining cases of the Tanglegram Single Edge Insertion Problem. 
As before, we consider flips at all ancestors of a vertex before considering the vertex itself, and make flip and switch choices according to crossings in the $E(u_j)$ and $E(v_j)$ sets. 
An example of the algorithm is shown in Figure \ref{insertion-ex}.

\begin{algorithm}[H]\DontPrintSemicolon
\caption{\algname{Insertion Case $u_0\not>_T u_{Smax}$ and $v_0\not>_S v_{Tmax}$}}\label{case3algorithm}
\KwIn{tanglegram $(T,S,\phi)$, index $i$ such that $(T_I,S_{\phi(I)},\phi|_I)$ is planar for $I=[n]\setminus \{i\}$}
\KwOut{layout of $(T,S,\phi)$ that restricts to a planar layout of $(T_I,S_{\phi(I)},\phi|_I)$}
\tcp*{Step 1: initialize the algorithm.}
$(X,Y),L\coloneqq\algname{ModifiedUntangle}(T,S,\phi|_I)$\\
construct $L(I)$ from $L$ using Definition \ref{Lsubtanglegram}\\
$u_0\coloneqq$ parent of $t_i$, $v_0\coloneqq$ parent of $s_{\phi(i)}$\\
$L(I)_S\coloneqq\{(u,v)\in L(I):u\not>_T t_i,v>_S s_{\phi(i)}\}$\\
$L(I)_T\coloneqq\{(u,v)\in L(I):u>_T t_i,v\not>_S s_{\phi(i)}\}$\\
\tcp*{Step 2: construct edge sets.}
linearly order $L(I)_S=\{(u_j,v_j)\}_{j=1}^{k}$ so that $v_1<_S v_2<_S \ldots <_S v_k$\\
linearly order $L(I)_T=\{(u_j,v_j)\}_{j=k+1}^{k+m}$ so that $u_{k+1}<_T u_{k+2}<_T \ldots <_T u_{k+m}$ \\
\For{$j=1,2,\ldots ,k$,}
    {$E(v_j) \coloneqq$ between-tree edges with an endpoint in $\leaf(v_j)\setminus \leaf(v_{j-1})$}
\For{$j=k+2,k+3,\ldots ,m$,} 
    {$E(u_j) \coloneqq$ between-tree edges with an endpoint in $\leaf(u_j)\setminus \leaf(u_{j-1})$}
$E(u_{k+1})\coloneqq \text{ between-tree edges with an endpoint in $\leaf(u_{k+1})\setminus \leaf(u_0)$}$\\
$E(v_0)\coloneqq \text{ between-tree  edges with an endpoint in }\leaf(v_0)\setminus \{s_{\phi(i)}\}$\\
$E(u_0)\coloneqq \text{ between-tree  edges with an endpoint in }\leaf(u_0)\setminus \{t_i\}$\\
\tcp*{Step 3: use paired flips in $L(I)_T$ and $L(I)_S$ to change crossings.}
\For{$j=k+m,\ldots ,k+2,k+1$,}
    {
    \If{$(t_i,s_{\phi(i)})$ crosses more than half of the edges in $E(u_j)$ in the layout $(X,Y)$,\label{ujstart}}
        {update $(X,Y)\coloneqq \algname{PairedFlip}((X,Y),(u_j,v_j))$ \label{ujend}}
    }
\For{$j=k,\ldots ,2,1$,}
    {
    \If{$(t_i,s_{\phi(i)})$ crosses more than half of the edges in $E(v_j)$ in the layout $(X,Y)$,}
        {update $(X,Y)\coloneqq\algname{PairedFlip}((X,Y),(u_j,v_j))$}
    }
\tcp*{Step 4: use subtree switches at $u_0$ and $v_0$ to change crossings.}
\If{$E(u_0)\cap E(v_0)=\emptyset$,}
    {
    \If{$(t_i,s_{\phi(i)})$ crosses the edges in $E(u_0)$,}
        {update $X\coloneqq \algname{SubtreeSwitch}(X,u_0)$}
    \If{$(t_i,s_{\phi(i)})$ crosses the edges in $E(v_0)$,}
        {update $Y\coloneqq \algname{SubtreeSwitch}(Y,v_0)$}
        \Return $X,Y$\\
    }
\Else
    {
    $X',Y'\coloneqq \algname{SubtreeSwitch}(X,u_0),\algname{SubtreeSwitch}(Y,v_0)$\label{tryallstart}\\
    $(X_{min},Y_{min})\coloneqq $ layout in $\{(X,Y),(X',Y),(X,Y'),(X',Y')\}$ with fewest crossings\\
    \Return $X_{min},Y_{min}$ \label{tryallend}
    }
\end{algorithm}

\begin{figure}[h]
\begin{center}
\begin{tikzpicture}[scale=0.55]
\filldraw[fill=black,draw=black] (1,4) circle (2pt);
\filldraw[fill=black,draw=black] (1,3) circle (2pt);
\filldraw[fill=black,draw=black] (1,2) circle (2pt);
\filldraw[fill=black,draw=black] (1,1) circle (2pt);
\filldraw[fill=black,draw=black] (1,0) circle (2pt);
\filldraw[fill=black,draw=black] (1,-1) circle (2pt);
\filldraw[fill=black,draw=black] (1,-2) circle (2pt);
\filldraw[fill=black,draw=black] (1,-3) circle (2pt);
\filldraw[fill=black,draw=black] (-1,4) circle (2pt);
\filldraw[fill=black,draw=black] (-1,3) circle (2pt);
\filldraw[fill=black,draw=black] (-1,2) circle (2pt);
\filldraw[fill=black,draw=black] (-1,1) circle (2pt);
\filldraw[fill=black,draw=black] (-1,0) circle (2pt);
\filldraw[fill=black,draw=black] (-1,-1) circle (2pt);
\filldraw[fill=black,draw=black] (-1,-2) circle (2pt);
\filldraw[fill=black,draw=black] (-1,-3) circle (2pt);
\draw[dashed] (-1,4) -- (1,4);
\draw[dashed] (-1,3) -- (1,3);
\draw[dashed] (-1,2) -- (1,2);
\draw[dashed] (-1,1) -- (1,1);
\draw[dashed] (-1,0) -- (1,0);
\draw[dashed] (-1,-1) -- (1,-1);
\draw[dashed] (-1,-2) -- (1,-2);
\draw[dashed] (-1,-3) -- (1,-3);
\draw (-1,4) -- (-1.5,3.5) -- (-1,3) -- (-1.5,3.5) -- (-2.5,2.5) -- (-1.5,1.5) -- (-1,2) -- (-1.5,1.5) -- (-1,1) -- (-2.5,2.5) -- (-4.5,0.5);
\draw (-1,-3) -- (-1.5,-2.5) -- (-1,-2) -- (-1.5,-2.5) -- (-2,-2) -- (-1,-1) -- (-2,-2) -- (-2.5,-1.5) -- (-1,0) -- (-2.5,-1.5) -- (-4.5,0.5);
\draw (1,4) -- (1.5,3.5) -- (1,3) -- (1.5,3.5) -- (2,3) -- (1,2) -- (2,3) -- (2.5,2.5) -- (1,1) -- (2.5,2.5) -- (4.5,0.5) ; 
\draw (1,-3) -- (1.5,-2.5) -- (1,-2) -- (1.5,-2.5) -- (4.5,0.5);
\draw (1,0) -- (1.5,-0.5) -- (1,-1) -- (1.5,-0.5) --(3.5,1.5);
\filldraw[fill=black,draw=black] (-1.5,1.5) circle (2pt);
\filldraw[fill=black,draw=black] (-2.5,-1.5) circle (2pt);
\filldraw[fill=black,draw=black] (-2,-2) circle (2pt);
\filldraw[fill=black,draw=black] (2,3) circle (2pt);
\filldraw[fill=black,draw=black] (1.5,-0.5) circle (2pt);
\filldraw[fill=black,draw=black] (3.5,1.5) circle (2pt);
\filldraw[fill=black,draw=black] (-1.5,3.5) circle (2pt);
\filldraw[fill=black,draw=black] (1.5,3.5) circle (2pt);
\filldraw[fill=black,draw=black] (2.5,2.5) circle (2pt);
\node at (-2.875,2.75) {$u_3$};
\node at (2.875,2.75) {$v_3$};
\node at (-2,3.5) {$u_2$};
\node at (2,3.5) {$v_2$};
\node at (-1.875,-2.75) {$u_1$};
\node at (1.875,-2.75) {$v_1$};
\node[color=red] at (-1.625,4) {$u_0$};
\node[color=red] at (1.625,-2) {$v_0$};
\filldraw[fill=black,draw=black] (-2.5,2.5) circle (2pt);
\filldraw[fill=black,draw=black] (4.5,0.5) circle (2pt);
\filldraw[fill=black,draw=black] (-4.5,0.5) circle (2pt);
\draw[color=red] (-1.25,3.75) -- (-1,3.5) -- (1,-2.5) -- (1.25,-2.25);
\filldraw[fill=red,draw=red] (-1.25,3.75) circle (2pt);
\filldraw[fill=red,draw=red] (-1,3.5) circle (2pt);
\filldraw[fill=red,draw=red] (1,-2.5) circle (2pt);
\filldraw[fill=red,draw=red] (1.25,-2.25) circle (2pt);
\node[color=red] at (-.625,3.5) {$t_i$};
\node[color=red] at (.375,-2.5) {$s_{\phi(i)}$};
\filldraw[fill=black,draw=black] (1.5,-2.5) circle (2pt);
\filldraw[fill=black,draw=black] (-1.5,-2.5) circle (2pt);
\draw (-1.2,2.2)--(1.2,2.2) -- (1.2,0.8) -- (-1.2,0.8) -- (-1.2,2.2);
\node at (0,1.5) {\small{$E(u_3)$}};
\node[color=white] at (0,-3.5) {\small{$E(v_1)$}};
\end{tikzpicture}
\phantom{-}
\begin{tikzpicture}[scale=0.55]
\filldraw[fill=black,draw=black] (1,4) circle (2pt);
\filldraw[fill=black,draw=black] (1,3) circle (2pt);
\filldraw[fill=black,draw=black] (1,2) circle (2pt);
\filldraw[fill=black,draw=black] (1,1) circle (2pt);
\filldraw[fill=black,draw=black] (1,0) circle (2pt);
\filldraw[fill=black,draw=black] (1,-1) circle (2pt);
\filldraw[fill=black,draw=black] (1,-2) circle (2pt);
\filldraw[fill=black,draw=black] (1,-3) circle (2pt);
\filldraw[fill=black,draw=black] (-1,4) circle (2pt);
\filldraw[fill=black,draw=black] (-1,3) circle (2pt);
\filldraw[fill=black,draw=black] (-1,2) circle (2pt);
\filldraw[fill=black,draw=black] (-1,1) circle (2pt);
\filldraw[fill=black,draw=black] (-1,0) circle (2pt);
\filldraw[fill=black,draw=black] (-1,-1) circle (2pt);
\filldraw[fill=black,draw=black] (-1,-2) circle (2pt);
\filldraw[fill=black,draw=black] (-1,-3) circle (2pt);
\draw[dashed] (-1,4) -- (1,4);
\draw[dashed] (-1,3) -- (1,3);
\draw[dashed] (-1,2) -- (1,2);
\draw[dashed] (-1,1) -- (1,1);
\draw[dashed] (-1,0) -- (1,0);
\draw[dashed] (-1,-1) -- (1,-1);
\draw[dashed] (-1,-2) -- (1,-2);
\draw[dashed] (-1,-3) -- (1,-3);
\draw (-1,4) -- (-1.5,3.5) -- (-1,3) -- (-1.5,3.5) -- (-2.5,2.5) -- (-1.5,1.5) -- (-1,2) -- (-1.5,1.5) -- (-1,1) -- (-2.5,2.5) -- (-4.5,0.5);
\draw (-1,-3) -- (-1.5,-2.5) -- (-1,-2) -- (-1.5,-2.5) -- (-2,-2) -- (-1,-1) -- (-2,-2) -- (-2.5,-1.5) -- (-1,0) -- (-2.5,-1.5) -- (-4.5,0.5);
\draw (1,1) -- (1.5,1.5) -- (1,2) -- (1.5,1.5) -- (2,2) -- (1,3) -- (2,2) -- (2.5,2.5) -- (1,4)  -- (4.5,0.5) ; 
\draw (1,-3) -- (1.5,-2.5) -- (1,-2) -- (1.5,-2.5) -- (4.5,0.5);
\draw (1,0) -- (1.5,-0.5) -- (1,-1) -- (1.5,-0.5) --(3.5,1.5);
\filldraw[fill=black,draw=black] (-1.5,1.5) circle (2pt);
\filldraw[fill=black,draw=black] (-2.5,-1.5) circle (2pt);
\filldraw[fill=black,draw=black] (-2,-2) circle (2pt);
\filldraw[fill=black,draw=black] (2,2) circle (2pt);
\filldraw[fill=black,draw=black] (1.5,-0.5) circle (2pt);
\filldraw[fill=black,draw=black] (3.5,1.5) circle (2pt);
\filldraw[fill=black,draw=black] (-1.5,3.5) circle (2pt);
\filldraw[fill=black,draw=black] (1.5,1.5) circle (2pt);
\filldraw[fill=black,draw=black] (2.5,2.5) circle (2pt);
\filldraw[fill=black,draw=black] (-2.5,2.5) circle (2pt);
\filldraw[fill=black,draw=black] (4.5,0.5) circle (2pt);
\filldraw[fill=black,draw=black] (-4.5,0.5) circle (2pt);
\filldraw[fill=black,draw=black] (1.5,-2.5) circle (2pt);
\filldraw[fill=black,draw=black] (-1.5,-2.5) circle (2pt);
\node at (-2.875,2.75) {$u_3$};
\node at (2.875,2.75) {$v_3$};
\node at (-2,1.5) {$u_2$};
\node at (2,1.5) {$v_2$};
\node at (-1.875,-2.75) {$u_1$};
\node at (1.75,-2.75) {$v_1$};
\node[color=red] at (-1.5,1) {$u_0$};
\node[color=red] at (1.5,-2) {$v_0$};
\draw[color=red] (-1.25,1.25) -- (-1,1.5) -- (1,-2.5) -- (1.25,-2.25);
\filldraw[fill=red,draw=red] (-1.25,1.25) circle (2pt);
\filldraw[fill=red,draw=red] (-1,1.5) circle (2pt);
\filldraw[fill=red,draw=red] (1,-2.5) circle (2pt);
\filldraw[fill=red,draw=red] (1.25,-2.25) circle (2pt);
\node[color=red] at (-.625,1.5) {$t_i$};
\node[color=red] at (.25,-2.5) {$s_{\phi(i)}$};
\draw (-1.2,2.2) -- (1.2,2.2) -- (1.2,1.8) -- (-1.2,1.8) -- (-1.2,2.2);
\node at (0,2.5) {\small{$E(u_2)$}};
\draw (-1.2,-3.2) -- (1.2,-3.2) -- (1.2,-2.8) -- (-1.2,-2.8) -- (-1.2,-3.2);
\node at (0,-3.5) {\small{$E(v_1)$}};
\end{tikzpicture}
\phantom{-}
\begin{tikzpicture}[scale=0.55]
\filldraw[fill=black,draw=black] (1,4) circle (2pt);
\filldraw[fill=black,draw=black] (1,3) circle (2pt);
\filldraw[fill=black,draw=black] (1,2) circle (2pt);
\filldraw[fill=black,draw=black] (1,1) circle (2pt);
\filldraw[fill=black,draw=black] (1,0) circle (2pt);
\filldraw[fill=black,draw=black] (1,-1) circle (2pt);
\filldraw[fill=black,draw=black] (1,-2) circle (2pt);
\filldraw[fill=black,draw=black] (1,-3) circle (2pt);
\filldraw[fill=black,draw=black] (-1,4) circle (2pt);
\filldraw[fill=black,draw=black] (-1,3) circle (2pt);
\filldraw[fill=black,draw=black] (-1,2) circle (2pt);
\filldraw[fill=black,draw=black] (-1,1) circle (2pt);
\filldraw[fill=black,draw=black] (-1,0) circle (2pt);
\filldraw[fill=black,draw=black] (-1,-1) circle (2pt);
\filldraw[fill=black,draw=black] (-1,-2) circle (2pt);
\filldraw[fill=black,draw=black] (-1,-3) circle (2pt);
\draw[dashed] (-1,4) -- (1,4);
\draw[dashed] (-1,3) -- (1,3);
\draw[dashed] (-1,2) -- (1,2);
\draw[dashed] (-1,1) -- (1,1);
\draw[dashed] (-1,0) -- (1,0);
\draw[dashed] (-1,-1) -- (1,-1);
\draw[dashed] (-1,-2) -- (1,-2);
\draw[dashed] (-1,-3) -- (1,-3);
\draw (-1,4) -- (-1.5,3.5) -- (-1,3) -- (-1.5,3.5) -- (-2.5,2.5) -- (-1.5,1.5) -- (-1,2) -- (-1.5,1.5) -- (-1,1) -- (-2.5,2.5) -- (-4.5,0.5);
\draw (-1,-3) -- (-1.5,-2.5) -- (-1,-2) -- (-1.5,-2.5) -- (-2,-2) -- (-1,-1) -- (-2,-2) -- (-2.5,-1.5) -- (-1,0) -- (-2.5,-1.5) -- (-4.5,0.5);
\draw (1,1) -- (1.5,1.5) -- (1,2) -- (1.5,1.5) -- (2,2) -- (1,3) -- (2,2) -- (2.5,2.5) -- (1,4)  -- (4.5,0.5) ; 
\draw (1,-3) -- (1.5,-2.5) -- (1,-2) -- (1.5,-2.5) -- (4.5,0.5);
\draw (1,0) -- (1.5,-0.5) -- (1,-1) -- (1.5,-0.5) --(3.5,1.5);
\filldraw[fill=black,draw=black] (-1.5,1.5) circle (2pt);
\filldraw[fill=black,draw=black] (-2.5,-1.5) circle (2pt);
\filldraw[fill=black,draw=black] (-2,-2) circle (2pt);
\filldraw[fill=black,draw=black] (2,2) circle (2pt);
\filldraw[fill=black,draw=black] (1.5,-0.5) circle (2pt);
\filldraw[fill=black,draw=black] (3.5,1.5) circle (2pt);
\filldraw[fill=black,draw=black] (-1.5,3.5) circle (2pt);
\filldraw[fill=black,draw=black] (1.5,1.5) circle (2pt);
\filldraw[fill=black,draw=black] (2.5,2.5) circle (2pt);
\filldraw[fill=black,draw=black] (-2.5,2.5) circle (2pt);
\filldraw[fill=black,draw=black] (4.5,0.5) circle (2pt);
\filldraw[fill=black,draw=black] (-4.5,0.5) circle (2pt);
\filldraw[fill=black,draw=black] (1.5,-2.5) circle (2pt);
\filldraw[fill=black,draw=black] (-1.5,-2.5) circle (2pt);
\node at (-2.875,2.75) {$u_3$};
\node at (2.875,2.75) {$v_3$};
\node at (-2,1.5) {$u_2$};
\node at (2,1.5) {$v_2$};
\node at (-1.875,-2.75) {$u_1$};
\node at (1.75,-2.75) {$v_1$};
\node[color=red] at (-1.5,1) {$u_0$};
\node[color=red] at (1.5,-2) {$v_0$};
\draw[color=red] (-1.25,1.25) -- (-1,0.5) -- (1,-1.5) -- (1.25,-2.25);
\filldraw[fill=red,draw=red] (-1.25,1.25) circle (2pt);
\filldraw[fill=red,draw=red] (-1,0.5) circle (2pt);
\filldraw[fill=red,draw=red] (1,-1.5) circle (2pt);
\filldraw[fill=red,draw=red] (1.25,-2.25) circle (2pt);
\node[color=red] at (-.625,0.5) {$t_i$};
\node[color=red] at (.25,-1.5) {$s_{\phi(i)}$};
\end{tikzpicture}
\caption{If the layout from the left is the output of {\normalfont{\algname{ModifiedUntangle}}$(T,S,\phi|_I)$}, then Algorithm \ref{case3algorithm} will first perform a paired flip at $(u_3,v_3)$. Afterwards, it will not perform a paired flip at $(u_2,v_2)$ or $(u_1,v_1)$. Since $E(u_0)\cap E(v_0)=\emptyset$, the algorithm will perform subtree switches at $u_0$ and $v_0$, returning the layout shown on the right.}
\label{insertion-ex}
\end{center}
\end{figure}
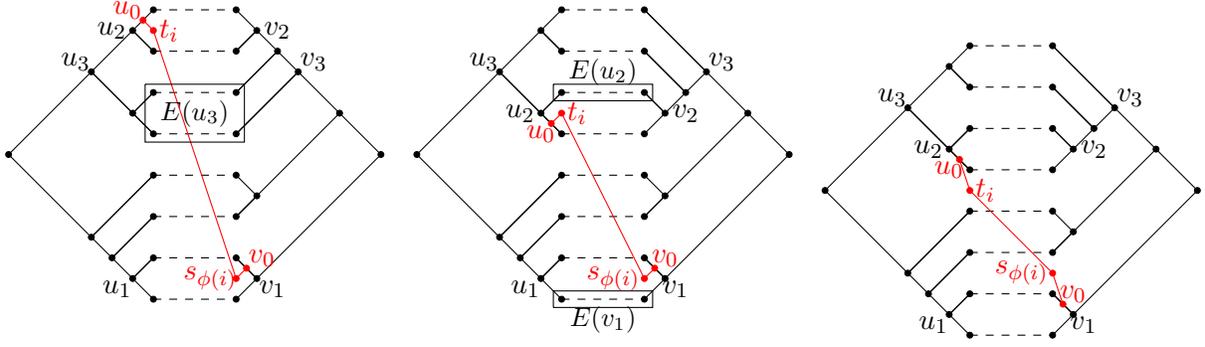

\begin{lemma}\label{insertioncase3}
Algorithm \ref{case3algorithm} solves the Tanglegram Single Edge Insertion Problem when $u_0\not>_T u_{Smax}$ and $v_0\not>_S v_{Tmax}$, where we also consider $L(I)_S=\emptyset$ and $L(I)_T=\emptyset$, respectively, as cases of $u_0\not>_T u_{Smax}$ and $v_0\not>_S v_{Tmax}$.
\end{lemma}
\begin{proof}
Using the same argument as in Lemma \ref{insertioncase1}, starting at $(X,Y)$ in Step 1 and performing some sequence of paired flips at elements in $L(I)_T\cup L(I)_S$ and subtree switches at $\{u_0,v_0\}$ solves the Tanglegram Single Edge Insertion Problem. 
Algorithm \ref{case3algorithm}  first considers the elements of $L(I)_T$ in the order $(u_{k+m},v_{k+m}),\ldots,(u_{k+1},v_{k+1})$ and then elements of $L(I)_S$ in the order $(u_k,v_k),\ldots,(u_1,v_1)$. Similar to Lemma \ref{insertioncase1}, we show that our choice at each element of $L(I)_T$ and $L(I)_S$ in this order extends to a solution to the Tanglegram Single Edge Insertion Problem. 

Consider $(u_j,v_j)\in L(I)_T$, and assume that our choices at all previously considered vertices extend to a solution to the Single Edge Insertion Problem. Some choice of operations at $(u_{j'},v_{j'})\in L(I)_T$ with $u_{j'}\leq_T u_j$, $(u_{\ell},v_{\ell})\in L(I)_S$, $u_0$, and $v_0$ solves the Single Edge Insertion Problem. This gives us a layout $(X_{min},Y_{min})$ that includes an ordering of $\leaf(u_j),\leaf(v_j)$. When $j\neq k+1$, this ordering has the form $(X_1UX_2,Y_1VY_2)$ with $U,V$ an ordering of $\leaf(u_{j-1}),\leaf(v_{j-1})$ as shown in Figure \ref{case3graph}. If we perform a paired flip at $(u_j,v_j)$, then we can perform a paired flip at $(u_{j-1},v_{j-1})$ to obtain the ordering $(\overline{X_2}U\overline{X_1},\overline{Y_2}V\overline{Y_1})$ on $\leaf(u_j),\leaf(v_j)$. 
Letting $\op{Edges}(U)$ be the between-tree edges with an endpoint in $U$, we see that the original layout $(X_{min},Y_{min})$ and the new layout have the same crossings between $(t_i,s_{\phi(i)})$ and $\op{Edges}(U)$. Thus, regardless of our choice at $(u_j,v_j)$, we can minimize crossings in $\op{Edges}(U)$.

\begin{figure}[h]
    \centering
    \begin{tikzpicture}[scale=.5]
    \draw (-1,7) -- (-2,6) -- (-1,5) -- (-1,7) -- (-4.5,3.5) -- (-3,2) -- (-1,4) -- (-1,0) -- (-4.5,3.5) -- (-6,2) -- (-1,-3) -- (-1,-1) -- (-2,-2);
    \draw[color=blue] (-2,2) -- (-1,3) -- (-1,1) -- (-2,2);
    \draw (1,7) -- (2,6) -- (1,5) -- (1,7) -- (4.5,3.5) -- (3,2) -- (1,4) -- (1,0) -- (4.5,3.5) -- (6,2) -- (1,-3) -- (1,-1) -- (2,-2);
    \draw[color=blue] (2,2) -- (1,3) -- (1,1) -- (2,2);
    \draw[dashed] (-1,6.5) -- (1,6.5);
    \draw[dashed] (-1,5.5) -- (1,5.5);
    \draw[dashed] (-1,6) -- (1,6);
    \draw[dashed] (-1,-2) -- (1,-2);
    \draw[dashed] (-1,-1.5) -- (1,-1.5);
    \draw[dashed] (-1,-2.5) -- (1,-2.5);
    \draw[color=red] (-1,2.25) -- (1,5.75);
    \draw[color=black!25!white,dashed] (-1,.5) -- (1,.5);
    \draw[color=black!25!white,dashed] (-1,2.5) -- (1,2.5);
    \draw[color=black!25!white,dashed] (-1,1.5) -- (1,1.5);
    \draw[color=black!25!white,dashed] (-1,3.5) -- (1,3.5);
    \draw[color=black!25!white,dashed] (-1,3) -- (1,3);
    \draw[color=black!25!white,dashed] (-1,2) -- (1,2);
    \draw[color=black!25!white,dashed] (-1,1) -- (1,1);
    \filldraw[fill=black,draw=black] (-6,2) circle (2pt);
    \filldraw[fill=black,draw=black] (-4.5,3.5) circle (2pt);
    \filldraw[fill=black,draw=black] (-2,6) circle (2pt);
    \filldraw[fill=black,draw=black] (-3,2) circle (2pt);
    \filldraw[fill=blue,draw=blue] (-2,2) circle (2pt);
    \filldraw[fill=black,draw=black] (-2,-2) circle (2pt);
    \filldraw[fill=black,draw=black] (6,2) circle (2pt);
    \filldraw[fill=black,draw=black] (4.5,3.5) circle (2pt);
    \filldraw[fill=black,draw=black] (2,6) circle (2pt);
    \filldraw[fill=black,draw=black] (3,2) circle (2pt);
    \filldraw[fill=blue,draw=blue] (2,2) circle (2pt);
    \filldraw[fill=black,draw=black] (2,-2) circle (2pt);
    \filldraw[fill=red,draw=red] (-1,2.25) circle (2pt);
    \filldraw[fill=red,draw=red] (1,5.75) circle (2pt);
    \node at (-3.5,2) {\small{$u_j$}};
    \node[color=blue] at (-2.25,1.5) {\small{$u_{j-1}$}};
    \node at (3.5,2) {\small{$v_j$}};
    \node[color=blue] at (2.25,1.5) {\small{$v_{j-1}$}};
    \node[color=red] at (-1.25,2) {\small{$t_i$}};
    \node[color=red] at (1.75,5.5) {\small{$s_{\phi(i)}$}};
    \node at (-2.375,6.25) {\small{$u_k$}};
    \node at (2.375,6.25) {\small{$v_k$}};
    \node at (-.5,3.5) {$X_1$};
    \node at (-.5,2) {$U$};
    \node at (-.5,0.5) {$X_2$};
    \node at (0.5,3.5) {$Y_1$};
    \node at (0.5,2) {$V$};
    \node at (0.5,0.5) {$Y_2$};
    \end{tikzpicture}
    \quad 
    \begin{tikzpicture}[scale=.5]
    \draw (-1,7) -- (-2,6) -- (-1,5) -- (-1,7) -- (-4.5,3.5) -- (-3,2) -- (-1,4) -- (-1,0) -- (-4.5,3.5) -- (-6,2) -- (-1,-3) -- (-1,-1) -- (-2,-2);
    \draw[color=blue] (-2,2) -- (-1,3) -- (-1,1) -- (-2,2);
    \draw (1,7) -- (2,6) -- (1,5) -- (1,7) -- (4.5,3.5) -- (3,2) -- (1,4) -- (1,0) -- (4.5,3.5) -- (6,2) -- (1,-3) -- (1,-1) -- (2,-2);
    \draw[color=blue] (2,2) -- (1,3) -- (1,1) -- (2,2);
    \draw[dashed] (-1,6.5) -- (1,6.5);
    \draw[dashed] (-1,5.5) -- (1,5.5);
    \draw[dashed] (-1,6) -- (1,6);
    \draw[dashed] (-1,-2) -- (1,-2);
    \draw[dashed] (-1,-1.5) -- (1,-1.5);
    \draw[dashed] (-1,-2.5) -- (1,-2.5);
    \draw[color=red] (-1,2.25) -- (1,5.75);
    \draw[color=black!25!white,dashed] (-1,.5) -- (1,.5);
    \draw[color=black!25!white,dashed] (-1,2.5) -- (1,2.5);
    \draw[color=black!25!white,dashed] (-1,1.5) -- (1,1.5);
    \draw[color=black!25!white,dashed] (-1,3.5) -- (1,3.5);
    \draw[color=black!25!white,dashed] (-1,3) -- (1,3);
    \draw[color=black!25!white,dashed] (-1,2) -- (1,2);
    \draw[color=black!25!white,dashed] (-1,1) -- (1,1);
    \filldraw[fill=black,draw=black] (-6,2) circle (2pt);
    \filldraw[fill=black,draw=black] (-4.5,3.5) circle (2pt);
    \filldraw[fill=black,draw=black] (-2,6) circle (2pt);
    \filldraw[fill=black,draw=black] (-3,2) circle (2pt);
    \filldraw[fill=blue,draw=blue] (-2,2) circle (2pt);
    \filldraw[fill=black,draw=black] (-2,-2) circle (2pt);
    \filldraw[fill=black,draw=black] (6,2) circle (2pt);
    \filldraw[fill=black,draw=black] (4.5,3.5) circle (2pt);
    \filldraw[fill=black,draw=black] (2,6) circle (2pt);
    \filldraw[fill=black,draw=black] (3,2) circle (2pt);
    \filldraw[fill=blue,draw=blue] (2,2) circle (2pt);
    \filldraw[fill=black,draw=black] (2,-2) circle (2pt);
    \filldraw[fill=red,draw=red] (-1,2.25) circle (2pt);
    \filldraw[fill=red,draw=red] (1,5.75) circle (2pt);
    \node at (-3.5,2) {\small{$u_j$}};
    \node[color=blue] at (-2.25,1.5) {\small{$u_{j-1}$}};
    \node at (3.5,2) {\small{$v_j$}};
    \node[color=blue] at (2.25,1.5) {\small{$v_{j-1}$}};
    \node[color=red] at (-1.25,2) {\small{$t_i$}};
    \node[color=red] at (1.75,5.5) {\small{$s_{\phi(i)}$}};
    \node at (-2.375,6.25) {\small{$u_k$}};
    \node at (2.375,6.25) {\small{$v_k$}};
    \node at (-.5,3.5) {$\overline{X_2}$};
    \node at (-.5,2) {$U$};
    \node at (-.5,0.5) {$\overline{X_1}$};
    \node at (0.5,3.5) {$\overline{Y_2}$};
    \node at (0.5,2) {$V$};
    \node at (0.5,0.5) {$\overline{Y_1}$};
    \end{tikzpicture}
    \caption{The effect of a paired flip at $(u_j,v_j)$ and $(u_{j-1},v_{j-1})$ when $j\geq 2$.}
    \label{case3graph}
\end{figure}
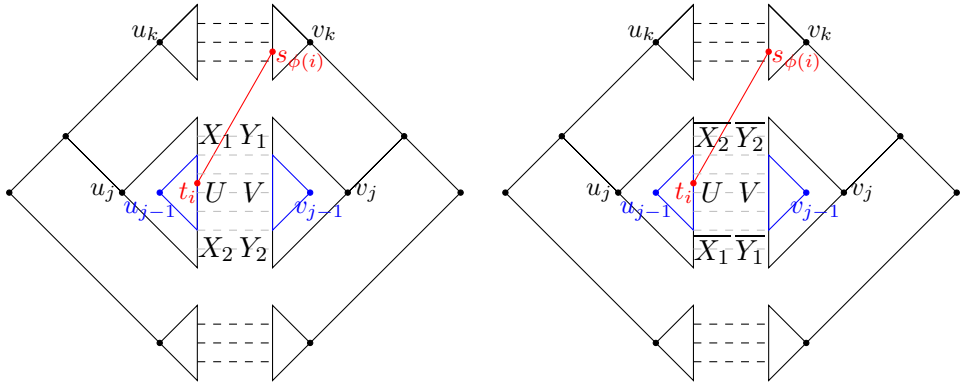

When $j=k+1$, the original ordering has the form $(X_1t_iUX_2,Y_0)$ or $(X_1Ut_iX_2,Y_0)$ with $U$ an ordering of $\leaf(u_0)\setminus \{t_i\}$. If we make a paired flip at $(u_{k+1},v_{k+1})$, we can perform a subtree switch at $u_0$ to respectively obtain $(\overline{X_2}t_i\overline{U}\,\overline{X_1},\overline{Y_0})$ or $(\overline{X_2}\,\overline{U}t_i\overline{X_1},\overline{Y_0})$. In either case, the resulting layout will again have the same crossings between $(t_i,s_{\phi(i)})$ and $\op{Edges}(U)$, so regardless of our choice at $(u_{k+1},v_{k+1})$, we can still minimize crossings in $\op{Edges}(U)$.

In both of the above cases, notice that the crossings involving $E(u_j)=\op{Edges}(X_1)\cup \op{Edges}(X_2)$ are affected, as $(t_i,s_{\phi(i)})$ will either cross $\op{Edges}(X_1)$ or $\op{Edges}(X_2)$. These crossings cannot be affected by paired flips at any $(u_{j'},v_{j'})\in L(I)_T$ with $j'<j$, nor by a subtree switch at $u_0$. Additionally, by Lemma \ref{case3order}, $E(u_j)$ does not intersect $E(v_{\ell})$ for $(u_\ell,v_\ell)\in L(I)_S$, nor does it intersect $E(v_0)$. Then $E(u_j)$ does not contain any edges with a leaf in $\leaf(v_\ell)$ for $\ell=0,1,\ldots, k$, and crossings in $E(u_j)$ cannot be affected by any of the choices at $(u_{\ell},v_{\ell})\in L(I)_S$ and $v_0$. Combined, we conclude that the choice at $(u_j,v_j)$ that extends to a solution to the Single Edge Insertion Problem must be the one that minimizes crossings in $E(u_j)$, which is precisely what the algorithm does in lines \ref{ujstart}-\ref{ujend}. Using induction, the sequence of choices in $L(I)_T$ extends to a solution.

Using a similar argument, we obtain the same conclusion for the choices of paired flips in $L(I)_S$. After making appropriate choices of paired flips for elements in $L(I)_T$ and $L(I)_S$, we know some combination of subtree switches at $\{u_0,v_0\}$ solves the Tanglegram Single Edge Insertion Problem.  When $E(u_0)\cap E(v_0)=\emptyset$, the choices at $u_0$ and $v_0$ affect different crossings, and the algorithm chooses to minimize crossings in each set.  When $E(u_0)\cap E(v_0)\neq \emptyset$, the algorithm checks all four possible layouts in lines \ref{tryallstart}-\ref{tryallend} and returns a layout with minimal crossings. In all cases, the output must be a solution to the Tanglegram Single Edge Insertion Problem. 
\end{proof}

Using these algorithms and results, we now define a combined algorithm that solves the Tanglegram Single Edge Insertion Problem below. The initial steps are similar to the previous algorithms. We then check for the conditions corresponding to each case, and proceed appropriately.

\begin{algorithm}[H]\DontPrintSemicolon
\caption{\algname{Insertion}}\label{insertionalgorithm}
\KwIn{tanglegram $(T,S,\phi)$, index $i$ such that $(T_I,S_{\phi(I)},\phi|_I)$ is planar for $I=[n]\setminus \{i\}$}
\KwOut{layout of $(T,S,\phi)$ that restricts to a planar layout of $(T_I,S_{\phi(I)},\phi|_I)$}
\tcp*{Step 1: initialize the algorithm.}
$(X,Y),L\coloneqq\algname{ModifiedUntangle}(T,S,\phi|_I)$\\
construct $L(I)$ from $L$ using Definition \ref{Lsubtanglegram}\\
$u_0\coloneqq$ parent of $t_i$, $v_0\coloneqq$ parent of $s_{\phi(i)}$\\
$L(I)_S\coloneqq\{(u,v)\in L(I):u\not>_T t_i,v>_S s_{\phi(i)}\}$ \label{lstart}\\
$L(I)_T\coloneqq\{(u,v)\in L(I):u>_T t_i,v\not>_S s_{\phi(i)}\}$ \label{lend}\\
\tcp*{Step 2: consider cases.}
use (\ref{maximums}) to define $u_{Smax}$ when $L(I)_S\neq\emptyset$ and $v_{Tmax}$ when $L(I)_T\neq \emptyset$\\
\If{$L(I)_S\neq \emptyset $ and $u_0>_T u_{Smax}$,\label{condition1}}
    {proceed from Step 2 of Algorithm \ref{case1algorithm}} 
\ElseIf{$L(I)_T\neq \emptyset$ and $v_0>_S v_{Tmax}$,\label{condition2}} 
    {proceed from Step 2 Algorithm \ref{case2algorithm}}
\Else 
    {proceed from Step 2 of Algorithm \ref{case3algorithm}} 
\end{algorithm}

\begin{proof}[Proof of Theorem \ref{insertiontime}]
Effectiveness of the \algname{Insertion} Algorithm follows from the cases considered in Step 2 combined with Lemma \ref{insertioncase1}, Remark \ref{insertioncase2}, and Lemma \ref{insertioncase3}, so it remains to show that the algorithm runs in $O(n^2)$ time and space. In Step 1, \algname{ModifiedUntangle} runs in $O(n^2)$ time and space by Remark \ref{modifieduntangletime}. Additionally, constructing $L(I)$ takes $O(n^2)$ time and space, as $L$ has size at most $n$, and for every $(u,v)\in L$, there are at most $n$ descendants of $u$ and $v$. Lines \ref{lstart}-\ref{lend} take $O(n)$ time since the list $L(I)$ also has size at most $n$. Defining $u_{Smax}$ and $v_{Tmax}$ and then checking the conditions in lines \ref{condition1} and \ref{condition2} also take $O(n)$ time since $L(I)_T$ and $L(I)_S$ have size at most $n$.

Now in Step 2 of all three cases, ordering the lists $L(I)_T$ and $L(I)_S$ can be done in $O(n \log n)$ time. Next, each of the $E$ sets can be calculated in $O(n)$ time and space, and we perform these calculations at most $n$ times for a total of $O(n^2)$ time and space. The paired flip and subtree switch choices in Step 3 of Algorithm \ref{case1algorithm}, \ref{case2algorithm}, or \ref{case3algorithm} run in $O(n)$ time each since counting crossings involving $(t_i,s_{\phi(i)})$ can be done in linear time, and performing paired flips and subtree switches can also be done in linear time. We make at most $n$ paired flip and subtree switch choices for a total of $O(n^2)$ time. In Algorithm \ref{case3algorithm}, we may also consider four layouts corresponding to combinations of $\{u_0,v_0\}$ in lines \ref{tryallstart}-\ref{tryallend}. Since each layout takes $O(n^2)$ space and crossings can be counted in $O(n^2)$ time by checking all ${n\choose 2}$ pairs of edges, these remaining steps also take $O(n^2)$ time and space. Thus, the \algname{Insertion} algorithm runs in $O(n^2)$ time and space, regardless of the case.
\end{proof}

If there exists a solution to the Tanglegram Layout Problem for $(T,S,\phi)$ such that all crossings involve a single edge, then \algname{Insertion} can be used to find a solution to the Tanglegram Layout Problem. In particular, one can use $\algname{Insertion}(T,S,\phi,i)$ over all possible $i$ to verify if $\crt(T,S,\phi)=1$. However, in general, the solution to the Single Edge Insertion Problem and the Layout Problem can differ by arbitrarily many crossings.

\begin{definition}
Let $(T,S,\phi)$ be a tanglegram of size $n$. For any $i\in [n]$ such that $(T_I,S_{\phi(I)},\phi|_I)$ with $I=[n]\setminus \{i\}$ is planar, define $\crtei((T,S,\phi),i)$ to be the common number of crossings in any solution to the Tanglegram Single Edge Insertion Problem. Define
\begin{equation}
    \crtei(T,S,\phi)=\min_i\{\crtei((T,S,\phi),i)\},
\end{equation}
where the minimum is taken over all well-defined choices of $i$.
\end{definition}

\begin{corollary}
\label{ex1}
For any $k\in \bbN$, there exists a tanglegram $(T,S,\phi)$ and such that $\crtei(T,S,\phi)-\crt(T,S,\phi)=k$.
\end{corollary}

\begin{proof}
For $k=1$, consider the tanglegram $(T,S,\phi)$ with two layouts in Figure \ref{example1}. Direct application of \algname{ModifiedUntangle} or the Tanglegram Kuratowski Theorem in \cite{kura} shows that the induced subtanglegram on $[n]\setminus \{i\}$ is planar only when $i=2$, and hence $\crtei(T,S,\phi)=\crtei((T,S,\phi),2)$.  $\algname{Insertion}((T,S,\phi),2)$ outputs the first layout in Figure \ref{example1} with three crossings, which is a solution to the Tanglegram Single Edge Insertion Problem by Theorem \ref{insertiontime}. Hence, $\crtei(T,S,\phi)=3$. Theorem \ref{insertiontime} also implies that if $\crt(T,S,\phi)=1$, then for some choice of $i\in [6]$, $\algname{Insertion}((T,S,\phi),i)$ will have one crossing. Since this not the case, the second layout in Figure \ref{example1} implies $\crt(T,S,\phi)=2$. Combined, we see that $\crtei(T,S,\phi)-\crt(T,S,\phi)=1$.

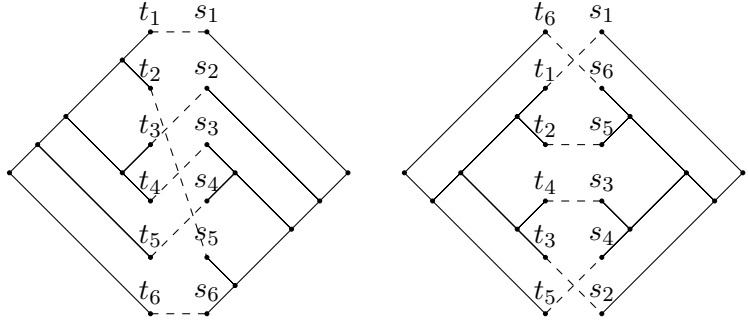
\begin{figure}[h]
    \centering
\begin{tikzpicture}[scale=0.75]
\filldraw[fill=black,draw=black] (-1,3) circle (1pt);
\filldraw[fill=black,draw=black] (-1,2) circle (1pt);
\filldraw[fill=black,draw=black] (-1,1) circle (1pt);
\filldraw[fill=black,draw=black] (-1,0) circle (1pt);
\filldraw[fill=black,draw=black] (-1,-1) circle (1pt);
\filldraw[fill=black,draw=black] (-1,-2) circle (1pt);
\filldraw[fill=black,draw=black] (0,3) circle (1pt);
\filldraw[fill=black,draw=black] (0,2) circle (1pt);
\filldraw[fill=black,draw=black] (0,1) circle (1pt);
\filldraw[fill=black,draw=black] (0,0) circle (1pt);
\filldraw[fill=black,draw=black] (0,-1) circle (1pt);
\filldraw[fill=black,draw=black] (0,-2) circle (1pt);
\draw (-1,3)-- (-1.5,2.5) -- (-1,2) -- (-1.5,2.5) -- (-2.5,1.5) -- (-1.5,0.5) -- (-1,1) -- (-1.5,0.5) -- (-1,0) -- (-2.5,1.5) -- (-3,1) -- (-1,-1) -- (-3,1) -- (-3.5,0.5)--(-1,-2);
\draw (0,3) -- (2.5,0.5) -- (2,0) -- (0,2) -- (2,0) -- (1.5,-0.5) -- (0.5,0.5) -- (0,1) -- (0.5,0.5) -- (0,0) -- (0.5,0.5) -- (1.5,-0.5) -- (0.5,-1.5) -- (0,-1) -- (0.5,-1.5) -- (0,-2) ;
\filldraw[fill=black,draw=black] (-1.5,2.5) circle (1pt);
\filldraw[fill=black,draw=black] (-2.5,1.5) circle (1pt);
\filldraw[fill=black,draw=black] (-1.5,0.5) circle (1pt);
\filldraw[fill=black,draw=black] (-3,1) circle (1pt);
\filldraw[fill=black,draw=black] (-3.5,0.5) circle (1pt);
\filldraw[fill=black,draw=black] (2.5,0.5) circle (1pt);
\filldraw[fill=black,draw=black] (2,0) circle (1pt);
\filldraw[fill=black,draw=black] (1.5,-0.5) circle (1pt);
\filldraw[fill=black,draw=black] (0.5,-1.5) circle (1pt);
\filldraw[fill=black,draw=black] (0.5,0.5) circle (1pt);
\draw[dashed] (-1,3) -- (0,3);
\draw[dashed] (-1,2) -- (0,-1);
\draw[dashed] (-1,1) -- (0,2);
\draw[dashed] (-1,0) -- (0,1);
\draw[dashed] (-1,-1) -- (0,0);
\draw[dashed] (-1,-2) -- (0,-2);
\node at (-1,3.325) {$t_1$};
\node at (-1,2.325) {$t_2$};
\node at (-1,1.325) {$t_3$};
\node at (-1,0.325) {$t_4$};
\node at (-1,-0.625) {$t_5$};
\node at (-1,-1.625) {$t_6$};
\node at (0,3.325) {$s_1$};
\node at (0,2.325) {$s_2$};
\node at (0,1.325) {$s_3$};
\node at (0,0.325) {$s_4$};
\node at (0,-0.625) {$s_5$};
\node at (0,-1.625) {$s_6$};

\filldraw[fill=black,draw=black] (6,3) circle (1pt);
\filldraw[fill=black,draw=black] (6,2) circle (1pt);
\filldraw[fill=black,draw=black] (6,1) circle (1pt);
\filldraw[fill=black,draw=black] (6,0) circle (1pt);
\filldraw[fill=black,draw=black] (6,-1) circle (1pt);
\filldraw[fill=black,draw=black] (6,-2) circle (1pt);
\filldraw[fill=black,draw=black] (7,3) circle (1pt);
\filldraw[fill=black,draw=black] (7,2) circle (1pt);
\filldraw[fill=black,draw=black] (7,1) circle (1pt);
\filldraw[fill=black,draw=black] (7,0) circle (1pt);
\filldraw[fill=black,draw=black] (7,-1) circle (1pt);
\filldraw[fill=black,draw=black] (7,-2) circle (1pt);
\draw[dashed] (6,3) -- (7,2);
\draw[dashed] (6,2) -- (7,3);
\draw[dashed] (6,1) -- (7,1);
\draw[dashed] (6,0) -- (7,0);
\draw[dashed] (6,-1) -- (7,-2);
\draw[dashed] (6,-2) -- (7,-1);
\node at (6,3.325) {$t_6$};
\node at (6,2.325) {$t_1$};
\node at (6,1.325) {$t_2$};
\node at (6,0.325) {$t_4$};
\node at (6,-0.625) {$t_3$};
\node at (6,-1.625) {$t_5$};
\node at (7,3.325) {$s_1$};
\node at (7,2.325) {$s_6$};
\node at (7,1.325) {$s_5$};
\node at (7,0.325) {$s_3$};
\node at (7,-0.625) {$s_4$};
\node at (7,-1.625) {$s_2$};
\draw (6,3) -- (3.5,0.5) -- (4,0) -- (6,2) -- (5.5,1.5) -- (6,1) -- (5.5,1.5) -- (4.5,0.5) -- (5.5,-.5) -- (6,0) -- (5.5,-.5) -- (6,-1) -- (4.5,0.5) -- (4,0) -- (6,-2)  ;
\draw (7,3) -- (9.5,0.5) -- (9,0) -- (8.5,0.5) -- (7.5,1.5) -- (7,2) -- (7.5,1.5) -- (7,1) -- (7.5,1.5) -- (8.5,0.5) -- (7.5,-.5) -- (7,0) -- (7.5,-.5) -- (7,-1) -- (8.5,0.5) -- (9,0) -- (7,-2);
\filldraw[fill=black,draw=black] (3.5,0.5) circle (1pt);
\filldraw[fill=black,draw=black] (4,0) circle (1pt);
\filldraw[fill=black,draw=black] (5.5,1.5) circle (1pt);
\filldraw[fill=black,draw=black] (4.5,0.5) circle (1pt);
\filldraw[fill=black,draw=black] (5.5,-0.5) circle (1pt);
\filldraw[fill=black,draw=black] (9.5,0.5) circle (1pt);
\filldraw[fill=black,draw=black] (9,0) circle (1pt);
\filldraw[fill=black,draw=black] (8.5,0.5) circle (1pt);
\filldraw[fill=black,draw=black] (7.5,1.5) circle (1pt);
\filldraw[fill=black,draw=black] (7.5,-0.5) circle (1pt);
\end{tikzpicture}
\caption{The output of $\normalfont{\algname{Insertion}}((T,S,\phi),2)$ and another layout for the same tanglegram.}
\label{example1}
\end{figure}

For $k>1$, form the size $k+5$ tanglegram $(T',S',\phi')$ by starting with $(T,S,\phi)$ and replacing the edge $(t_4,s_3)$ with a planar tanglegram $(\widetilde{T},\widetilde{S},\widetilde{\phi})$ of size $k$. Since $(T',S',\phi')$ contains $(T,S,\phi)$ as a subtanglegram, we know $\crt(T',S',\phi')\geq 2$, as any layout of $(T',S',\phi')$ with fewer than two crossings would restrict to a layout of $(T,S,\phi)$ with fewer than two crossings. We can take the second layout in Figure \ref{example1} and replace $(t_4,s_3)$ with a planar layout of $(\widetilde{T},\widetilde{S},\widetilde{\phi})$ to obtain a layout of $(T',S',\phi')$ with exactly 2 crossings, so $\crt(T',S',\phi')=2$.

Since $(T_{[6]\setminus \{i\}},S_{\phi([6]\setminus \{i\})},\phi|_{[6]\setminus \{i\}})$ is planar only when $i=2$, the same is true for the induced subtanglegrams $(T'_{[k+5]\setminus \{i\}},S'_{\phi([k+5]\setminus \{i\})},\phi'|_{[k+5]\setminus \{i\}})$, as  $(T'_{[k+5]\setminus \{i\}},S'_{\phi([k+5]\setminus \{i\})},\phi'|_{[k+5]\setminus \{i\}})$ contains $(T_{[6]\setminus \{i\}},S_{\phi([6]\setminus \{i\})},\phi|_{[6]\setminus \{i\}})$ as an induced subtanglegram. 
Starting with the first layout in Figure \ref{example1}, construct the layout $(X',Y')$ for $(T',S',\phi')$ by replacing $(t_4,s_3)$ with a planar layout of $(\widetilde{T},\widetilde{S},\widetilde{\phi})$. Notice that $(X',Y')$ has exactly $k+2$ crossings and restricts to a planar layout of $(T'_{[k+5]\setminus \{2\}},S'_{\phi([k+5]\setminus \{2\})},\phi'|_{[k+5]\setminus \{2\}})$. Furthermore, notice that the only leaf-matched pairs of $(T'_{[k+5]\setminus \{2\}},S'_{\phi([k+5]\setminus \{2\})},\phi'|_{[k+5]\setminus \{2\}})$ are the roots of the two trees and leaf-matched pairs of $(\widetilde{T},\widetilde{S},\widetilde{\phi})$, so Corollary \ref{existence} implies that a solution to the Tanglegram Single Edge Insertion Problem for $(T',S',\phi')$ and $i=2$ can be obtained by starting with $(X',Y')$ and using subtree switches at the parents of $t_2$ and $s_5$. However, these subtree switches increase the number of crossings, so $(X',Y')$ itself must be a solution to the Single Edge Insertion Problem. Hence, $\crtei(T',S',\phi')=\crtei((T',S',\phi'),2)=k+2$, and therefore $\crtei(T',S',\phi')-\crt(T',S',\phi')=k$.
\end{proof}

\section{Multiple edge insertion}\label{MEI}

In this section, we consider the Tanglegram Multiple Edge Insertion Problem, which we restate below for convenience of the reader. Similar to the corresponding problem for graphs, this problem is NP-hard, which we now show.

\begin{problem*}[\textbf{Tanglegram Multiple Edge Insertion}]
Given a tanglegram $(T,S,\phi)$ and a planar subtanglegram $(T_I,S_{\phi(I)},\phi|_I)$ induced by $I\subseteq [n]$, find a layout of $(T,S,\phi)$ that restricts to a planar layout of $(T_I,S_{\phi(I)},\phi|_I)$ and has the minimal number of crossings possible.
\end{problem*}

\begin{proof}[Proof of Theorem \ref{NP-hard}]
By Theorem \ref{nchoose2}, the crossing number of any tanglegram $(T,S,\phi)$ of size $n$ is strictly less than $\frac{1}{2}{n\choose 2}$. Then in an optimal layout of a tanglegram $(T,S,\phi)$, there must exist some edges $(t_i,s_{\phi(i)})$ and $(t_j,s_{\phi(j)})$ that do not cross. If we solve the Multiple Edge Insertion Problem for all index sets $I$ of size two, one of them will be $I=\{i,j\}$. After all ${n\choose 2}=O(n^2)$ iterations, one of these layouts will be a solution to the Tanglegram Layout Problem, which is NP-hard by Theorem \ref{nphard}.
\end{proof}

\subsection{Iterated single edge insertion}

We start by giving an algorithm that inserts a single edge at a time using the approach from the \algname{Insertion} algorithm. While this will not solve the Tanglegram Multiple Edge Insertion Problem, it has two advantages. First, the algorithm will run in polynomial time. Second, Theorem \ref{change} implies that if $(T_I,S|_{\phi(I)},\phi|_I)$ is a size $n-1$ subtanglegram of the size $n$ tanglegram $(T,S,\phi)$, then $\crt(T,S,\phi)$ and $\crt(T_I,S|_{\phi(I)},\phi|_I)$ differ by at most $n-3$. This can be used to obtain a general bound on a tanglegram's crossing number as a function of the number of edges that we remove to obtain a planar subtanglegram. The algorithm we give in this subsection will achieve this bound. 

\begin{lemma}\label{bound}
Let $(T,S,\phi)$ be a size $n$ tanglegram, and let $I\subseteq [n]$. If $(T_{I},S_{\phi(I)},\phi|_{I})$ is planar, then $\crt(T,S,\phi)\leq \frac{(n-|I|)\cdot (n+|I|-5)}{2}$.
\end{lemma}

\begin{proof}
We assume without loss of generality that $I=\{1,2,\ldots,k\}$, as otherwise we can relabel the tanglegram. Using the assumption $\crt(T_I,S_{\phi(I)},\phi|_I)=0$ with Theorem \ref{change},
\begin{equation*}
    \begin{split}
        \crt(T,S,\phi) & =\sum_{j=k+1}^{n} \left[\crt(T_{[j]},S_{\phi([j])},\phi|_{[j]})-\crt(T_{[j-1]},S_{\phi([j-1])},\phi|_{[j-1]})\right] \\
        & \leq \sum_{j=k+1}^{n} (j-3) 
        = \frac{(n-k)(n+k-5)}{2}. \qedhere
    \end{split}
\end{equation*}
\end{proof}

We now give an algorithm that starts with a planar layout of $(T_{I},S_{ \phi(I)},\phi|_{ I})$ and inserts one edge at a time, with operations chosen in a similar manner as in $\algname{Insertion}$. As before, we restrict ourselves to paired flips and subtree switches, though we perform a paired flip at each $(u,v)\in L(I)$ at most once to ensure we achieve the bound in Lemma \ref{bound}. 

\begin{algorithm}[H]\DontPrintSemicolon
\caption{\algname{IteratedInsertion}}
\KwIn{tanglegram $(T,S,\phi)$, index set $I$ such that $(T_{I},S_{ \phi(I)},\phi|_{ I})$ is planar}
\KwOut{layout of $(T,S,\phi)$ that restricts to a planar layout of $(T_{I},S_{ \phi(I)},\phi|_{ I})$}
$(X,Y),L(I)\coloneqq  \algname{ModifiedUntangle} (T_{I},S_{ \phi(I)},\phi|_{I})$ \label{iterstart1} 
$J \coloneqq I$, $M\coloneqq L(I)$ \label{iterstart2}\\
\While{$[n]\setminus J \neq \emptyset$,}
    {
    $i\coloneqq$ smallest element of $[n]\setminus J$ \tcp*{Choose an edge to insert.}
    extend $(X,Y)$ to a layout of $(T_{J\cup \{i\}},S_{\phi(J\cup \{i\})},\phi|_{J\cup \{i\}})$ \label{iterextend}\\
    $M_T\coloneqq\{(u,v)\in M:u>_T t_i,v\not>_S s_{\phi(i)}\}$ \tcp*{Perform paired flips at}
    $M_S\coloneqq\{(u,v)\in M:u\not>_T t_i,v>_Ss_{\phi(i)}\}$ \tcp*{ these pairs of vertices.}
    update $M\coloneqq M\setminus (M_T\cup M_S)$ \tcp*{Prevent future flips in $M_T\cup M_S$.}
    $u_0,v_0\coloneqq$ parents of $t_i,s_{\phi(i)}$ in $(T_{J\cup \{i\}},S_{\phi(J\cup \{i\})},\phi|_{J\cup \{i\}})$\\
    proceed from Step 2 of \algname{Insertion} with $L(I)_T,L(I)_S$ respectively replaced by $M_T,M_S$, and obtain $(X,Y)$ after operations at $M_T,M_S,u_0,v_0$\\
    update $J\coloneqq J\cup \{i\}$
    }
\Return $(X,Y)$
\end{algorithm}

\begin{theorem}
\label{iteratedinsertion}
Let $(T,S,\phi)$ be a size $n$ tanglegram and $I\subseteq [n]$ such that $(T_{I},S_{ \phi(I)},\phi|_{ I})$ is planar. \algname{IteratedInsertion} finds a layout of $(T,S,\phi)$ with at most $\frac{(n-|I|)\cdot (n+|I|-5)}{2}$ crossings in $O(n^3)$ time and $O(n^2)$ space.
\end{theorem}

\begin{proof}
We start by proving the run-time and space claims. Lines \ref{iterstart1}-\ref{iterstart2} run in $O(n^2)$ time and space. By construction, the \textbf{while} loop will have at most $n-|I|$ iterations since we add an element to $J$ at the end of each iteration. Line \ref{iterextend} can be done in $O(n^2)$ time and space by starting with the trees $T_{J\cup \{i\}}$ and $S_{\phi(J\cup \{i\})}$ and performing operations until the leaves of $T_{J}$ and $S_{\phi(J)}$ appear in the order indicated by $(X,Y)$. The remaining steps of the \textbf{while} loop also run in $O(n^2)$ time since \algname{Insertion} runs in $O(n^2)$ time, for a combined total of $O(n^3)$ time. Since we re-use the same variables $X,Y,J,M,M_T,M_S$, we also see that the algorithm runs in $O(n^2)$ space since \algname{Insertion} runs in $O(n^2)$ space by Theorem \ref{insertiontime}.

We now show that the output has at most $\frac{(n-|I|)\cdot (n+|I|-5)}{2}$ crossings. First, note that since $(T_I,S_{\phi(I)},\phi|_I)$ is planar, the layout $(X,Y)$ originally obtained from $\algname{ModifiedUntangle}(T_I,S_{\phi(I)},\phi|_I)$ in line 1 has 0 crossings. We show that during each iteration of the \textbf{while} loop, the algorithm adds at most $|J|-3$ crossings. This would imply that the total number of crossings at the end of \algname{IteratedInsertion} is bounded by the sum from the proof of Lemma \ref{bound}, implying the claim. 

We consider two cases for each iteration of the \textbf{while} loop. First, suppose that in this iteration of the loop, the sets $M_T$ and $M_S$ are empty. Then the algorithm only considers subtree switches at $u_0$ and $v_0$, the parents of $t_i$ and $s_{\phi(i)}$. In particular, we do not affect the relative ordering of any vertices in $(T_{J},S_{\phi(J)},\phi|_J)$, and thus we do not affect any previous crossings. Some combination of subtree switches will produce at most $|J|-3$ of the $|J|-1$ possible crossings. Proceeding from Step 2 of \algname{Insertion} will lead to using Algorithm \ref{case3algorithm}, which will consider subtree switches at $u_0$ and $v_0$ before returning the option with the fewest number of crossings. Thus, the algorithm adds at most $|J|-3$ new crossings from this iteration, all involving edge $(t_i,s_{\phi(i)})$.

Next, suppose that $M_T\cup M_S$ is nonempty, and let $(u,v)\in M_T\cup M_S$. By Lemma \ref{modified}, $(u,v)$ is a leaf-matched pair of $(T_I,S_{\phi(I)},\phi|_I)$. Since $(u,v)$ was not previously considered in the algorithm, it must be that for all previously inserted edges $(t,s)$, both or neither of $u>_T t$ and $v>_S s$ are true. From this, we see that crossings between $(t,s)$ and edges in $(T_I,S_{\phi(I)},\phi|_I)$ are not affected by paired flips at $(u,v)$. Furthermore, for any  other previously inserted edge $(t',s')$, we also know that either both or neither of $u>_T t'$ and $v>_S s'$ are true. In all cases, whether or not $(t,s)$ and $(t',s')$ intersect is not affected by a paired flip at $(u,v)$. Thus, all previously existing crossings are unaffected, and we only create new crossings involving the current edge $(t_i,s_{\phi(i)})$. Using the same argument as in the previous paragraph, proceeding from Step 2 of \algname{Insertion} will add at most $|J|-3$ new crossings, all of which involve edge $(t_i,s_{\phi(i)})$.
\end{proof}

\subsection{MultiInsertion Algorithm}

We now generalize our results from Section 4 and then give an algorithm for solving the Tanglegram Multiple Edge Insertion Problem. For this subsection, fix a tanglegram $(T,S,\phi)$ of size $n$, and fix $I\subseteq [n]$ such that $(T_I,S_{\phi(I)},\phi|_I)$ is a planar subtanglegram. Let $(X,Y)$ be a layout of $(T,S,\phi)$ that restricts to a planar layout of $(T_I,S_{\phi(I)},\phi|_I)$, and let $L(I)$ be the set of leaf-matched pairs of $(T_I,S_{\phi(I)},\phi|_I)$. Recall that for $u\in T$ and $v\in S$, $\leaf(u)$ and $\leaf(v)$ respectively denote the descendants of $u$ and $v$ that are leaves in $T$ and $S$. For $J\subseteq [n]$, we will abuse notation by using $\leaf(u)\cap J$ or $\leaf(v)\cap \phi(J)$ to denote the leaves in $\leaf(u)$ or $\leaf(v)$ that are indexed by $J$ or $\phi(J)$, respectively. Letting $I^c$ denote the complement of $I$, define the following sets:
\begin{equation}\label{LIsubsets}
    \begin{split}
        L(I)_0 &= \{(u,v)\in L(I): \leaf(u)\cap I^c \text{ is matched with } \leaf(v) \cap \phi(I)^c \text{ by $\phi$}\}\\
        L(I)_T &= \{(u,v)\in L(I): \text{ $|\leaf(u)\cap I^c|=1$ and $|\leaf(v)\cap \phi(I)^c|=0$}\}\\
        L(I)_S &= \{(u,v)\in L(I): \text{ $|\leaf(u)\cap I^c|=0$ and $|\leaf(v)\cap \phi(I)^c|=1$}\}\\
        L(I)_1& = L(I)\setminus (L(I)_0 \cup L(I)_T \cup L(I)_S )\\
        M(I) &= \{\text{internal vertices of $(T,S,\phi)$ that are not in $(T_I,S_{\phi(I)},\phi|_I)$.}\}
    \end{split}
\end{equation}
Notice that $L(I)_0,L(I)_T,L(I)_S$, and $L(I)_1$ partition the set $L(I)$. When $|I|=n-1$, the sets $L(I)_0$, $L(I)_T$ and $L(I)_S$ reduce to the same ones defined in Section \ref{TEI}, and $L(I)_1=\emptyset$. An example of these sets is shown in Figure \ref{LIexample}.

\begin{figure}[h]
    \centering
\begin{tikzpicture}[scale=0.725]
\filldraw[fill=black,draw=black] (1,4) circle (2pt);
\filldraw[fill=black,draw=black] (1,3) circle (2pt);
\filldraw[fill=black,draw=black] (1,2) circle (2pt);
\filldraw[fill=black,draw=black] (1,1) circle (2pt);
\filldraw[fill=black,draw=black] (1,0) circle (2pt);
\filldraw[fill=black,draw=black] (1,-1) circle (2pt);
\filldraw[fill=black,draw=black] (1,-2) circle (2pt);
\filldraw[fill=black,draw=black] (1,-3) circle (2pt);
\filldraw[fill=black,draw=black] (-1,4) circle (2pt);
\filldraw[fill=black,draw=black] (-1,3) circle (2pt);
\filldraw[fill=black,draw=black] (-1,2) circle (2pt);
\filldraw[fill=black,draw=black] (-1,1) circle (2pt);
\filldraw[fill=black,draw=black] (-1,0) circle (2pt);
\filldraw[fill=black,draw=black] (-1,-1) circle (2pt);
\filldraw[fill=black,draw=black] (-1,-2) circle (2pt);
\filldraw[fill=black,draw=black] (-1,-3) circle (2pt);
\draw[dashed] (-1,4) -- (1,4);
\draw[dashed] (-1,3) -- (1,3);
\draw[dashed] (-1,2) -- (1,2);
\draw[dashed] (-1,1) -- (1,1);
\draw[dashed] (-1,0) -- (1,0);
\draw[dashed] (-1,-1) -- (1,-1);
\draw[dashed] (-1,-2) -- (1,-2);
\draw[dashed] (-1,-3) -- (1,-3);
\draw (-1,4) -- (-1.5,3.5) -- (-1,3) -- (-1.5,3.5) -- (-2.5,2.5) -- (-1.5,1.5) -- (-1,2) -- (-1.5,1.5) -- (-1,1) -- (-2.5,2.5) -- (-4.5,0.5);
\draw (-1,-3) -- (-1.5,-2.5) -- (-1,-2) -- (-1.5,-2.5) -- (-2,-2) -- (-1,-1) -- (-2,-2) -- (-2.5,-1.5) -- (-1,0) -- (-2.5,-1.5) -- (-4.5,0.5);
\draw (1,4) -- (1.5,3.5) -- (1,3) -- (1.5,3.5) -- (2,3) -- (1,2) -- (2,3) -- (2.5,2.5) -- (1,1) -- (2.5,2.5) -- (4.5,0.5) ; 
\draw (1,-3) -- (1.5,-2.5) -- (1,-2) -- (1.5,-2.5) -- (4.5,0.5);
\draw (1,0) -- (1.5,-0.5) -- (1,-1) -- (1.5,-0.5) --(3.5,1.5);
\filldraw[fill=black,draw=black] (-1.5,1.5) circle (2pt);
\filldraw[fill=black,draw=black] (-2.5,-1.5) circle (2pt);
\filldraw[fill=black,draw=black] (-2,-2) circle (2pt);
\filldraw[fill=black,draw=black] (2,3) circle (2pt);
\filldraw[fill=black,draw=black] (1.5,-0.5) circle (2pt);
\filldraw[fill=black,draw=black] (3.5,1.5) circle (2pt);
\filldraw[fill=black,draw=black] (-1.5,3.5) circle (2pt);
\filldraw[fill=black,draw=black] (1.5,3.5) circle (2pt);
\filldraw[fill=black,draw=black] (2.5,2.5) circle (2pt);
\filldraw[fill=black,draw=black] (-2.5,2.5) circle (2pt);
\filldraw[fill=black,draw=black] (4.5,0.5) circle (2pt);
\filldraw[fill=black,draw=black] (-4.5,0.5) circle (2pt);
\filldraw[fill=black,draw=black] (1.5,-2.5) circle (2pt);
\filldraw[fill=black,draw=black] (-1.5,-2.5) circle (2pt);
\draw[color=red] (-2.875,2.125) -- (-1,0.25) -- (-1.25,0.5) -- (-1,0.75) -- (1,3.5) -- (1.25,3.25);
\draw[color=red] (-1,0.25) -- (1,-2.5) -- (1.25,-2.75);
\draw[color=red] (-1.75,3.25) -- (-1,2.5) -- (1,0.5) -- (1.75,-0.25);
\filldraw[color=red] (-1,0.25) circle (2pt);
\filldraw[color=red] (-1,0.75) circle (2pt);
\filldraw[color=red] (-2.875,2.125) circle (2pt);
\filldraw[color=red] (-1.25,0.5) circle (2pt);
\filldraw[color=red] (1,3.5) circle (2pt);
\filldraw[color=red] (1.25,3.25) circle (2pt);
\filldraw[color=red] (1,-2.5) circle (2pt);
\filldraw[color=red] (1.25,-2.75) circle (2pt);
\filldraw[color=red] (-1.75,3.25) circle (2pt);
\filldraw[color=red] (-1,2.5) circle (2pt);
\filldraw[color=red] (1,0.5) circle (2pt);
\filldraw[color=red] (1.75,-0.25) circle (2pt);
\node at (-0.625,3.75) {$t_1$};
\node at (-0.625,3.25) {$t_{2}$};
\node[color=red] at (-0.625,2.5) {$t_{3}$};
\node at (-0.625,1.75) {$t_4$};
\node at (-0.625,1.25) {$t_{5}$};
\node[color=red] at (-0.625,0.75) {$t_{6}$};
\node[color=red] at (-0.625,0.25) {$t_7$};
\node at (-0.625,-0.25) {$t_{8}$};
\node at (-0.625,-1.25) {$t_9$};
\node at (-0.625,-1.75) {$t_{10}$};
\node at (-0.625,-2.75) {$t_{11}$};
\node at (-5,0.5) {$u_{4}$};
\node at (5,0.5) {$v_{4}$};
\node[color=red] at (-3.125,2.375) {$u_7$};
\node[color=red] at (-1.625,0.5) {$u_6$};
\node at (-2.75,2.75) {$u_3$};
\node at (2.75,2.75) {$v_3$};
\node[color=red] at (-2,3.5) {$u_5$};
\node at (-1.75,3.75) {$u_2$};
\node at (1.75,3.75) {$v_2$};
\node[color=red] at (1.5,3) {$v_7$};
\node[color=red] at (2,-0.5) {$v_6$};
\node at (-1.75,-2.75) {$u_1$};
\node at (1.75,-2.75) {$v_1$};
\node[color=red] at (1.5,-3.125) {$v_{5}$};
\end{tikzpicture}
\caption{The subtanglegram induced by $I=\{1,2,4,5,8,9,10,11\}\subseteq [11]$ is shown in black. We see that $L(I)=\{(u_1,v_1),(u_2,v_2),(u_3,v_3),(u_{4},v_{4})\}$ and $M(I)=\{u_5,u_6,u_7,v_5,v_6,v_7\}$. Using (\ref{LIsubsets}), we partition $L(I)$ into $L(I)_0=\{(u_{4},v_{4})\}$, $L(I)_T=\emptyset$, $L(I)_S=\{(u_1,v_1),(u_2,v_2)\}$, and $L(I)_1=\{(u_3,v_3)\}$.}
\label{LIexample}
\end{figure}
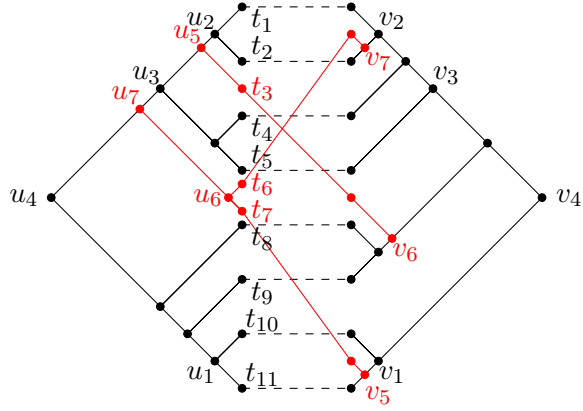

\begin{lemma}\label{multioperations}
Let $(X',Y')$ be the image of $(X,Y)$ after any sequence of paired flips at elements in $L(I)$ and subtree switches at elements of $M(I)$. Then $(X',Y')$ also restricts to a planar layout of $(T_I,S_{\phi(I)},\phi|_I)$.
\end{lemma}

\begin{proof}
It suffices to show the claim when $(X',Y')$ is the image of $(X,Y)$ after a single subtree switch or paired flip. If we start with $(X,Y)$ and perform a flip at $(u,v)\in L(I)$, then restricted to the subtanglegram $(T_I,S_{\phi(I)},\phi|_I)$, this has the same effect as performing a paired flip on a planar layout of $(T_I,S_{\phi(I)},\phi|_I)$. Thus, $(X',Y')$ is still planar when restricted to $(T_I,S_{\phi(I)},\phi|_I)$.

Now suppose we start with $(X,Y)$ and perform a subtree switch at $u\in M(I)$ to obtain $(X',Y')$. Without loss of generality, we assume $u\in T$, as $u\in S$ is done similarly. Then we can decompose $(X,Y)=(X_1U_1U_2X_2,Y)$ and $(X',Y')=(X_1U_2U_1X_2,Y)$, where $U_1$ and $U_2$ order the leaves of the subtrees rooted at the children of $u$. Since $u$ is not a vertex of $(T_I,S_{\phi(I)},\phi|_I)$, Definitions~\ref{subtreedefinition}~and~\ref{subtanglegramdefinition} imply that either 
\begin{enumerate}[label=(\arabic*)]
    \item $u$ is not in the minimal subtree of $T$ containing the leaves indexed by $I$, or
    \item $u$ is one of the internal vertices with one child that is suppressed when forming $T_I$ from $T$. 
\end{enumerate}
In the case of (1), the entire subtree rooted at $u$ does not appear in $T_I$. Then both $U_1$ and $U_2$ are deleted when we restrict $(X,Y)$ and $(X',Y')$ to $(T_I,S_{\phi(I)},\phi|_I)$, and both layouts restrict to the same layout of the subtanglegram. In the case of (2), the entire subtree rooted at one of the children of $u$ does not appear in $T_I$, so either $U_1$ or $U_2$ is deleted when $(X,Y)$ and $(X',Y')$ are restricted to $(T_I,S_{\phi(I)},\phi|_I)$. Again, $(X',Y')$ and $(X,Y)$ will restrict to the exact same layout of $(T_I,S_{\phi(I)},\phi|_I)$ as $(X,Y)$. In both cases, since $(X,Y)$ restricts to a planar layout of $(T_I,S_{\phi(I)},\phi|_I)$, so does $(X',Y')$. 
\end{proof}

\begin{lemma}
\label{samecrossings}
Let $(X',Y')$ be the image of $(X,Y)$ after a paired flip at $(u,v) \in L(I)_0$. Then $(X,Y)$ and $(X',Y')$ have the same crossings.
\end{lemma}

\begin{proof}
Since $(u,v)\in L(I)_0$, notice that both or neither of $u>_T t_i$ and $v>_S s_{\phi(i)}$ are true for any $i\notin I$. If 
$u>_Tt_i$ and $v>_Ss_{\phi(i)}$, a paired flip at $(u,v)$ flips the induced subtanglegram on the subtrees rooted at $u$ and $v$, preserving all crossings in $(T,S,\phi)$ that involve $(t_i,s_{\phi(i)})$. 
Otherwise, if $u\not>_Tt_i$ and $v\not>_Ss_{\phi(i)}$, then $(t_i,s_{\phi(i)})$ cross all or none of the edges between the subtrees rooted at $u$ and $v$, and a paired at $(u,v)$ does not affect these crossings.
\end{proof}

\begin{lemma}\label{MEIsolution}
A solution to the Tanglegram Multiple Edge Insertion Problem can be obtained by starting at $(X,Y)$ and performing a sequence of subtree switches at elements in $ M(I)$ and paired flips at elements in $L(I)_T\cup L(I)_S\cup L(I)_1$.
\end{lemma}

\begin{proof}
Let $(X_{min},Y_{min})$ be a solution to the Multiple Edge Insertion Problem. Some composition of flips $f_m\circ \ldots \circ f_2\circ f_1$ maps $(X,Y)$ to $(X_{min},Y_{min})$, as flips generate all trees isomorphic to $T$ and $S$. All of these flips commute and have order 2, so we can also assume that all $f_i$ are distinct, i.e., no flips occur at any vertex more than once. 

If none of the flips in $\{f_1,f_2,\ldots,f_m\}$ involve vertices in $M(I)$, then restricting $(X,Y)$ and $(X_{min},Y_{min})$ to the subtanglegram $(T_I,S_{\phi(I)},\phi|_I)$, these flips are equivalent to a sequence of flips mapping one planar layout of $(T_I,S_{\phi(I)},\phi|_I)$ to another. By Theorem \ref{planarlayouts}, these flips must be equivalent to a sequence of paired flips at elements in $L(I)$. Since all flips commute, we can commute paired flips involving $(u,v)\in L(I)_0$ to be the last ones performed. By Lemma \ref{samecrossings}, paired flips at $(u,v)\in L(I)_0$ do not affect any crossings, so we can obtain another solution to the Multiple Edge Insertion Problem by excluding them, leaving only the paired flips at elements in $L(I)\setminus L(I)_0=L(I)_T\cup L(I)_S\cup L(I)_1$.

Now suppose some flip in $\{f_1,f_2,\ldots,f_m\}$ involves an element in $M(I)$. Without loss of generality, assume $f_1$ is a flip at $u\in M(I)$, $u\in T$, and no flips occur at any $u'>_T u$. 
Now define $g_1$ and $h_1$ to be flips at the children of $u$, or treat these as the identity if the corresponding child is a leaf. Notice that the composition
$$f_m\circ \ldots \circ f_2\circ h_1\circ g_1\circ h_1\circ g_1\circ f_1$$
also maps $(X,Y)$ to $(X_{min},Y_{min})$. The composition $h_1\circ g_1\circ f_1$ is a subtree switch at $u$, and the remaining operations $g_1,h_1,f_2,\ldots, f_m$ do not involve flips at any $u'\in M(I)$ with $u'\geq_T u$. If $g_1$ or $h_1$ are equivalent to any of the $f_i$, then we commute flips and replace $f_i\circ g_1$ or $f_i \circ h_1$ with the identity. We will find that $f_m\circ \ldots \circ f_2\circ h_1\circ g_1$ is equivalent to
$\widetilde{f}_k\circ \ldots \circ \widetilde{f}_2\circ \widetilde{f}_{1},$
and in this second composition, there are no repeated flips at any vertex, and no flips occur at $u'\in M(I)$ with $u'\geq_T u$. If some $\widetilde{f}_i$ involves a flip at a vertex in $M(I)$, then we iterate this argument. Since we choose a maximal vertex $u\in M(I)$ at each iteration, this process will eventually terminate. We will then have a sequence of subtree switches mapping $(X,Y)$ to some layout $(X',Y')$, and the remaining flips mapping $(X',Y')$ to $(X_{min},Y_{min})$ do not involve vertices in $M(I)$. The conclusion then follows from the preceding paragraph.
\end{proof}

We now focus on $M(I)$ and define subsets $M(I)_0$, $M(I)_T$, and $M(I)_S$ that allow us to generalize our results from \algname{Insertion}. For any leaves $t_i$ and $s_{\phi(i)}$ with $i\notin I$, we let $P(t_i)$ and $P(s_{\phi(i)})$ respective denote their parents. We define $M(I)_0$ to contain all $P(t_i)$ and $P(s_{\phi(i)})$ for $i\notin I$, such that
\begin{itemize}
    \item $\leaf(P(t_i))\cap I^c=\{t_i\}$,
    \item $\leaf(P(s_{\phi(i)}))\cap {\phi(I)^c}=\{s_{\phi(i)}\}$, and
    \item there exists $(u,v)\in L(I)$ with either $P(t_i)>_Tu$ \text{ and } $v>_SP(s_{\phi(i)})$, \text{ or } $P(s_{\phi(i)})>_Sv$ \text{ and } $u>_TP(t_i)$. 
\end{itemize}
Note that these combined properties imply $(u,v)\notin L(I)_0$ by definition. The set $M(I)_0$ is intended to generalize the results in Algorithms \ref{case1algorithm} and \ref{case2algorithm}, which are the cases $u_0>_T u_{Smax}$ and $v_0>_S v_{Tmax}$.

Now we define $M(I)_T$ and $M(I)_S$, which generalize additional situations from our \algname{Insertion} algorithm. For any $t_i\in T$ with $i\notin I$, define $A(t_i)$ to be the minimal distance ancestor $u$ of $t_i$ such that $\leaf(u)\cap I\neq \emptyset$.
 Now define $M(I)_S$ to contain all $P(s_{\phi(i)})\in M(I)\setminus M(I)_0$ for $i\notin I$ such that $\leaf(P(s_{\phi(i)}))\cap {\phi(I)^c}=\{s_{\phi(i)}\}$, and either
 \begin{itemize}
     \item $\leaf(A(t_i))\cap I$ and $\leaf(P(s_{\phi(i)}))\cap \phi(I)$ do not contain matched leaves, or
     \item there exists $(u,v)\in L(I)$ such that $A(t_i)>_Tu$ and $v>_SP(s_{\phi(i)})$. In this case, note that  these properties imply $(u,v)\notin L(I)_0$.
 \end{itemize}
Similarly, define $A(s_{\phi(i)})$ to be the minimal distance ancestor $v$ of $s_{\phi(i)}$ such that $\leaf(v)\cap \phi(I)\neq \emptyset$. Define $M(I)_T$ to contain all $P(t_i)\in M(I)\setminus M(I)_0$ for $i\notin I$ such that $\leaf(P(t_i))\cap {I^c}=\{t_i\}$, and either
\begin{itemize}
    \item $\leaf(P(t_i))\cap I$ and $\leaf(A(s_{\phi(i)}))\cap \phi(I)$ do not contain matched leaves, or
    \item there exists $(u,v)\in L(I)$ such that $u>_TP(t_i)$ and $A(s_{\phi(i)}) >_S v$. Again, $(u,v)\notin L(I)_0$. 
\end{itemize}
Finally, we define $M(I)_1=M(I)\setminus (M(I)_0\cup M(I)_T\cup M(I)_S)$ to be the remaining inserted vertices so that $M(I)_0, M(I)_T, M(I)_S$, and $M(I)_1$ partition $M(I)$. An example of these sets is shown in Figure \ref{MIexample}.

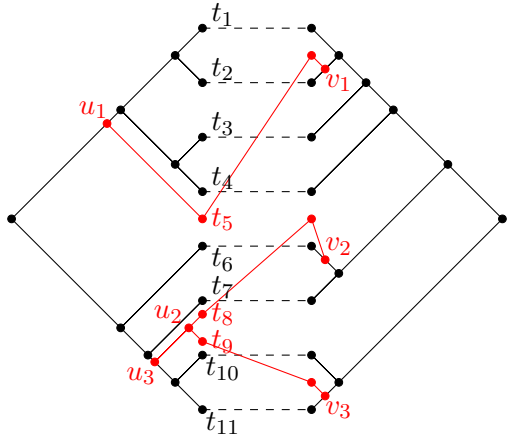
\begin{figure}[h]
    \centering
\begin{tikzpicture}[scale=0.725]
\filldraw[fill=black,draw=black] (1,4) circle (2pt);
\filldraw[fill=black,draw=black] (1,3) circle (2pt);
\filldraw[fill=black,draw=black] (1,2) circle (2pt);
\filldraw[fill=black,draw=black] (1,1) circle (2pt);
\filldraw[fill=black,draw=black] (1,0) circle (2pt);
\filldraw[fill=black,draw=black] (1,-1) circle (2pt);
\filldraw[fill=black,draw=black] (1,-2) circle (2pt);
\filldraw[fill=black,draw=black] (1,-3) circle (2pt);
\filldraw[fill=black,draw=black] (-1,4) circle (2pt);
\filldraw[fill=black,draw=black] (-1,3) circle (2pt);
\filldraw[fill=black,draw=black] (-1,2) circle (2pt);
\filldraw[fill=black,draw=black] (-1,1) circle (2pt);
\filldraw[fill=black,draw=black] (-1,0) circle (2pt);
\filldraw[fill=black,draw=black] (-1,-1) circle (2pt);
\filldraw[fill=black,draw=black] (-1,-2) circle (2pt);
\filldraw[fill=black,draw=black] (-1,-3) circle (2pt);
\draw[dashed] (-1,4) -- (1,4);
\draw[dashed] (-1,3) -- (1,3);
\draw[dashed] (-1,2) -- (1,2);
\draw[dashed] (-1,1) -- (1,1);
\draw[dashed] (-1,0) -- (1,0);
\draw[dashed] (-1,-1) -- (1,-1);
\draw[dashed] (-1,-2) -- (1,-2);
\draw[dashed] (-1,-3) -- (1,-3);
\draw (-1,4) -- (-1.5,3.5) -- (-1,3) -- (-1.5,3.5) -- (-2.5,2.5) -- (-1.5,1.5) -- (-1,2) -- (-1.5,1.5) -- (-1,1) -- (-2.5,2.5) -- (-4.5,0.5);
\draw (-1,-3) -- (-1.5,-2.5) -- (-1,-2) -- (-1.5,-2.5) -- (-2,-2) -- (-1,-1) -- (-2,-2) -- (-2.5,-1.5) -- (-1,0) -- (-2.5,-1.5) -- (-4.5,0.5);
\draw (1,4) -- (1.5,3.5) -- (1,3) -- (1.5,3.5) -- (2,3) -- (1,2) -- (2,3) -- (2.5,2.5) -- (1,1) -- (2.5,2.5) -- (4.5,0.5) ; 
\draw (1,-3) -- (1.5,-2.5) -- (1,-2) -- (1.5,-2.5) -- (4.5,0.5);
\draw (1,0) -- (1.5,-0.5) -- (1,-1) -- (1.5,-0.5) --(3.5,1.5);
\filldraw[fill=black,draw=black] (-1.5,1.5) circle (2pt);
\filldraw[fill=black,draw=black] (-2.5,-1.5) circle (2pt);
\filldraw[fill=black,draw=black] (-2,-2) circle (2pt);
\filldraw[fill=black,draw=black] (2,3) circle (2pt);
\filldraw[fill=black,draw=black] (1.5,-0.5) circle (2pt);
\filldraw[fill=black,draw=black] (3.5,1.5) circle (2pt);
\filldraw[fill=black,draw=black] (-1.5,3.5) circle (2pt);
\filldraw[fill=black,draw=black] (1.5,3.5) circle (2pt);
\filldraw[fill=black,draw=black] (2.5,2.5) circle (2pt);
\filldraw[fill=black,draw=black] (-2.5,2.5) circle (2pt);
\filldraw[fill=black,draw=black] (4.5,0.5) circle (2pt);
\filldraw[fill=black,draw=black] (-4.5,0.5) circle (2pt);
\filldraw[fill=black,draw=black] (1.5,-2.5) circle (2pt);
\filldraw[fill=black,draw=black] (-1.5,-2.5) circle (2pt);
\draw[color=red] (1.25,3.25) -- (1,3.5) -- (-1,0.5) -- (-2.75,2.25);
\draw[color=red] (1.25,-2.75) -- (1,-2.5) -- (-1,-1.75) -- (-1.25,-1.5) -- (-1.875,-2.125) -- (-1,-1.25) -- (1,0.5)-- (1.25,-0.25);
\filldraw[color=red] (-1,0.5) circle (2pt);
\node[color=red] at (1.5,3) {$v_1$};
\filldraw[color=red] (1.25,3.25) circle (2pt);
\filldraw[color=red] (1,3.5) circle (2pt);
\node[color=red] at (-3,2.5) {$u_1$};
\filldraw[color=red] (-2.75,2.25) circle (2pt);
\node[color=red] at (1.5,-3) {$v_3$};
\filldraw[color=red] (1.25,-2.75) circle (2pt);
\filldraw[color=red] (1,-2.5) circle (2pt);
\filldraw[color=red] (-1,-1.75) circle (2pt);
\node[color=red] at (-1.625,-1.25) {$u_2$};
\filldraw[color=red] (-1.25,-1.5) circle (2pt);
\node[color=red] at (-2.125,-2.375) {$u_3$};
\filldraw[color=red] (-1.875,-2.125) circle (2pt);
\filldraw[color=red] (-1,-1.25) circle (2pt);
\filldraw[color=red] (1,0.5) circle (2pt);
\node[color=red] at (1.5,0) {$v_2$};
\filldraw[color=red] (1.25,-0.25) circle (2pt);
\node at (-0.625,4.25) {$t_1$};
\node at (-0.625,3.25) {$t_{2}$};
\node[color=red] at (-0.625,0.5) {$t_{5}$};
\node at (-0.625,2.25) {$t_3$};
\node at (-0.625,1.25) {$t_{4}$};
\node[color=red] at (-0.625,-1.25) {$t_{8}$};
\node[color=red] at (-0.625,-1.75) {$t_9$};
\node at (-0.625,-0.25) {$t_{6}$};
\node at (-0.625,-0.75) {$t_7$};
\node at (-0.625,-2.25) {$t_{10}$};
\node at (-0.625,-3.25) {$t_{11}$};
\end{tikzpicture}
\caption{For the subtanglegram induced by $I=\{1,2,3,4,6,7,10,11\}\subseteq [11]$ shown in black, $M(I)$ is partitioned into $M(I)_0=\{u_1,v_1\}$, $M(I)_T=\emptyset$, $M(I)_S=\{v_2,v_3\}$, and $M(I)_1=\{u_2,u_3\}$.}
\label{MIexample}
\end{figure}

We now define a natural partial order $\preceq$ on $\mathcal{P}=L(I)_T \cup L(I)_S \cup L(I)_1\cup M(I)$ using the partial orders on $T$ and $S$. This will be useful for both determining what crossings to focus on at each element, as well as the order in which we should consider the elements when performing subtree switches and paired flips.
\begin{itemize}
    \item For $(u_1,v_1),(u_2,v_2)\in \mathcal{P}\cap L(I)$, $(u_1,v_1)\preceq (u_2,v_2)$ if $u_1\leq_T u_2$ and $v_1\leq_S v_2$.
    \item For $(u_1,v_1)\in \mathcal{P}\cap L(I),u\in \mathcal{P}\cap T$, $(u_1,v_1)\preceq u$ if $u_1\leq_T u$, and $u\preceq (u_1,v_1)$ if $u\leq_T u_1$.
    \item For $(u_1,v_1)\in \mathcal{P}\cap L(I),v\in \mathcal{P}\cap S$, $(u_1,v_1)\preceq v$ if $v_1\leq_S v$, and $v\preceq (u_1,v_1)$ if $v\leq_S v_1$.
    \item For $u_1,u_2\in  \mathcal{P}\cap T$, $u_1\preceq u_2$ if $u_1\leq_T u_2$. 
    \item For $v_1,v_2\in  \mathcal{P}\cap S$, $v_1\preceq v_2$ if $v_1\leq_S v_2$.
    \item Finally, take the transitive closure of the above relations.
\end{itemize}
Recall that for partially ordered sets, we say $z\in \mathcal{P}$ \textit{covers} $x\in \mathcal{P}$ if $x\prec z$ and there do not exist any $y\in \mathcal{P}$ with $x\prec y\prec z$, i.e., $z$ is larger than $x$ and there are no elements between them. We use the notation $x\succdot y$ or $y\precdot  x$ to denote the covering relation.

For the elements in $L(I)_T,L(I)_S,M(I)_0,M(I)_T,$ and $M(I)_S$, we now define generalizations of the $E(u_j)$ and $E(v_j)$ sets from \algname{Insertion}. These sets, which we denote $\cross(u,v)$, $\cross(u)$, or $\cross(v)$, will determine what crossings we should consider at the corresponding elements. Using the notation $e_j$ for the between-tree edge $(t_j,s_{\phi(j)})$, we organize these definitions in Table \ref{crossingsets}. 
        
\begin{table}[h!]
    \centering
    \begin{tabular}{ |c | c | c |}
        \hline
        Element & Relevant Properties & Crossing Set \\
        \hline
        $(u,v)\in L(I)_T$ & \makecell{$\leaf(u)\cap I^c=\{t_i\}$} & \makecell{$\{(e_i,e_j):t_j\in \leaf(u)\setminus \leaf(u')$ \\
        $\text{ where $(u,v)\succdot u'$ or $(u,v)\succdot (u',v')$}\}$} \\
         \hline 
        $(u,v)\in L(I)_S$ & \makecell{$\leaf(v)\cap \phi(I)^c=\{s_{\phi(i)}\}$ } & \makecell{$\{(e_i,e_j):s_{\phi(j)}\in \leaf(v)\setminus \leaf(v')$ \\ $\text{ where $(u,v)\succdot v'$ or $(u,v)\succdot (u',v')$}\}$} \\
        \hline 
        $P(t_i)\in M(I)_0$ & \makecell{there exists $(u,v)\in L(I)$  with\\ $P(t_i)>_T u$ and $v>_S P(s_{\phi(i)})$} & \makecell{$\{(e_i,e_j):t_j\in \leaf(P(t_i))\setminus \leaf(u)$ \\ $\text{ where $P(t_i)\succdot (u,v)$}\}$}\\
        \hline 
        $P(s_{\phi(i)})\in M(I)_0$ & \makecell{there exists $(u,v)\in L(I)$  with\\ $P(t_i)>_T u$ and $v>_S P(s_{\phi(i)})$} & $\{(e_i,e_j):s_{\phi(j)}\in \leaf(P(s_{\phi(i)}))\setminus \{s_{\phi(i)}\}\}$ \\
        \hline 
        $P(t_i)\in M(I)_0$ & \makecell{there exists $(u,v)\in L(I)$ with \\  $P(s_{\phi(i)})>_S v$ and $u>_T P(t_i)$} & $\{(e_i,e_j):t_j\in \leaf(P(t_i))\setminus \{t_i\}\}$ \\
        \hline 
        $P(s_{\phi(i)})\in M(I)_0$ & \makecell{there exists $(u,v)\in L(I)$  with \\ $P(s_{\phi(i)})>_S v$ and $u>_T P(t_i)$} &  \makecell{$\{(e_i,e_j):s_{\phi(j)}\in \leaf(P(s_{\phi(i)}))\setminus \leaf(v)$\\ $\text{ where $P(s_{\phi(i)})\succdot (u,v)$}$\}}\\
        \hline 
        $P(t_i) \in M(I)_T$ & $\leaf(P(t_i))\cap{I^c}=\{t_i\}$ & $\{(e_i,e_j):t_j\in \leaf(P(t_i))\setminus \{t_i\}\}$ \\
        \hline 
        $P(s_{\phi(i)})\in M(I)_S$ & $\leaf(P(s_{\phi(i)}))\cap{\phi(I)^c}=\{s_{\phi(i)}\}$ & $\{(e_i,e_j):t_j\in \leaf(P(s_{\phi(i)}))\setminus \{s_{\phi(i)}\}\}$ \\
        \hline
    \end{tabular}
    \caption{Crossing sets for elements in $L(I)_T$, $L(I)_S$, $M(I)_0$, $M(I)_T$, and $M(I)_S$.}
    \label{crossingsets}
\end{table}

Finally, recall that for partially ordered sets $\mathcal{P}$, a \textit{linear extension} $F:\mathcal{P}\to \{1,2,\ldots,|\mathcal{P}|\}$ is a bijective map such that $F(w)\geq F(w')$ in the usual order on $\bbN$ whenever $w\succeq w'$, or equivalently, a total ordering of $\mathcal{P}$ that respects the partial ordering. We can use a linear extension on $\mathcal{P}$ with respect to $\preceq$ defined above to determine the order that we perform operations, as by construction of $\preceq$, the order $F^{-1}(|\mathcal{P}|),\ldots ,F^{-1}(2),F^{-1}(1)$ guarantees that we consider all ancestors of a vertex in $T$ or $S$ before we consider the vertex itself. We now give an example of the sets in Table \ref{crossingsets}, followed by our \algname{MultiInsertion} algorithm.

\begin{example}\label{crossingexample}
Consider the tanglegram in Figure \ref{multiexample}, and let $I=\{1,2,4,5,8,9,10,11\}$. We index elements in $\mathcal{P}$ by their images under a linear extension, so
\begin{equation}
    \mathcal{P}=\{v_1,(u_2,v_2),v_3,v_4,(u_5,v_5),u_6,(u_7,v_7),u_8,u_9\}.
\end{equation}
Observe that $L(I)_T=\emptyset$, $L(I)_S=\{(u_2,v_2),(u_5,v_5)\}$, $L(I)_1=\{(u_7,v_7)\}$, $M(I)_0=\emptyset$, $M(I)_T=\{u_6\}$, $M(I)_S=\{v_1,v_3,v_4\}$, and $M(I)_1=\{u_8,u_9\}$. Table \ref{crossingsets} states 
\begin{itemize}
    \item $\cross(v_1)=\{(e_7,e_{11})\}$,
    \item $\cross(u_2,v_2)=\{(e_{7},e_{10})\}$,
    \item $\cross(v_3)=\{(e_3,e_8),(e_3,e_9)\}$,
    \item $\cross(v_4)=\{(e_6,e_2)\}$,
    \item $\cross(u_5,v_5)=\{(e_6,e_1)\}$, and
    \item $\cross(u_6)=\{(e_3,e_1),(e_3,e_2)\}$.
\end{itemize}

\begin{figure}[h]
    \centering
\begin{tikzpicture}[scale=0.7]
\filldraw[fill=black,draw=black] (1,4) circle (2pt);
\filldraw[fill=black,draw=black] (1,3) circle (2pt);
\filldraw[fill=black,draw=black] (1,2) circle (2pt);
\filldraw[fill=black,draw=black] (1,1) circle (2pt);
\filldraw[fill=black,draw=black] (1,0) circle (2pt);
\filldraw[fill=black,draw=black] (1,-1) circle (2pt);
\filldraw[fill=black,draw=black] (1,-2) circle (2pt);
\filldraw[fill=black,draw=black] (1,-3) circle (2pt);
\filldraw[fill=black,draw=black] (-1,4) circle (2pt);
\filldraw[fill=black,draw=black] (-1,3) circle (2pt);
\filldraw[fill=black,draw=black] (-1,2) circle (2pt);
\filldraw[fill=black,draw=black] (-1,1) circle (2pt);
\filldraw[fill=black,draw=black] (-1,0) circle (2pt);
\filldraw[fill=black,draw=black] (-1,-1) circle (2pt);
\filldraw[fill=black,draw=black] (-1,-2) circle (2pt);
\filldraw[fill=black,draw=black] (-1,-3) circle (2pt);
\draw[dashed] (-1,4) -- (1,4);
\draw[dashed] (-1,3) -- (1,3);
\draw[dashed] (-1,2) -- (1,2);
\draw[dashed] (-1,1) -- (1,1);
\draw[dashed] (-1,0) -- (1,0);
\draw[dashed] (-1,-1) -- (1,-1);
\draw[dashed] (-1,-2) -- (1,-2);
\draw[dashed] (-1,-3) -- (1,-3);
\draw (-1,4) -- (-1.5,3.5) -- (-1,3) -- (-1.5,3.5) -- (-2.5,2.5) -- (-1.5,1.5) -- (-1,2) -- (-1.5,1.5) -- (-1,1) -- (-2.5,2.5) -- (-4.5,0.5);
\draw (-1,-3) -- (-1.5,-2.5) -- (-1,-2) -- (-1.5,-2.5) -- (-2,-2) -- (-1,-1) -- (-2,-2) -- (-2.5,-1.5) -- (-1,0) -- (-2.5,-1.5) -- (-4.5,0.5);
\draw (1,4) -- (1.5,3.5) -- (1,3) -- (1.5,3.5) -- (2,3) -- (1,2) -- (2,3) -- (2.5,2.5) -- (1,1) -- (2.5,2.5) -- (4.5,0.5) ; 
\draw (1,-3) -- (1.5,-2.5) -- (1,-2) -- (1.5,-2.5) -- (4.5,0.5);
\draw (1,0) -- (1.5,-0.5) -- (1,-1) -- (1.5,-0.5) --(3.5,1.5);
\filldraw[fill=black,draw=black] (-1.5,1.5) circle (2pt);
\filldraw[fill=black,draw=black] (-2.5,-1.5) circle (2pt);
\filldraw[fill=black,draw=black] (-2,-2) circle (2pt);
\filldraw[fill=black,draw=black] (2,3) circle (2pt);
\filldraw[fill=black,draw=black] (1.5,-0.5) circle (2pt);
\filldraw[fill=black,draw=black] (3.5,1.5) circle (2pt);
\filldraw[fill=black,draw=black] (-1.5,3.5) circle (2pt);
\filldraw[fill=black,draw=black] (1.5,3.5) circle (2pt);
\filldraw[fill=black,draw=black] (2.5,2.5) circle (2pt);
\filldraw[fill=black,draw=black] (-2.5,2.5) circle (2pt);
\filldraw[fill=black,draw=black] (4.5,0.5) circle (2pt);
\filldraw[fill=black,draw=black] (-4.5,0.5) circle (2pt);
\filldraw[fill=black,draw=black] (1.5,-2.5) circle (2pt);
\filldraw[fill=black,draw=black] (-1.5,-2.5) circle (2pt);
\draw[color=red] (-2.875,2.125) -- (-1,0.25) -- (-1.25,0.5) -- (-1,0.75) -- (1,3.5) -- (1.25,3.25);
\draw[color=red] (-1,0.25) -- (1,-2.5) -- (1.25,-2.75);
\draw[color=red] (-1.75,3.25) -- (-1,2.5) -- (1,0.5) -- (1.75,-0.25);
\filldraw[color=red] (-1,0.25) circle (2pt);
\filldraw[color=red] (-1,0.75) circle (2pt);
\filldraw[color=red] (-2.875,2.125) circle (2pt);
\filldraw[color=red] (-1.25,0.5) circle (2pt);
\filldraw[color=red] (1,3.5) circle (2pt);
\filldraw[color=red] (1.25,3.25) circle (2pt);
\filldraw[color=red] (1,-2.5) circle (2pt);
\filldraw[color=red] (1.25,-2.75) circle (2pt);
\filldraw[color=red] (-1.75,3.25) circle (2pt);
\filldraw[color=red] (-1,2.5) circle (2pt);
\filldraw[color=red] (1,0.5) circle (2pt);
\filldraw[color=red] (1.75,-0.25) circle (2pt);
\node[color=red] at (-3.125,2.375) {$u_9$};
\node[color=red] at (-1.625,0.5) {$u_8$};
\node at (-2.75,2.75) {$u_7$};
\node at (2.75,2.75) {$v_7$};
\node[color=red] at (-2,3.5) {$u_6$};
\node at (-1.75,3.75) {$u_5$};
\node at (1.75,3.75) {$v_5$};
\node[color=red] at (1.5,3) {$v_4$};
\node[color=red] at (2,-0.5) {$v_3$};
\node at (-1.75,-2.75) {$u_2$};
\node at (1.75,-2.75) {$v_2$};
\node[color=red] at (1.5,-3.125) {$v_{1}$};
\node at (-0.625,3.75) {$t_1$};
\node at (-0.625,3.25) {$t_{2}$};
\node[color=red] at (-0.625,2.5) {$t_{3}$};
\node at (-0.625,1.75) {$t_4$};
\node at (-0.625,1.25) {$t_{5}$};
\node[color=red] at (-0.625,0.75) {$t_{6}$};
\node[color=red] at (-0.625,0.25) {$t_7$};
\node at (-0.625,-0.25) {$t_{8}$};
\node at (-0.625,-1.25) {$t_9$};
\node at (-0.625,-1.75) {$t_{10}$};
\node at (-0.625,-2.75) {$t_{11}$};
\end{tikzpicture}
\quad 
\begin{tikzpicture}[scale=0.7]
\filldraw[fill=black,draw=black] (1,4) circle (2pt);
\filldraw[fill=black,draw=black] (1,3) circle (2pt);
\filldraw[fill=black,draw=black] (1,2) circle (2pt);
\filldraw[fill=black,draw=black] (1,1) circle (2pt);
\filldraw[fill=black,draw=black] (1,0) circle (2pt);
\filldraw[fill=black,draw=black] (1,-1) circle (2pt);
\filldraw[fill=black,draw=black] (1,-2) circle (2pt);
\filldraw[fill=black,draw=black] (1,-3) circle (2pt);
\filldraw[fill=black,draw=black] (-1,4) circle (2pt);
\filldraw[fill=black,draw=black] (-1,3) circle (2pt);
\filldraw[fill=black,draw=black] (-1,2) circle (2pt);
\filldraw[fill=black,draw=black] (-1,1) circle (2pt);
\filldraw[fill=black,draw=black] (-1,0) circle (2pt);
\filldraw[fill=black,draw=black] (-1,-1) circle (2pt);
\filldraw[fill=black,draw=black] (-1,-2) circle (2pt);
\filldraw[fill=black,draw=black] (-1,-3) circle (2pt);
\draw[dashed] (-1,4) -- (1,4);
\draw[dashed] (-1,3) -- (1,3);
\draw[dashed] (-1,2) -- (1,2);
\draw[dashed] (-1,1) -- (1,1);
\draw[dashed] (-1,0) -- (1,0);
\draw[dashed] (-1,-1) -- (1,-1);
\draw[dashed] (-1,-2) -- (1,-2);
\draw[dashed] (-1,-3) -- (1,-3);
\draw (-1,4) -- (-1.5,3.5) -- (-1,3) -- (-1.5,3.5) -- (-2.5,2.5) -- (-1.5,1.5) -- (-1,2) -- (-1.5,1.5) -- (-1,1) -- (-2.5,2.5) -- (-4.5,0.5);
\draw (-1,-3) -- (-1.5,-2.5) -- (-1,-2) -- (-1.5,-2.5) -- (-2,-2) -- (-1,-1) -- (-2,-2) -- (-2.5,-1.5) -- (-1,0) -- (-2.5,-1.5) -- (-4.5,0.5);
\draw (1,4) -- (1.5,3.5) -- (1,3) -- (1.5,3.5) -- (2,3) -- (1,2) -- (2,3) -- (2.5,2.5) -- (1,1) -- (2.5,2.5) -- (4.5,0.5) ; 
\draw (1,-3) -- (1.5,-2.5) -- (1,-2) -- (1.5,-2.5) -- (4.5,0.5);
\draw (1,0) -- (1.5,-0.5) -- (1,-1) -- (1.5,-0.5) --(3.5,1.5);
\filldraw[fill=black,draw=black] (-1.5,1.5) circle (2pt);
\filldraw[fill=black,draw=black] (-2.5,-1.5) circle (2pt);
\filldraw[fill=black,draw=black] (-2,-2) circle (2pt);
\filldraw[fill=black,draw=black] (2,3) circle (2pt);
\filldraw[fill=black,draw=black] (1.5,-0.5) circle (2pt);
\filldraw[fill=black,draw=black] (3.5,1.5) circle (2pt);
\filldraw[fill=black,draw=black] (-1.5,3.5) circle (2pt);
\filldraw[fill=black,draw=black] (1.5,3.5) circle (2pt);
\filldraw[fill=black,draw=black] (2.5,2.5) circle (2pt);
\filldraw[fill=black,draw=black] (-2.5,2.5) circle (2pt);
\filldraw[fill=black,draw=black] (4.5,0.5) circle (2pt);
\filldraw[fill=black,draw=black] (-4.5,0.5) circle (2pt);
\filldraw[fill=black,draw=black] (1.5,-2.5) circle (2pt);
\filldraw[fill=black,draw=black] (-1.5,-2.5) circle (2pt);
\draw[color=red] (-2.875,2.125) -- (-1,0.25) -- (-1.25,0.5) -- (-1,0.75) -- (1,2.5) -- (1.25,3.25);
\draw[color=red] (-1,0.25) -- (1,-1.5) -- (1.25,-2.25);
\draw[color=red] (-1.75,3.25) -- (-1,2.5) -- (1,0.5) -- (1.75,-0.25);
\filldraw[color=red] (-1,0.25) circle (2pt);
\filldraw[color=red] (-1,0.75) circle (2pt);
\filldraw[color=red] (-2.875,2.125) circle (2pt);
\filldraw[color=red] (-1.25,0.5) circle (2pt);
\filldraw[color=red] (1,2.5) circle (2pt);
\filldraw[color=red] (1.25,3.25) circle (2pt);
\filldraw[color=red] (1,-1.5) circle (2pt);
\filldraw[color=red] (1.25,-2.25) circle (2pt);
\filldraw[color=red] (-1.75,3.25) circle (2pt);
\filldraw[color=red] (-1,2.5) circle (2pt);
\filldraw[color=red] (1,0.5) circle (2pt);
\filldraw[color=red] (1.75,-0.25) circle (2pt);
\node[color=red] at (-3.125,2.375) {$u_9$};
\node[color=red] at (-1.625,0.5) {$u_8$};
\node at (-2.75,2.75) {$u_7$};
\node at (2.75,2.75) {$v_7$};
\node[color=red] at (-2,3.5) {$u_6$};
\node at (-1.75,3.75) {$u_5$};
\node at (1.75,3.75) {$v_5$};
\node[color=red] at (1.5,3) {$v_4$};
\node[color=red] at (2,-0.5) {$v_3$};
\node at (-1.75,-2.75) {$u_2$};
\node at (1.75,-2.75) {$v_2$};
\node[color=red] at (1.5,-2) {$v_{1}$};
\node[color=white] at (1.5,-3.125) {$v_{1}$};
\node at (-0.625,3.75) {$t_1$};
\node at (-0.625,3.25) {$t_{2}$};
\node[color=red] at (-0.625,2.5) {$t_{3}$};
\node at (-0.625,1.75) {$t_4$};
\node at (-0.625,1.25) {$t_{5}$};
\node[color=red] at (-0.625,0.75) {$t_{6}$};
\node[color=red] at (-0.625,0.25) {$t_7$};
\node at (-0.625,-0.25) {$t_{8}$};
\node at (-0.625,-1.25) {$t_9$};
\node at (-0.625,-1.75) {$t_{11}$};
\node at (-0.625,-2.75) {$t_{10}$};
\end{tikzpicture}
\caption{Two layouts for the same tanglegram.}
\label{multiexample}
\end{figure}
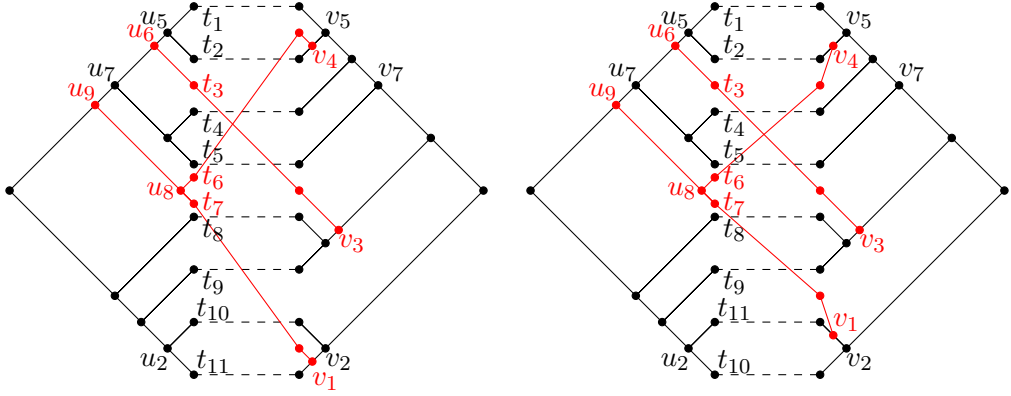
\end{example}

\begin{algorithm}[H]\DontPrintSemicolon
\caption{\algname{MultiInsertion}}
\KwIn{tanglegram $(T,S,\phi)$, index set $I$ such that $(T_I,S_{\phi(I)},\phi|_I)$ is planar}
\KwOut{layout of $(T,S,\phi)$ that restricts to a planar layout of $(T_I,S_{\phi(I)},\phi|_I)$}
$(X,Y),L\coloneqq\algname{ModifiedUntangle}(T,S,\phi|_I)$ \label{multistart1}\\
construct $L(I)$ from $L$ using Definition \ref{Lsubtanglegram}\\
construct $M(I)$, $M(I)_0$, $M(I)_T$, $M(I)_S$, $M(I)_1$, $L(I)_0$, $L(I)_T$, $L(I)_S$, and $L(I)_1$ \\
calculate $\cross(w)$ for all $w\in L(I)_T\cup L(I)_S\cup M(I)_0 \cup M(I)_T\cup M(I)_S$ using Table \ref{crossingsets}\\
$F\coloneqq$ linear extension of $\mathcal{P}=L(I)_T\cup L(I)_S\cup L(I)_1\cup M(I)$ with respect to partial order $\preceq$\\
$(X_{min},Y_{min})\coloneqq(X,Y)$ \label{multistart2}\\
\For{$C \subseteq  L(I)_1\cup M(I)_1$,\label{multiforloop}}
    {
    $(X',Y')\coloneqq(X,Y)$ \\
    \For{$j=|\mathcal{P}|,\ldots ,2,1$,\label{multiforloop2}}
        {
        \If{[$F^{-1}(j)\in C \cap L(I)_1$] or [$F^{-1}(j)\in L(I)_T\cup L(I)_S$ and more than half of the crossings in $\cross(F^{-1}(j))$ occur in $(X',Y')$],}
            {update $(X',Y')\coloneqq \algname{PairedFlip}((X',Y'),F^{-1}(j))$}
        \ElseIf{[$F^{-1}(j)\in C \cap M(I)_1 \cap T$] or [$F^{-1}(j)\in M(I)_T \cup (M(I)_0\cap T)$ and more than half of the crossings in $\cross(F^{-1}(j))$ occur in $(X',Y')$],}
            {update $X'\coloneqq \algname{SubtreeSwitch}(X',F^{-1}(j))$}
        \ElseIf{[$F^{-1}(j)\in C \cap M(I)_1 \cap S$] or [$F^{-1}(j)\in M(I)_S \cup (M(I)_0\cap S)$ and more than half of the crossings in $\cross(F^{-1}(j))$ occur in $(X',Y')$],}
            {update $Y'\coloneqq \algname{SubtreeSwitch}(Y',F^{-1}(j))$}
        }
    \If{$(X',Y')$ has fewer crossings than $(X_{min},Y_{min})$,\label{remainingsteps}}
        {update $(X_{min},Y_{min})\coloneqq(X',Y')$\label{trackmin}}
    }
\Return $(X_{min},Y_{min})$
\end{algorithm}

\begin{example}
Consider the tanglegram $(T,S,\phi)$ and linear extension of $\mathcal{P}$ from Example \ref{crossingexample}. Suppose that the the output of $\algname{ModifiedUntangle}(T,S,\phi|_I)$ is the layout on the left in Figure \ref{multiexample}. \algname{MultiInsertion} will consider $2^{|L(I)_1\cup M(I)_1|}=8$ iterations in the \textbf{for} loop in line \ref{multiforloop}. In each iteration, it produces another layout, and ultimately the algorithm returns the layout encountered that has the fewest number of crossings. In the iteration $C=\emptyset$, the algorithm will perform a subtree switch at $v_4$, a paired flip at $(u_2,v_2)$, and a subtree switch at $v_1$, resulting in a layout with seven crossings shown on the right in Figure \ref{multiexample}. Note that this is not a solution to the Multiple Edge Insertion Problem since the iteration with $C=\{(u_7,v_7)\}$ will produce a layout with fewer crossings.
\end{example}

\begin{lemma}\label{forloop}
In some iteration of the \textbf{for} loop in line \ref{multiforloop} of \algname{MultiInsertion}, the resulting layout $(X',Y')$ is a solution to the Tanglegram Multiple Edge Insertion Problem.
\end{lemma}

\begin{proof}
By Lemma \ref{MEIsolution},  starting with $(X,Y)$ in line \ref{multistart1} and performing some sequence of paired flips and subtree switches at elements in $\mathcal{P}$ produces a solution. In some iteration of the \textbf{for} loop in line \ref{multiforloop}, the choices at elements in $L(I)_1\cup M(I)_1$ all extend to a solution. We will assume that we are in this iteration, and we will show that the choice at each element in  $L(I)_T\cup L(I)_S\cup M(I)_0\cup M(I)_T\cup M(I)_S$ extends to a solution, provided that all prior choices extend to a solution.

Consider $(u,v)\in L(I)_T$. By definition of $L(I)_T$, $u$ is an ancestor of a single inserted leaf $t_i$ and $v$ is not an ancestor of any inserted leaves. Thus, the choice at $(u,v)$ can only affect the crossings 
$C=\{(e_i,e_j): t_j\in \leaf(u)\}.$
Applying the same arguments as in Lemma \ref{insertioncase1} and \ref{insertioncase3}, if we perform a paired flip at $(u,v)$, then we can perform either a paired flip or a subtree flip at the element covered by $(u,v)$ to obtain the same crossings in $C\setminus \cross(u,v)$. The case of a paired flip is illustrated in Figure \ref{multifig1}. From this, we see that the choice at $(u,v)$ that extends to a solution must be one that minimizes crossings in $\cross(u,v)$, which is what \algname{MultiInsertion} does.  The case $(u,v)\in L(I)_S$ follows by a similar argument.

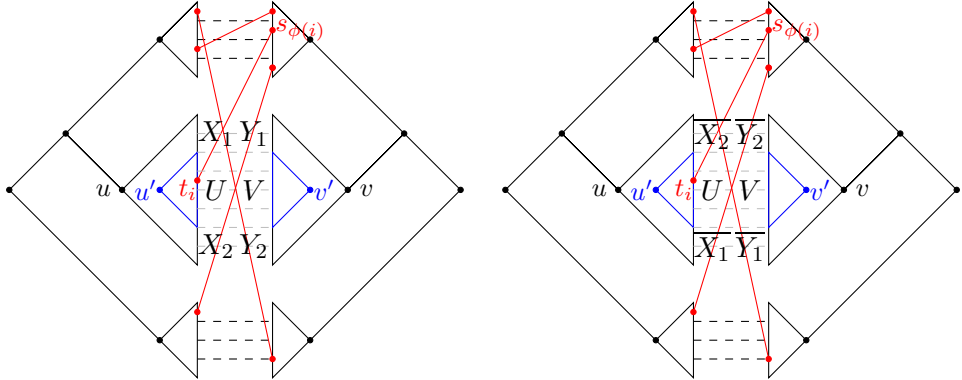
\begin{figure}[h]
    \centering
    \begin{tikzpicture}[scale=.5]
    \draw (-1,7) -- (-2,6) -- (-1,5) -- (-1,7) -- (-4.5,3.5) -- (-3,2) -- (-1,4) -- (-1,0) -- (-4.5,3.5) -- (-6,2) -- (-1,-3) -- (-1,-1) -- (-2,-2);
    \draw[color=blue] (-2,2) -- (-1,3) -- (-1,1) -- (-2,2);
    \draw (1,7) -- (2,6) -- (1,5) -- (1,7) -- (4.5,3.5) -- (3,2) -- (1,4) -- (1,0) -- (4.5,3.5) -- (6,2) -- (1,-3) -- (1,-1) -- (2,-2);
    \draw[color=blue] (2,2) -- (1,3) -- (1,1) -- (2,2);
    \draw[dashed] (-1,6.5) -- (1,6.5);
    \draw[dashed] (-1,5.5) -- (1,5.5);
    \draw[dashed] (-1,6) -- (1,6);
    \draw[dashed] (-1,-2) -- (1,-2);
    \draw[dashed] (-1,-1.5) -- (1,-1.5);
    \draw[dashed] (-1,-2.5) -- (1,-2.5);
    \draw[color=red] (-1,2.25) -- (1,6.25);
    \draw[color=black!25!white,dashed] (-1,.5) -- (1,.5);
    \draw[color=black!25!white,dashed] (-1,2.5) -- (1,2.5);
    \draw[color=black!25!white,dashed] (-1,1.5) -- (1,1.5);
    \draw[color=black!25!white,dashed] (-1,3.5) -- (1,3.5);
    \draw[color=black!25!white,dashed] (-1,3) -- (1,3);
    \draw[color=black!25!white,dashed] (-1,2) -- (1,2);
    \draw[color=black!25!white,dashed] (-1,1) -- (1,1);
    \filldraw[fill=black,draw=black] (-6,2) circle (2pt);
    \filldraw[fill=black,draw=black] (-4.5,3.5) circle (2pt);
    \filldraw[fill=black,draw=black] (-2,6) circle (2pt);
    \filldraw[fill=black,draw=black] (-3,2) circle (2pt);
    \filldraw[fill=blue,draw=blue] (-2,2) circle (2pt);
    \filldraw[fill=black,draw=black] (-2,-2) circle (2pt);
    \filldraw[fill=black,draw=black] (6,2) circle (2pt);
    \filldraw[fill=black,draw=black] (4.5,3.5) circle (2pt);
    \filldraw[fill=black,draw=black] (2,6) circle (2pt);
    \filldraw[fill=black,draw=black] (3,2) circle (2pt);
    \filldraw[fill=blue,draw=blue] (2,2) circle (2pt);
    \filldraw[fill=black,draw=black] (2,-2) circle (2pt);
    \filldraw[fill=red,draw=red] (-1,2.25) circle (2pt);
    \filldraw[fill=red,draw=red] (1,6.25) circle (2pt);
    \draw[color=red] (-1,5.75) -- (1,6.75);
    \draw[color=red] (-1,6.75) -- (1,-2.5);
    \draw[color=red] (-1,-1.25) -- (1,5.25);
        \filldraw[fill=red,draw=red] (-1,5.75) circle (2pt);
    \filldraw[fill=red,draw=red] (1,6.75) circle (2pt);
    \filldraw[fill=red,draw=red] (-1,6.75) circle (2pt);
    \filldraw[fill=red,draw=red] (1,5.25) circle (2pt);
    \filldraw[fill=red,draw=red] (-1,-1.25) circle (2pt);
    \filldraw[fill=red,draw=red] (1,-2.5) circle (2pt);
    \filldraw[fill=red,draw=red] (1,5.25) circle (2pt);
    \node at (-3.5,2) {\small{$u$}};
    \node[color=blue] at (-2.375,2) {\small{$u'$}};
    \node at (3.5,2) {\small{$v$}};
    \node[color=blue] at (2.375,2) {\small{$v'$}};
    \node[color=red] at (-1.25,2) {\small{$t_i$}};
    \node[color=red] at (1.75,6.25) {\small{$s_{\phi(i)}$}};
    \node at (-.5,3.5) {$X_1$};
    \node at (-.5,2) {$U$};
    \node at (-.5,0.5) {$X_2$};
    \node at (0.5,3.5) {$Y_1$};
    \node at (0.5,2) {$V$};
    \node at (0.5,0.5) {$Y_2$};
    \end{tikzpicture}
    \quad 
    \begin{tikzpicture}[scale=.5]
    \draw (-1,7) -- (-2,6) -- (-1,5) -- (-1,7) -- (-4.5,3.5) -- (-3,2) -- (-1,4) -- (-1,0) -- (-4.5,3.5) -- (-6,2) -- (-1,-3) -- (-1,-1) -- (-2,-2);
    \draw[color=blue] (-2,2) -- (-1,3) -- (-1,1) -- (-2,2);
    \draw (1,7) -- (2,6) -- (1,5) -- (1,7) -- (4.5,3.5) -- (3,2) -- (1,4) -- (1,0) -- (4.5,3.5) -- (6,2) -- (1,-3) -- (1,-1) -- (2,-2);
    \draw[color=blue] (2,2) -- (1,3) -- (1,1) -- (2,2);
    \draw[dashed] (-1,6.5) -- (1,6.5);
    \draw[dashed] (-1,5.5) -- (1,5.5);
    \draw[dashed] (-1,6) -- (1,6);
    \draw[dashed] (-1,-2) -- (1,-2);
    \draw[dashed] (-1,-1.5) -- (1,-1.5);
    \draw[dashed] (-1,-2.5) -- (1,-2.5);
    \draw[color=red] (-1,2.25) -- (1,6.25);
    \draw[color=black!25!white,dashed] (-1,.5) -- (1,.5);
    \draw[color=black!25!white,dashed] (-1,2.5) -- (1,2.5);
    \draw[color=black!25!white,dashed] (-1,1.5) -- (1,1.5);
    \draw[color=black!25!white,dashed] (-1,3.5) -- (1,3.5);
    \draw[color=black!25!white,dashed] (-1,3) -- (1,3);
    \draw[color=black!25!white,dashed] (-1,2) -- (1,2);
    \draw[color=black!25!white,dashed] (-1,1) -- (1,1);
    \filldraw[fill=black,draw=black] (-6,2) circle (2pt);
    \filldraw[fill=black,draw=black] (-4.5,3.5) circle (2pt);
    \filldraw[fill=black,draw=black] (-2,6) circle (2pt);
    \filldraw[fill=black,draw=black] (-3,2) circle (2pt);
    \filldraw[fill=blue,draw=blue] (-2,2) circle (2pt);
    \filldraw[fill=black,draw=black] (-2,-2) circle (2pt);
    \filldraw[fill=black,draw=black] (6,2) circle (2pt);
    \filldraw[fill=black,draw=black] (4.5,3.5) circle (2pt);
    \filldraw[fill=black,draw=black] (2,6) circle (2pt);
    \filldraw[fill=black,draw=black] (3,2) circle (2pt);
    \filldraw[fill=blue,draw=blue] (2,2) circle (2pt);
    \filldraw[fill=black,draw=black] (2,-2) circle (2pt);
    \filldraw[fill=red,draw=red] (-1,2.25) circle (2pt);
    \filldraw[fill=red,draw=red] (1,6.25) circle (2pt);
    \draw[color=red] (-1,5.75) -- (1,6.75);
    \draw[color=red] (-1,6.75) -- (1,-2.5);
    \draw[color=red] (-1,-1.25) -- (1,5.25);
    \filldraw[fill=red,draw=red] (-1,5.75) circle (2pt);
    \filldraw[fill=red,draw=red] (1,6.75) circle (2pt);
    \filldraw[fill=red,draw=red] (-1,6.75) circle (2pt);
    \filldraw[fill=red,draw=red] (1,5.25) circle (2pt);
    \filldraw[fill=red,draw=red] (-1,-1.25) circle (2pt);
    \filldraw[fill=red,draw=red] (1,5.25) circle (2pt);
    \filldraw[fill=red,draw=red] (1,-2.5) circle (2pt);
    \node at (-3.5,2) {\small{$u$}};
    \node[color=blue] at (-2.375,2) {\small{$u'$}};
    \node at (3.5,2) {\small{$v$}};
    \node[color=blue] at (2.375,2) {\small{$v'$}};
    \node[color=red] at (-1.25,2) {\small{$t_i$}};
    \node[color=red] at (1.75,6.25) {\small{$s_{\phi(i)}$}};
    \node at (-.5,3.5) {$\overline{X_2}$};
    \node at (-.5,2) {$U$};
    \node at (-.5,0.5) {$\overline{X_1}$};
    \node at (0.5,3.5) {$\overline{Y_2}$};
    \node at (0.5,2) {$V$};
    \node at (0.5,0.5) {$\overline{Y_1}$};
    \end{tikzpicture}
    \caption{The effect of a paired flip at $(u,v)\in L(I)_T$ and $(u',v')$ with $(u,v)\succdot (u',v')$}
    \label{multifig1}
\end{figure}

Now consider $P(t_i),P(s_{\phi(i)})\in M(I)_0$. Our choices at $P(t_i),P(s_{\phi(i)})$ can only affect crossings between $(t_i,s_{\phi(i)})$ and edges in $(T_I,S_{\phi(I)},\phi|_I)$ with an endpoint in $\leaf(P(t_i))$ or $\leaf(P(s_{\phi(i)}))$. In the case that there is a leaf-matched pair $(u,v)\in L(I)$ with $P(t_i)>_T u$ and $v>_S P(s_{\phi(i)})$, we can use the same arguments from Lemma \ref{insertioncase1} with $P(t_i)$ in place of $u_0$ and $P(s_{\phi(i)})$ in place of $v_0$ to conclude that \algname{MultiInsertion} makes choices that extend to a solution. The case when there is a leaf-matched pair $(u,v)\in L(I)$ such that $P(s_{\phi(i)})>_S v$ and $u>_T P(t_i)$ is done similarly.

Finally, consider $P(t_i)\in M(I)_T$. Notice that the choice at $P(t_i)$ can only affect the crossings $\cross(P(t_i))=\{(e_i,e_j):t_j\in \leaf(P(t_i))\setminus \{t_i\}\}$, where all such $t_j$ must be leaves in the subtanglegram $(T_I,S_{\phi(I)},\phi|_I)$ since $\leaf(P(t_i))\cap I^c=\{t_i\}$. We claim that our future choices at elements in $\mathcal{P}$ do not affect these crossings. 
\begin{itemize}
    \item Any leaf-matched pairs $(u,v)\in L(I)$ with $(u,v)\succ P(t_i)$ have already been considered based on the linear extension, and there are no leaf-matched pairs $(u,v)\in L(I)$ with $P(t_i)\succ (u,v)$.
    \item If we perform subtree switches at $u\in M(I)\cap T\setminus \{P(t_i)\}$, then the resulting layouts restrict to the same layout of $(T_{I\cup \{i\}},S_{\phi(I\cup \{i\})},\phi|_{I\cup \{i\}})$. Hence, the same crossings occur in $\cross(P(t_i))$.
    \item If we perform subtree switches at $v\in M(I)\cap S$ with $v\neq A(s_{\phi(i)})$, then the resulting layouts restrict to the same layout of $(T_{I\cup \{i\}},S_{\phi(I\cup \{i\})},\phi|_{I\cup \{i\}})$. Again, the same crossings occur in $\cross(P(t_i))$.
    \item Now for $A(s_{\phi(i)})$, we consider two cases. In the case that $\leaf(A(s_{\phi(i)}))\cap \phi(I)$ and $\leaf(P(t_i))\cap I$ do not contain matched leaves, an operation at $A(s_{\phi(i)})$ cannot affect the crossings in $\cross(P(t_i))$. In the case that there exists $(u,v)\in L(I)\setminus L(I)_0$ with $A(s_{\phi(i)})\succ (u,v) \succ P(t_i)$, an operation at $A(s_{\phi(i)})$ has already been considered.
\end{itemize}
Thus, a choice at $P(t_i)$ that extends to a solution must be one that minimizes crossings in $\cross(P(t_i))$, which is precisely what the algorithm does. The case $P(s_{\phi(i)})\in M(I)_S$ follows by a similar argument.
\end{proof}

\begin{theorem}
\algname{MultiInsertion} solves the Tanglegram Multiple Edge Insertion Problem in $O(2^{|L(I)_1\cup M(I)_1|} n^2)$ time and $O(n^2)$ space.
\end{theorem}
\begin{proof}
By Lemma \ref{multioperations}, \algname{MultiInsertion} only encounters layouts that restrict to planar layouts of $(T_I,S_{\phi(I)},\phi|_I)$. By Lemma \ref{forloop}, \algname{MultiInsertion} encounters a solution $(X_{min},Y_{min})$ to the Multiple Edge Insertion Problem in some iteration of the \textbf{for} loop in line \ref{multiforloop}. From lines \ref{remainingsteps}-\ref{trackmin}, we see that the algorithm stores the layout with the fewest number of crossings considered over all iterations of the \textbf{for} loop, so it will return either $(X_{min},Y_{min})$ or another layout with the same number of crossings. Thus, the output of \algname{MultiInsertion} is a solution to the Multiple Edge Insertion Problem.

For the run-time and space claims, first note that lines \ref{multistart1}-\ref{multistart2} can all be completed in $O(n^2)$ time. 
 The \textbf{for} loop in line \ref{multiforloop} then runs for $2^{|L(I)_1\cup M(I)_1|}$ iterations. The \textbf{for} loop in line \ref{multiforloop2} has at most $2n$ steps, and each step takes $O(n)$ time, as all of the \textbf{if} and \textbf{else if} conditions can be checked in $O(n)$ time, and paired flips and subtree switches take $O(n)$ time. The remaining steps after line \ref{remainingsteps} take $O(n^2)$ time for a total of  $O(2^{|L(I)_1\cup M(I)_1|} n^2)$ time. For the space claim, notice that storing all of the sets and layouts takes $O(n^2)$ space.
\end{proof}

\begin{remark}
In the special case that $|I|=n-1$, the set $L(I)_1$ is empty and $|M(I)_1|\leq 2$. Here, \algname{MultiInsertion} reduces to a less efficient version of \algname{Insertion} that still runs in $O(n^2)$ time.
\end{remark}

\section{Future work}\label{future}

In Section \ref{planar}, we defined the flip graph of a planar tanglegram. While paired flips will generate all vertices in this graph, it is possible that some flips do not produce a new layout, as tanglegrams are considered up to isomorphism on $T$ and $S$. One simple example of this is the unique tanglegram of size $2$, where a paired flip at the roots of both trees does not produce a new layout. As such, we pose the following problem.

\begin{problem*}
For any planar tanglegram $(T,S,\phi)$, characterize the flip graph $\Gamma(T,S,\phi)$. In particular, determine the number of vertices, the number of edges, and the degree of any vertex.
\end{problem*}

Billey, Konvalinka, and Matsen previous gave an algorithm for generating tanglegrams uniformly at random \cite{count}. We propose a corresponding problem for the case of planar tanglegrams. Since there is a bijection between irreducible planar tanglegrams and pairs of triangulations with no common diagonal from \cite{countingplanar}, solving this problem may lead to solutions to other open problems.

\begin{problem*}
Construct an efficient algorithm {generating} planar tanglegrams uniformly at random.
\end{problem*}

Finally, we pose a problem about using \algname{MultiInsertion} to approximate the tanglegram crossing number, where the bound is modeled after one in \cite{graphcrossing}. For any tanglegram $(T,S,\phi)$, this also requires finding a planar subtanglegram $(T_I,S_{\phi(I)},\phi|_I)$, and from Corollary \ref{ex1}, we know that we do not necessarily want a subtanglegram of maximum size. Additionally, we insist on an efficient algorithm, so we must modify \algname{MultiInsertion} at the vertices in $L(I)_1\cup M(I)_1$.

\begin{problem*}
Use \algname{MultiInsertion} to construct an efficient algorithm that finds a tanglegram layout with at most $O(\crt(T,S,\phi)\cdot \text{poly}(\log n))$ crossings, where $n$ is the size of $(T,S,\phi)$ and $\text{poly}(x)$ is some polynomial in $x$.
\end{problem*}

\section*{Acknowledgements}
\noindent
We would like to thank Sara Billey for proposing tanglegrams as an area of research, suggesting improvements to several proofs, and giving valuable feedback in earlier versions of this paper. We would also like to thank Herman Chau,  Paige Helms, Matja\v{z} Konvalinka, Stark Ryan, and Stephan Wagner for valuable feedback in earlier versions of this paper. The author was partially supported by the National Science Foundation grant DMS-1764012.

\printbibliography


\section{Appendix}

\begin{algorithm}[H]\DontPrintSemicolon
\caption{\algname{Insertion Case $v_0>_S v_{Tmax}$}}\label{case2algorithm}
\KwIn{tanglegram $(T,S,\phi)$, index $i$ such that $(T_I,S_{\phi(I)},\phi|_I)$ is planar for $I=[n]\setminus \{i\}$}
\KwOut{layout of $(T,S,\phi)$ that restricts to a planar layout of $(T_I,S_{\phi(I)},\phi|_I)$}
\tcp*{Step 1: initialize the algorithm.}
$(X,Y),L\coloneqq\algname{ModifiedUntangle}(T,S,\phi|_I)$\\
construct $L(I)$ from $L$ using Definition \ref{Lsubtanglegram}\\
$u_0\coloneqq$ parent of $t_i$, $v_0\coloneqq$ parent of $s_{\phi(i)}$\\
$L(I)_T\coloneqq\{(u,v)\in L(I):u>_T t_i,v\not>_S s_{\phi(i)}\}$\\
\tcp*{Step 2: construct edge sets.}
linearly order $L(I)_T=\{(u_j,v_j)\}_{j=1}^m$ so that $u_1<_T u_2<_T \ldots <_T u_m$\\
$E(v_0) \coloneqq$ between-tree edges with an endpoint in $\leaf(v_0)\setminus \leaf(v_m)$\\
$E(u_0) \coloneqq$ between-tree edges with an endpoint in $\leaf(u_0)\setminus \{t_i\}$\\
\For{$j=1,2,\ldots ,m$,}
    {$E(u_j) \coloneqq$ between-tree edges with an endpoint in $\leaf(u_j)\setminus \leaf(u_{j-1})$}
\tcp*{Step 3: use paired flips and subtree switches to change crossings.}
\If{$(t_i,s_{\phi(i)})$ crosses more than half of the edges in $E(v_0)$ in the layout $(X,Y)$,} 
    {update $Y\coloneqq \algname{SubtreeSwitch}(Y,v_0)$} 
\For{$j=m,\ldots ,2,1$,} 
    {
    \If{$(t_i,s_{\phi(i)})$ crosses more than half of the edges in $E(u_j)$ in the layout $(X,Y)$,}  
        {update $(X,Y)\coloneqq \algname{PairedFlip}((X,Y),(u_j,v_j))$}
    }
\If{$(t_i,s_{\phi(i)})$ crosses more than half of the edges in $E(u_0)$ in the layout $(X,Y)$,} 
    {update $X\coloneqq \algname{SubtreeSwitch}(X,u_0)$} 
\Return $(X,Y)$
\end{algorithm}

\noindent
   \textsc{Kevin Liu, Department of Mathematics, University of Washington, Seattle, WA 98125, USA}\\
   \textit{E-mail address}: \texttt{kliu15@uw.edu}

\end{document}